\documentclass[10pt,reqno]{amsart}

\usepackage{amsmath,amsfonts,amsthm,amssymb,amsxtra}
\usepackage{amsxtra, amssymb, mathrsfs, pstricks}
\usepackage[latin2]{inputenc}
\usepackage[english]{babel}
\usepackage{amsmath,amsfonts,amstext,amsbsy,amsopn,amssymb,amsthm,amscd}
\usepackage{amsxtra,latexsym}
\usepackage{exscale}
\usepackage{graphicx,mathrsfs}
\usepackage{psfrag}
\usepackage{paralist,eucal,enumerate}
\usepackage{color}

\definecolor{dred}{rgb}{.8,0,0}
\definecolor{ddmagenta}{rgb}{0.7,0,0.9}
\definecolor{ddcyan}{rgb}{0,0.2,1.0}
\definecolor{dblue}{rgb}{0,0,0.7}

\numberwithin{equation}{section}

\def\trait #1 #2 #3 {\vrule width #1pt height #2pt depth #3pt}
\def\fin{
    \trait .3 5 0
    \trait 5 .3 0
    \kern-5pt
    \trait 5 5 -4.7
    \trait 0.3 5 0
\medskip}

\newcommand{\QED}{\mbox{}\hfill\rule{5pt}{5pt}\medskip\par}

\makeatletter
\let\@fnsymbol\@arabic
\makeatother

\definecolor{ddmagenta}{rgb}{0.7,0,1.0}
\definecolor{ddcyan}{rgb}{0,0.1,1.0}
\definecolor{dred}{rgb}{.8,0,0}
\definecolor{ddgreen}{rgb}{0,0.4,0.4}

\newcommand{\Gammac}{{\ele {\Gmc}}}
\newcommand{\Gdir}{\Gamma_{\mbox{\tiny\rm D}}}
\newcommand{\Gnew}{\Gamma_{\mbox{\tiny\rm N}}}
\newcommand{\Gmc}{\Gamma_{\mbox{\tiny\rm C}}}
\newcommand{\BV}{\mathrm{BV}}

\newcommand{\bele}{\begin{lemm}}
\newcommand{\enle}{\end{lemm}}
\newcommand{\bedef}{\begin{defi}}
\newcommand{\bete}{\begin{teor}}
\newcommand{\eddef}{\end{defi}}
\newcommand{\ente}{\end{teor}}
\newcommand{\beos}{\begin{remark}}
\newcommand{\eddos}{\end{remark}}
\newcommand{\bepr}{\begin{prop}}
\newcommand{\empr}{\end{prop}}
\newcommand{\bepro}{\begin{prob}}
\newcommand{\empro}{\end{prob}}
\newcommand{\bede}{\begin{defin}}
\newcommand{\edde}{\end{defin}}
\newcommand{\beco}{\begin{coro}}
\newcommand{\enco}{\end{coro}}

\newcommand{\beeq}[1]{\begin{equation}
 \label{#1}}
\newcommand{\eddeq}{\end{equation}}

\newcommand{\beeqa}[1]{\begin{eqnarray}
  \label{#1}}
\newcommand{\eddeqa}{\end{eqnarray}}

\newcommand{\beal}[1]{\begin{align}
 \label{#1}}
\newcommand{\eddal}{\end{align}}

\newcommand{\bespl}[1]{\begin{split}
 \label{#1}}
\newcommand{\edspl}{\end{split}}

\newcommand{\bega}[1]{\begin{gather}
 \label{#1}}
\newcommand{\edga}{\end{gather}}

\newcommand{\beeqax}{\begin{eqnarray*}}
\newcommand{\eddeqax}{\end{eqnarray*}}

\newcommand{\no}{\nonumber}

\newcommand{\tensore}{\varepsilon({\bf u})}
\newcommand{\tensoret}{\varepsilon(\partial_t \uu)}

\numberwithin{equation}{section}

\newcommand{\teta}{\vartheta}

\newcommand{\dt}{\partial_t}

\newcommand{\uu}{\mathbf{u}}

\newcommand{\ww}{\mathbf{w}}
\newcommand{\vv}{\mathbf{v}}

\newcommand{\eeta}{{\mbox{\boldmath$\eta$}}}

\newcommand{\mmu}{{\mbox{\boldmath$\mu$}}}
\newcommand{\WW}{{\bf W}}

\newcommand{\epsi}{\varepsilon}

\newcommand{\weak}{\rightharpoonup}
\newcommand{\weakstar}{\mathop{\rightharpoonup}^{*}}

 \DeclareMathOperator{\dive}{div}

\DeclareMathOperator{\Id}{Id}

\let\TeXchi\chi
\def\chi{{\setbox0 \hbox{\mathsurround0pt
$\TeXchi$}\hbox{\raise\dp0 \copy0 }}}

\newtheorem{theorem}{Theorem}[section]

\newtheorem{lemma}{Lemma}[section]
\newtheorem{proposition}[lemma]{Proposition}
\newtheorem{definition}[lemma]{Definition}%
\newtheorem{remark}[lemma]{Remark}%
\newtheorem{problem}[lemma]{Problem}
\newtheorem{notation}[lemma]{Notation}%
\newtheorem{hypothesis}[theorem]{Hypothesis}
\newtheorem{example}[lemma]{Example}

\begin{document}

\newcommand{\Qt}{\Omega \times (0,t)}
\newcommand{\Qtau}{\Omega \times (0,\tau)}
\newcommand{\til}{\widetilde}
\renewcommand{\part}{\partial_t}
\renewcommand{\teta}{\vartheta}
\newcommand{\accaunoz}{H^1_0(\Omega)}
\newcommand{\accadue}{H^2(\Omega)}
\newcommand{\nnu}{\nonumber}
\newcommand{\irre}{\rho}
\newcommand{\vinc}{\beta}
\newcommand{\weaksto}{{\rightharpoonup^*}}
\newcommand{\weakto}{\rightharpoonup}
\newcommand{\debole}{\,\weak\,}
\newcommand{\debolestar}{\,\weakstar\,}
\newcommand{\pairing}[4]{ \sideset{_{#1 }}{_{ #2}}  {\mathop{\left\langle #3 , #4  \right\rangle}}}
\newcommand{\comments}[1]{\marginpar{\tiny\textit{#1}}}
\newcommand{\nns}{{\ele \nn_s}}
\newcommand{\tetatm}{\teta_{\tau_m}}
\newcommand{\tetatem}{\teta_{\varepsilon_m}^{\tau_m}}
\newcommand{\tetatemj}{\teta_{\varepsilon_{m_j^1}}^{\tau_{m_j^1}}}
\newcommand{\tetastm}{{\teta}_{s,\tau_m}}
\newcommand{\tetastem}{{\teta}_{s,\varepsilon_m}^{\tau_m}}
\newcommand{\tetastemj}{{\teta}_{s,\varepsilon_{m_j^1}}^{\tau_{m_j^1}}}
\newcommand{\uutm}{\uu_{\tau_m}}
\newcommand{\widetildeuutm}{\widetilde{\uu}_{\tau_m}}
\newcommand{\uutem}{\uu_{\varepsilon_m}^{\tau_m}}
\newcommand{\chitm}{\chi_{\tau_m}}
\newcommand{\widetildechitm}{\widetilde{\chi}_{\tau_m}}
\newcommand{\chitem}{\chi_{\varepsilon_m}^{\tau_m}}
\newcommand{\tetat}{\teta_{\tau}}
\newcommand{\tetast}{{\teta}_{s,\tau}}
\newcommand{\widetildetetas}{\widetilde{\teta}_{s}}
\newcommand{\widetildetetast}{\widetilde{\teta}_{s,\tau}}
\newcommand{\uutt}{\uu_{\tau}}
\newcommand{\widetildechitt}{\widetilde{\chi}_{\tau}}
\newcommand{\widetildeuutt}{\widetilde{\uu}_{\tau}}
\newcommand{\chitt}{\chi_{\tau}}
\newcommand{\ele}{\color{black}}

 \def\fin{\hfill
         \trait .3 5 0
         \trait 5 .3 0
         \kern-5pt
         \trait 5 5 -4.7
         \trait 0.3 5 0
 \medskip}
 \def\trait #1 #2 #3 {\vrule width #1pt height #2pt depth #3pt}
\newcommand{\forae}{\text{for a.a.}}
\newcommand{\aein}{\text{a.e.\ in}}

\newcommand{\down}{\downarrow}
\newcommand{\up}{\to}


\def\bbA{{\mathbb A}} \def\bbB{{\mathbb B}} \def\bbC{{\mathbb C}}
\def\bbD{{\mathbb D}} \def\bbE{{\mathbb E}} \def\bbF{{\mathbb F}}
\def\bbG{{\mathbb G}} \def\bbH{{\mathbb H}} \def\bbI{{\mathbb I}}
\def\bbJ{{\mathbb J}} \def\bbK{{\mathbb K}} \def\bbL{{\mathbb L}}
\def\bbM{{\mathbb M}} \def\bbN{{\mathbb N}} \def\bbO{{\mathbb O}}
\def\bbP{{\mathbb P}} \def\bbQ{{\mathbb Q}} \def\bbR{{\mathbb R}}
\def\bbS{{\mathbb S}} \def\bbT{{\mathbb T}} \def\bbU{{\mathbb U}}
\def\bbV{{\mathbb V}} \def\bbW{{\mathbb W}} \def\bbX{{\mathbb X}}
\def\bbY{{\mathbb Y}} \def\bbZ{{\mathbb Z}}

\newcommand{\R}{\bbR}
\newcommand{\N}{\bbN}
\def\calA{{\mathcal A}} \def\calB{{\mathcal B}} \def\calC{{\mathcal C}}
\def\calD{{\mathcal D}} \def\calE{{\mathcal E}} \def\calF{{\mathcal F}}
\def\calG{{\mathcal G}} \def\calH{{\mathcal H}} \def\calI{{\mathcal I}}
\def\calJ{{\mathcal J}} \def\calK{{\mathcal K}} \def\calL{{\mathcal L}}
\def\calM{{\mathcal M}} \def\calN{{\mathcal N}} \def\calO{{\mathcal O}}
\def\calP{{\mathcal P}} \def\calQ{{\mathcal Q}} \def\calR{{\mathcal R}}
\def\calS{{\mathcal S}} \def\calT{{\mathcal T}} \def\calU{{\mathcal U}}
\def\calV{{\mathcal V}} \def\calW{{\mathcal W}} \def\calX{{\mathcal X}}
\def\calY{{\mathcal Y}} \def\calZ{{\mathcal Z}}
\def\rma{{\mathrm a}} \def\rmb{{\mathrm b}} \def\rmc{{\mathrm c}}
\def\rmd{{\mathrm d}} \def\rme{{\mathrm e}} \def\rmf{{\mathrm f}}
\def\rmg{{\mathrm g}} \def\rmh{{\mathrm h}} \def\rmi{{\mathrm i}}
\def\rmj{{\mathrm j}} \def\rmk{{\mathrm k}} \def\rml{{\mathrm l}}
\def\rmm{{\mathrm m}} \def\rmn{{\mathrm n}} \def\rmo{{\mathrm o}}
\def\rmp{{\mathrm p}} \def\rmq{{\mathrm q}} \def\rmr{{\mathrm r}}
\def\rms{{\mathrm s}} \def\rmt{{\mathrm t}} \def\rmu{{\mathrm u}}
\def\rmv{{\mathrm v}} \def\rmw{{\mathrm w}} \def\rmx{{\mathrm x}}
\def\rmy{{\mathrm y}} \def\rmz{{\mathrm z}}
\def\FG{\mathbf}
\def\mdot{\FG{:}}  \def\vdot{\FG{\cdot}}
\def\bfa{{\FG a}} \def\bfb{{\FG b}} \def\bfc{{\FG c}}
\def\bfd{{\FG d}} \def\bfe{{\FG e}} \def\bff{{\FG f}}
\def\bfg{{\FG g}} \def\bfh{{\FG h}} \def\bfi{{\FG i}}
\def\bfj{{\FG j}} \def\bfk{{\FG k}} \def\bfl{{\FG l}}
\def\bfm{{\FG m}} \def\bfn{{\FG n}} \def\bfo{{\FG o}}
\def\bfp{{\FG p}} \def\bfq{{\FG q}} \def\bfr{{\FG r}}
\def\bfs{{\FG s}} \def\bft{{\FG t}} \def\bfu{{\FG u}}
\def\bfv{{\FG v}} \def\bfw{{\FG w}} \def\bfx{{\FG x}}
\def\bfy{{\FG y}} \def\bfz{{\FG z}}
\def\bfA{{\FG A}} \def\bfB{{\FG B}} \def\bfC{{\FG C}}
\def\bfD{{\FG D}} \def\bfE{{\FG E}} \def\bfF{{\FG F}}
\def\bfG{{\FG G}} \def\bfH{{\FG H}} \def\bfI{{\FG I}}
\def\bfJ{{\FG J}} \def\bfK{{\FG K}} \def\bfL{{\FG L}}
\def\bfM{{\FG M}} \def\bfN{{\FG N}} \def\bfO{{\FG O}}
\def\bfP{{\FG P}} \def\bfQ{{\FG Q}} \def\bfR{{\FG R}}
\def\bfS{{\FG S}} \def\bfT{{\FG T}} \def\bfU{{\FG U}}
\def\bfV{{\FG V}} \def\bfW{{\FG W}} \def\bfX{{\FG X}}
\def\bfY{{\FG Y}} \def\bfZ{{\FG Z}}
\def\BS{\boldsymbol} 
\def\bfalpha{{\BS\alpha}}      \def\bfbeta{{\BS\beta}}
\def\bfgamma{{\BS\gamma}}      \def\bfdelta{{\BS\delta}}
\def\bfepsilon{{\BS\varepsilon}} \def\bfzeta{{\BS\zeta}}
\def\bfeta{{\BS\eta}}          \def\bftheta{{\BS\theta}}
\def\bfiota{{\BS\iota}}        \def\bfkappa{{\BS\kappa}}
\def\bflambda{{\BS\lambda}}    \def\bfmu{{\BS\mu}}
\def\bfnu{{\BS\nu}}            \def\bfxi{{\BS\xi}}
\def\bfpi{{\BS\pi}}            \def\bfrho{{\BS\rho}}
\def\bfsigma{{\BS\sigma}}      \def\bftau{{\BS\tau}}
\def\bfphi{{\BS\phi}}          \def\bfchi{{\BS\chi}}
\def\bfpsi{{\BS\psi}}          \def\bfomega{{\BS\omega}}
\def\bfvarphi{{\BS\varphi}}
\def\bfGamma{{\BS\Gamma}}      \def\bfDelta{{\BS\Delta}}
\def\bfTheta{{\BS\Theta}}      \def\bfLambda{{\BS\Lambda}}
\def\bfXi{{\BS\Xi}}            \def\bfPi{{\BS\Pi}}
\def\bfSigma{{\BS\Sigma}}      \def\bfPhi{{\BS\Phi}}
\def\bfPsi{{\BS\Psi}}          \def\bfOmega{{\BS\Omega}}

\newcommand{\dd}{\, \mathrm{d}}
\newcommand{\eps}{\varepsilon}

\newcommand{\rmC}{\mathrm{C}}
\newcommand{\rmD}{\mathrm{D}}

\newcommand{\piecewiseConstant}[2]{\overline{#1}_{\kern-1pt#2}^\eps}
\newcommand{\pwC}{\piecewiseConstant}
\newcommand{\underlinepiecewiseConstant}[2]{\underline{#1}_{\kern-1pt#2}^\eps}
\newcommand{\upwC}{\underlinepiecewiseConstant}

\newcommand{\piecewiseLinear}[2]{{#1}_{\kern-1pt#2}^\eps}
\newcommand{\pwL}{\piecewiseLinear}
\newcommand{\pwM}[2]{\widetilde{#1}_{\kern-1pt#2}}
 \def\trait #1 #2 #3 {\vrule width #1pt height #2pt depth #3pt}
\newcommand{\pwN}[2]{#1_{\kern-1pt#2}}
 \def\trait #1 #2 #3 {\vrule width #1pt height #2pt depth #3pt}

\newcommand{\uk}{\pwN {\uu^k}{\tau}}
\newcommand{\un}{\pwN {\uu^n}{\tau}}
\newcommand{\Fn}{\pwN {J^n}{\tau}}
\newcommand{\uku}{\pwN {\uu^{k-1}}{\tau}}
\newcommand{\unu}{\pwN {\uu^{n-1}}{\tau}}
\newcommand{\chin}{\pwN {\chi^n}{\tau}}
\newcommand{\Fun}{\pwN {\mathbf{F}^n}{\tau}}
\newcommand{\gun}{\pwN {\mathbf{g}^n}{\tau}}
\newcommand{\chinuno}{\pwN {\chi^{n+1}}{\tau}}
\newcommand{\dom}{\text{dom}}

\newcommand{\uuh}{\widehat{\uu}}
\newcommand{\chih}{\widehat{\chi}}
\newcommand{\eetah}{\widehat{\eeta}}
\newcommand{\xih}{\widehat{\xi}}
\newcommand{\zetah}{\widehat{\zeta}}

\def\bsA{{\FG A}} \def\bsB{{\FG B}} \def\bsC{{\FG C}}
\def\bsD{{\FG D}} \def\bsE{{\FG E}} \def\bsF{{\FG F}}
\def\bsG{{\FG G}} \def\bsH{{\FG H}} \def\bsI{{\FG I}}
\def\bsJ{{\FG J}} \def\bsK{{\FG K}} \def\bsL{{\FG L}}
\def\bsM{{\FG M}} \def\bsN{{\FG N}} \def\bsO{{\FG O}}
\def\bsP{{\FG P}} \def\bsQ{{\FG Q}} \def\bsR{{\FG R}}
\def\bsS{{\FG S}} \def\bsT{{\FG T}} \def\bsU{{\FG U}}
\def\bsV{{\FG V}} \def\bsW{{\FG W}} \def\bsX{{\FG X}}
\def\bsY{{\FG Y}} \def\bsZ{{\FG Z}}


\newcommand{\uun}{u_{\mathrm{N}}}
\newcommand{\whuun}{\widehat{u}_{\mathrm{N}}}
\newcommand{\uuni}[1]{(u_{#1})_{\mathrm{N}}}
\newcommand{\uut}{\uu_{\mathrm{T}}}
\newcommand{\wwt}{\ww_{\mathrm{T}}}
\newcommand{\dotu}{\partial_t{\uu}}
\newcommand{\dotui}[1]{\partial_t{\uu}_{#1}}
\newcommand{\dotut}{\partial_t{\uu}_{\mathrm{T}}}
\newcommand{\vvn}{v_{\mathrm{N}}}
\newcommand{\vvt}{\vv_{\mathrm{T}}}
\newcommand{\nn}{\mathbf{n}}
\newcommand{\zz}{\mathbf{z}}
\newcommand{\dotv}{\partial_t{\vv}}
\newcommand{\dotvt}{\partial_t{\vv}_{\mathrm{T}}}
\newcommand{\ssigma}{{\mbox{\boldmath$\sigma$}}}
\newcommand{\sigman}{\sigma_{\mathrm{N}}}
\newcommand{\sigmat}{\ssigma_{\mathrm{T}}}
\newcommand{\gc}{\Gammac}
\newcommand{\cn}{c_{\mathrm{N}}}
\newcommand{\ct}{c_{\mathrm{T}}}
\newcommand{\Cweak}{\mathrm{C}^0_{\mathrm{weak}}}
\newcommand{\Reg}{\mathcal{R}}
\newcommand{\foraa}{\text{for a.a.}}
\newcommand{\funeta}[1]{\mathcal{J}_{#1}}
\newcommand{\plfuneta}[1]{\partial\mathcal{J}_{#1}}

\newcommand{\uue}{\uu_{\eps}}
\newcommand{\ud}[1]{\uu_{#1}^{\eps}}
\newcommand{\chid}[1]{\chi_{#1}^{\eps}}
\newcommand{\etad}[1]{\eeta_{#1}^{\eps}}
\newcommand{\zud}[1]{\zz_{#1}^{\eps}}
\newcommand{\mud}[1]{\mmu_{#1}^{\eps}}
\newcommand{\xid}[1]{\xi_{#1}^{\eps}}
\newcommand{\udno}[1]{(\uu_{#1}^{\eps})_{\mathrm{N}}}
\newcommand{\uuen}{(u_{\eps})_{\mathrm{N}}}
\newcommand{\uune}{(u_{\eps})_{\mathrm{N}}}
\newcommand{\uute}{(\uue)_{\mathrm{T}}}
\newcommand{\dotue}{\partial_t{\uu}_{\eps}}
\newcommand{\dotute}{(\partial_t{\uue})_{\mathrm{T}}}
\newcommand{\chie}{\chi_{\eps}}
\newcommand{\zetae}{\zeta_\eps}
\newcommand{\zze}{\mathbf{z}_\eps}
\newcommand{\He}{\mathcal{H}}
\newcommand{\matrid}{\mathbf{1}}
\newcommand{\fc}{\mathfrak{c}}
\newcommand{\tetas}{\teta_s}
\newcommand{\tetase}{\teta_{s,\eps}}



\newcommand{\Hc}{H_{\Gammac}}
\newcommand{\V}{V}
\newcommand{\Vc}{V_{\Gammac}}
\newcommand{\Yc}{Y_{\Gammac}}
\newcommand{\bvarphi}{{\mbox{\boldmath$\varphi$}}}
\newcommand{\hunoc}{H^1 (\Gammac)}
\newcommand{\FF}{\mathbf{F}}
\newcommand{\sig}{\sigma}
\newcommand{\HH}{\mathrm{H}}
\newcommand{\applog}{{\cal L}_\varepsilon}
\newcommand{\applogteta}{\widetilde{L}_\varepsilon}
\newcommand{\applogtetas}{\widetilde{\ell}_\varepsilon}
\newcommand{\hc}{\Hc}
\newcommand{\lteta}{L}
\newcommand{\ltetas}{\ell}
\newcommand{\resteta}{R_\eps}
\newcommand{\restetas}{r_\eps}
\newcommand{\fteta}{g}
\newcommand{\ftetas}{f}
\newcommand{\pseupot}{\widehat{\rho}}

\newenvironment{rcomm}{\color{black}}{\color{black}}
\newenvironment{newr}{\color{black}}{\color{black}}
\newcommand{\ber}{\begin{rcomm}}
\newcommand{\edr}{\end{rcomm}}
\newcommand{\berc}{\begin{newr}}
\newcommand{\eerc}{\end{newr}}
\newcommand{\beroc}{\begin{rcomm}}
\newcommand{\eroc}{\end{rcomm}}
\newenvironment{controlla}{\color{black}}{\color{black}}
\newcommand{\bec}{\begin{controlla}}
\newcommand{\eec}{\end{controlla}}
\newenvironment{new}{\color{black}}{\color{black}}
\newcommand{\bee}{\begin{new}}
\newcommand{\ede}{\end{new}}



                                %
\title{Analysis of a model coupling  volume and surface processes in thermoviscoelasticity}

\author{Elena Bonetti}
\address{Elena Bonetti\\ Dipartimento di Matematica ``F.\
Casorati'' \\ Universit\`a di Pavia\\ Via Ferrata 1 \\ I-27100 Pavia\\
Italy} \email{elena.bonetti\,@\,unipv.it}

\author{Giovanna Bonfanti}
\address{Giovanna Bonfanti\\  DICATAM - Sezione di Matematica \\ Universit\`{a} di Brescia\\ Via Valotti 9\\ I-25133 Brescia\\ Italy}
\email{bonfanti@ing.unibs.it}

\author{Riccarda Rossi}
\address{Riccarda Rossi\\  DICATAM - Sezione di Matematica \\ Universit\`{a} di Brescia\\ Via Valotti 9\\ I-25133 Brescia\\ Italy}
\email{riccarda.rossi@ing.unibs.it}

\date {January 31st, 2014}
 \maketitle

\begin{abstract}
We focus
on a highly nonlinear evolutionary abstract PDE system describing volume processes coupled with
surfaces processes in thermoviscoelasticity, featuring the
quasi-static momentum balance, the equation for the unidirectional
evolution of an internal variable on the surface, and the equations
for the temperature in the bulk \ber domain \edr and the temperature on the
surface. A significant example of \bec our  system \eec occurs in the
 modeling  for  the
unidirectional evolution of the adhesion between a body and a rigid
support, subject to thermal fluctuations and   in contact with
friction.

We investigate the related initial-boundary value problem, and in particular the
issue of  existence of \emph{global-in-time} solutions, on an abstract level. This allows us
to highlight the analytical features of the problem and, at the same time, to exploit the
tight coupling between the various equations in order to deduce suitable estimates on (an approximation) of the problem.

Our existence result is proved by passing to the limit in a carefully tailored approximate problem,
and by extending the obtained local-in-time solution by means of a refined prolongation argument.
\end{abstract}
\section{\bf Introduction}
\noindent
This paper tackles the analysis of a PDE system describing a class of models where volume and surface processes are coupled. Our  main example, and the motivation for our study,
stems from a specific PDE system modeling   
 adhesive contact, with frictional effects, in thermoviscoelasticity.

\paragraph{\bf A  PDE system for  contact with adhesion, friction, and thermal effects.}
    Contact and delamination  arise in many \bec fields \eec  in solid mechanics:  among others,
     we may mention here the  {\ele application to (structural) adhesive materials  in civil engineering,}
      the investigation  of earthquakes, and   the study  of layered composite structures within machine design and manufacturing. Indeed,  the degradation of the adhesive substance between the various laminates leads to material failure.  That is why, there is a rich literature on this kind of problems, both from the engineering and from the mathematical community: we refer to the monographs \cite{eck05} and \cite{sofonea-han-shillor}, and to the references in \cite{bbr1}--\cite{bbr6},  for some survey.

     In this paper, following up on the recent \cite{bbr6}, we  focus on 
      a PDE system for frictional adhesion 
 between a viscoelastic body, subject to  thermal fluctuations, in contact with a   \emph{rigid support}. In \cite{bbr6} this system was  derived, according to the laws of Thermomechanics,
on the basis of  the modeling approach proposed by M.\ Fr\'emond (cf.\ \cite{fremond-nedjar,fre,fre2012}).  Such approach has already  been applied in the previous \cite{bbr1, bbr2, bbr3, bbr4} to the investigation of adhesive contact, both for an isothermal and for a temperature-dependent system, as well as in \cite{bbr5}, for isothermal adhesive contact with
frictional effects.    The  model we consider here, encompassing friction,
adhesion, and thermal effects, was  analyzed in \cite{bbr6} in the case where no \emph{irreversibility}
of  the degradation of the adhesive  substance  is enforced. 
In the present contribution, we broaden our investigation 
 by
  encompassing in the model this \emph{unidirectionality} feature. 
   As we are going to demonstrate in what follows, this extension of the model from \cite{bbr6} brings about substantial analytical difficulties.

 Let us now get a closer look at the PDE system under investigation.
  In accord with Fr\'emond's modeling ansatz, adhesion is described in terms of an internal variable $\chi$
which can be interpreted as a  \emph{surface damage paramater}:  it
accounts for the state of the bonds between the body and the support, on which some adhesive substance is present. The other state variables are {\ele the strain tensor
(in a small-strain framework), related to the
displacement vector $\uu$},  and the possibly different (absolute) temperatures
 $\teta$ and $\teta_s$
 in the body  and on the contact surface.
 \bec We allow for \eec
 {\ele  two different temperatures in the bulk domain and on the contact surface \bec because \eec we are modeling a physical situation in which some adhesive substance may be present on the contact surface, with different thermomechanical properties \bec in comparison \eec to the material \bec in \eec the domain.}
The evolution of $(\teta,\teta_s,\uu,\chi)$ is described by a  rather complex PDE system, consisting of  the quasi-static momentum balance (where the  inertial contributions are neglected), {\ele a} parabolic-type evolution equations for the temperature $\teta$
and $\teta_s$, and a doubly nonlinear differential inclusion for $\chi$.
Denoting by $\Omega
\subset \R^3$ the (sufficiently regular) domain occupied by the body and  by $\Gmc$ the part of the boundary $\partial\Omega$
on which the body \ber may be \edr  in contact with the support,  the \ber system reads \edr as follows
\begin{subequations}
\label{system-concrete}
\begin{align}
\label{e1-concrete} &\partial_t \ln(\vartheta) - \dive
(\partial_t\mathbf{u}) -\bec \mathrm{div}(\nabla g(\teta)) \eec = h \qquad \text{in $\Omega
\times (0,T)$,}
\\
\label{condteta-concrete-1} & \bec \nabla(g(\teta)) \cdot \nn \eec  =0  \qquad \text{in $(\partial \Omega \setminus \Gmc)
\times (0,T)$,}
\\
\label{condteta-concrete-2} & \partial_{\nn } \teta  \in - k(\chi)(\vartheta-\vartheta_s)-\fc'(\teta-\teta_s) {\partial
I_{(-\infty,0]}( \uun )} | \dotut|  \qquad  \text{in $ \Gmc \times
(0,T)$,}
\\
\label{eqtetas-concrete} &\partial_t \ln(\vartheta_s) -
\partial_t \lambda(\chi)
 -\mathrm{div}(\nabla f(\teta_s))  \in  k(\chi)
(\vartheta-\vartheta_s)+ \fc'(\teta-\teta_s)  {\partial
I_{(-\infty,0]}( \uun )} | \dotut|  \quad  \text{ in $\Gmc \times
(0,T)$,}
\\
\label{condtetas-concrete} & \nabla ( f(\tetas)) \cdot  \nn_s =0
\qquad \text{in $\partial \Gmc \times (0,T)$,}
\\
&-  \hbox{div}\ssigma={\bf f}  \ \text{ with } \ \ssigma=K_e\tensore+ K_v\tensoret+\vartheta \matrid \quad \hbox{in }\Omega\times (0,T),\label{eqI-concrete}\\
&{\bf u}={\bf 0}\quad\hbox{in }\Gdir\times (0,T),
\quad  \ssigma{\bf n}={\bf g} \quad \hbox{in }\Gnew\times (0,T),\label{condIi-concrete}\\
&\sigman\in -\chi \uun - {\partial I_{(-\infty,0]}(\uun)}
\quad\hbox{in }{\Gmc} \times (0,T),\label{condIii-concrete}\\
& \sigmat\in -\chi \uut -\fc(\teta-\teta_s){\partial
I_{(-\infty,0]}( \uun )} {\bf d} (\dotu)
\quad\hbox{in }{\Gmc} \times (0,T),\label{condIiii-concrete}\\
&\partial_t{\chi} + \partial I_{(-\infty, 0]}(\partial_t \chi)
-\Delta\chi+\partial I_{[0,1]}(\chi) +{\ele \gamma}'(\chi)  \ni
-\lambda'(\chi)(\vartheta_s-\vartheta_{\mathrm{eq}})- \frac 1 2
|\uu|^2\quad\hbox{in }{\Gmc} \times (0,T),\label{eqII-concrete}
\\
&\partial_{\nn_s} \chi=0 \quad \text{in $\partial {\Gmc} \times
(0,T)$.} \label{bord-chi-concrete}
\end{align}
\end{subequations}

In  system \eqref{system-concrete}, we have used the following notation:
 $\mathbf n$ is the outward unit normal to $\partial\Omega$, for which we suppose
 $\partial\Omega = \overline{\Gdir} \cup  \overline{\Gnew} \cup \overline{\Gmc}$,
with $\Gdir$ ($\Gnew$, resp.) the Dirichlet (the Neumann) part of the boundary where zero displacement (a fixed traction) is prescribed.
 As for $\Gmc$, we require that it is also sufficiently smooth (and denote by $\nns$ the outward unit normal to $\partial \Gmc$) and \emph{flat} (cf.\ condition \eqref{assumpt-domain} below); for simplicity
  we shall
  write $v$, in place of $v_{|\Gammac}$, for the trace
on $\Gammac$ of a function $v$ defined in $\Omega$.
  We also
adopt
 the following  convention:
 given a vector   $\vv\in \R^3$,  we denote by $\vvn$ and
$\vvt$ its normal component and its tangential part, i.e.\
 $\vvn:= \vv \cdot \mathbf{n} $ and $ \vvt:=
\vv - \vvn \mathbf{n} $.
Analogously, the normal component and the tangential part of the {\ele Cauchy}
stress tensor $\ssigma$ (while $\tensore$ is the small-strain tensor), are denoted by $\sigman$ and $\sigmat$, with
$\sigman:= \ssigma \mathbf{n} \cdot
\mathbf{n}$ and  $\sigmat:=  \ssigma \mathbf{n}  - \sigman
\mathbf{n}$. Finally, the
multivalued
 operator $\partial I_C:
\R \rightrightarrows \R$
(with $C$ the interval $[0,1]$ or the half-line $(-\infty,0])$ is the subdifferential (in the sense of convex analysis) of the indicator function of the  convex set $C$.

While referring the reader to \cite[Sec.\ 2]{bbr6} for the rigorous derivation of system \eqref{system-concrete}, let
us now briefly comment on the features, and the meaning, of the single equations.
The momentum balance \eqref{eqI-concrete}, where the elasticity and viscosity tensors $K_e$ and $K_v$ are positive-definite
and satisfy suitable symmetry conditions and $\mathbf{f}$ is a given volume force,  is supplemented by  the boundary conditions \eqref{condIi-concrete} on $\Gdir$
and $\Gnew$ (with $\mathbf{g}$ given), and by \eqref{condIii-concrete}--\eqref{condIiii-concrete} on the contact surface $\Gmc$.
Observe that \eqref{condIii-concrete} can be recast  in complementarity form as
\begin{align}
&\uun\leq 0,\quad \sigman+  \chi \uun \leq 0,\quad \uun (\sigman+
\chi \uun)=0 \quad\hbox{in }{\Gmc} \times (0,T).\label{Signorini}
\end{align}
For $\chi=0$ (i.e.\ no adhesion), these conditions reduce to the classical Signorini
conditions for unilateral contact. Instead, for $0 < \chi \leq 1$ \eqref{condIii-concrete} allows for  $\sigman$
positive, namely the  action  of the adhesive substance on $\gc$  prevents
separation when a tension is applied.
In \eqref{condIiii-concrete}, $\mathbf{d}: \R^3 \rightrightarrows \R^3$ is the subdifferential of
the functional $\Psi: \R^3 \to [0,+\infty)$ given by
$\Psi(\vv) = |\vvt|$, viz.\
\[
\mathbf{d}(\vv) = \begin{cases}
\frac{\vv_{\mathrm{T}}}{|\vv_{\mathrm{T}} |} & \text{if
$\vv_{\mathrm{T}} \neq \mathbf{0}$}
\\
\{ \ww_{\mathrm{T}}\, : \ \ww \in \overline{B}_1 \} & \text{if
$\vv_{\mathrm{T}} = \mathbf{0}$,}
\end{cases}
\]
 with $\overline{B}_1$
 the closed unit ball in $\R^3$. Therefore,
in view of   \eqref{condIii-concrete},
 condition  \eqref{condIiii-concrete} rephrases as
\begin{equation}
\label{Coulomb}
\begin{aligned}
&|\sigmat +\chi \uut|\leq  \fc(\teta-\teta_s)|\sigman + \chi \uun | && \text{in } \Gmc \times (0,T),\\
&|\sigmat +\chi \uut|< \fc(\teta-\teta_s)|\sigman + \chi \uun |\Longrightarrow \dotut = \mathbf{0} && \text{in } \Gmc \times (0,T),\\
&|\sigmat +\chi \uut|= \fc(\teta-\teta_s) |\sigman + \chi \uun |\Longrightarrow
\exists\,  \nu \geq 0: \ \dotut=-\nu (\sigmat +\chi \uut) && \text{in } \Gmc \times (0,T),
\end{aligned}
\end{equation}
which generalize the {\it dry friction} Coulomb law, to the case
when adhesion effects are taken into account. Note  that the \emph{positive} function $\fc$ in \eqref{Coulomb}  \ber has the meaning of a \edr
 \emph{friction coefficient}.

The temperature equations \eqref{e1-concrete} and \eqref{eqtetas-concrete}
(where
\bec $g$ is related to the heat flux in the bulk, $h$ is a given heat source, and \eec
$k$ and $\lambda$ are suitably smooth functions),
 feature the singular terms $ \ln(\teta)$
 and $\ln(\tetas)$,  which originate from deriving  \eqref{e1-concrete} and \eqref{eqtetas-concrete}
 from  the \emph{entropy-balance}, in place of the  internal energy balance, equations in the bulk  domain
 and on the contact surface.  These
 terms ensure the strict positivity of $\teta$ and $\tetas$, which is  a {\ele necessary} property in view of the
 thermodynamical consistency of the model.
 Equation  \eqref{e1-concrete} for $\teta$ is coupled to the
 quasi-static momentum balance
 though the term  $- \dive
(\partial_t\mathbf{u})$, {\ele related} to the {\ele presence of a } thermal expansion contribution $\vartheta \matrid
$ {\ele in } the stress tensor $\ssigma$ \eqref{eqI-concrete}. The coupling to the equations for $\tetas$ and $\chi$ occurs through the Robin type boundary condition \eqref{condteta-concrete-2}, where the frictional contribution
$\fc'(\teta-\teta_s)  {\partial I_{(-\infty,0]}( \uun )} |
\dotut|$  features as  a source of heat on the contact surface $\gc$.
Accordingly, this term also appears on the right-hand side of the equation for $\tetas$, where the function
$\ftetas$, {\ele in terms of which the heat flux on the contact surface is defined,}  will be chosen in a suitable way, cf.\ \eqref{ftetas_intro} below.

 The  evolution equation \eqref{eqII-concrete} for the  internal
variable $\chi$ has a   doubly nonlinear character, in that it
features the subdifferential operators $\partial I_{(-\infty, 0]}, \ \partial I_{[0,1]}
: \R \rightrightarrows \R$. The former acts on $\partial_t \chi$, thereby encompassing in the model
the constraint $\partial_t \chi \leq 0$, i.e.\ that the degradation process of the
adhesive on $\Gmc$ is \emph{irreversible}. The latter subdifferential operator enforces
 the constraint that $\chi$ takes values
 in the (physically admissible) interval $[0,1]$.
 The function ${\ele \gamma}'$ is a smooth {\ele (possibly)} non-monotone perturbation of the monotone term
 $\partial I_{[0,1]}$;  {\ele \bec it derives from \eec  some (possibly) non-convex contribution to the free energy functional}.  \bec Finally, \eec
$ \vartheta_{\mathrm{eq}} $ is a critical temperature.

\paragraph{\bf Analytical difficulties.}
The analysis of system
\eqref{system-concrete} presents   difficulties of various type:
\begin{compactenum}
\item[1)]
The coupling of \emph{bulk}
and \emph{(contact) surface} equations  requires sufficient
 regularity of the bulk variables $\teta$ and $\uu$ for their
traces on $\gc$ to make sense.
\item[2)] On the other hand,
 the highly nonlinear character of the
 equations (due to the presence of several singular and  multivalued operators),
 as well as the \emph{mixed} boundary conditions for the bulk  variables $\uu$
 and $\teta$, do not allow for elliptic regularity estimates which could enhance the spatial regularity of $\uu$, $\partial_t \uu$, and $\teta$. In particular, we are not in the position to get    $H^2$-regularity for the
 bulk variables.
 \item[3)] A further obstacle is the \emph{singular} character of
  the temperature  equations \eqref{e1} and \eqref{eqtetas}, due to the terms
$\partial_t \ln(\teta)$ and $\partial_t \ln(\teta_s)$.
Because of these terms, the basic energy estimates on system
\eqref{system-concrete} leads to a very weak
time-regularity of $\teta$ and $\teta_s$. As we will see, it is possible to improve such regularity,
for the sole $\tetas$,
only through a series of enhanced estimates, which in turn rely on a precise form for the {\ele function}
$\ftetas$, cf.\  \eqref{ftetas_intro} below.
\item[4)] The \emph{doubly nonlinear} character of equation \eqref{eqII-concrete} for $\chi$, with the two \emph{unbounded}
subdifferential terms $\partial I_{(\infty,0]}(\partial_t \chi)$ and $\partial I_{[0,1]}(\chi)$.
\item[5)] A major analytical problem is brought about by the coupling of unilateral contact and the
dry friction Coulomb law in the model. This leads to
 the presence of the  multivalued operator
 $ \partial I_{(-\infty,0]} $
  in the coupling terms between the equations for $\teta$,
$\teta_s$, and~$\uu$, and to the product of two subdifferential terms in   \eqref{condIiii-concrete}.
\end{compactenum}

Let us stress that contact problems with friction involve severe, and unresolved, difficulties
even in
 the
case without adhesion. That is why, as  done in many other works (cf.\ the references in
   \cite{bbr5}), starting from the pioneering paper \cite{duvaut} by  \textsc{Duvaut},
 we  have to regularize \eqref{condIiii-concrete}   by resorting to  a {\it nonlocal version} of the Coulomb law.
 More precisely, we
 shall replace the nonlinearity in \eqref{condIiii-concrete} involving friction  by
the  term
\begin{equation}
\label{replacement} \text{
 $ \fc(\teta-\teta_s)|\Reg( \partial I_{(-\infty,0]}(\uun))| \mathbf{d}(\partial_t\uu)$,}
\end{equation}
and, correspondingly, the term
$\fc'(\teta-\teta_s)  \partial
I_{(-\infty,0]}( \uun ) | \dotut| $ in \ber \eqref{condteta-concrete-2} and \eqref{eqtetas-concrete} \edr by
\begin{equation}
\label{replacement2}
\fc'(\teta-\teta_s)   | \Reg(\partial
I_{(-\infty,0]}( \uun )) |  | \dotut|.
\end{equation}
  In \eqref{replacement} and \eqref{replacement2}, $\Reg$ is a
 regularization operator
 with suitable properties, cf.\ Hypothesis \ref{hyp:4}.  The regularized friction law resulting from
 the replacement \eqref{replacement} in \eqref{condIiii-concrete}  can be interpreted as
  taking into account
 nonlocal interactions on the contact surface.

 We  now dwell on the difficulties attached to the doubly nonlinear character of \eqref{eqII-concrete},
  which is   due to the inclusion of \emph{unidirectionality} in the model.
 In order to
 prove the existence of a solution $\chi$ to a reasonably strong formulation of \eqref{eqII-concrete}, featuring
 two selections $\zeta \in \partial I_{(\infty,0]}(\partial_t \chi)$ and $\xi \in \partial I_{[0,1]}(\chi)$,
    it is essential to
   estimate the terms
   $\partial I_{(\infty,0]}(\partial_t \chi)$ and $\partial I_{[0,1]}(\chi)$ (formally written as single-valued)
     \emph{separately} in  some suitable function space, in fact $L^2(0,T; L^2(\gc))$. Starting from the paper
      \cite{bfl} on the analysis of a  model  for irreversible  phase transition, it has been well known that such
      an estimate can be achieved by testing \eqref{eqII-concrete} by the (formally written) term
       $\partial_t (-\Delta \chi + \partial I_{[0,1]}(\chi)) $ and employing monotonicity arguments.
       The related calculations also involve an integration by parts in time on the right-hand side of
\eqref{eqII-concrete}, which in turn requires suitable time-regularity for $\tetas$.

That is why, in order to carry out the crucial
estimate for handling of the terms
   $\partial I_{(\infty,0]}(\partial_t \chi)$ and $\partial I_{[0,1]}(\chi)$ in \eqref{eqII-concrete}, we will need to enhance the time-regularity of $\tetas$. This can be done through a series of estimates, which  partially
  rely on choosing in \ber \eqref{eqtetas-concrete} \edr a {\ele function} $f$ tailored to the logarithmic nonlinearity therein.

  It seems to us that the structure of these estimates, and the technical reasons underlying our hypotheses on the various nonlinearities of the system, can be highlighted upon developing our analysis on an \emph{abstract} version of system
  \eqref{system-concrete}. 
   \paragraph{\bf A generalization of system \eqref{system-concrete}.} Hereafter, we shall address the following
   PDE system coupling  a volume process in a domain $\Omega \subset \R^3$, with a   surface process
    occurring on a portion $\Gammac$ of the boundary of $\Omega$, which fulfills $\partial\Omega = \overline{\Gdir} \cup  \overline{\Gnew} \cup \overline{\Gammac} $. The surface process is described by  a suitable internal variable
     $\chi$, and thermal effects, in the bulk and on the surface, are accounted for   through the temperature variables $\teta$ and $\tetas$.  Accordingly, we have
    \begin{subequations}
 \label{abstract-system}
\begin{align}
 \label{e1} &\partial_t (\lteta(\vartheta)) -
\dive (\partial_t\mathbf{u}) -\mathrm{div}(\nabla \fteta(\teta)) = h
\qquad \text{in $\Omega \times (0,T)$,}
\\
\label{condteta} & \nabla \fteta(\teta)  \cdot \mathrm{n}
\begin{cases} =  0 & \text{in $(\partial \Omega \setminus \Gammac)
\times (0,T)$,}
\\
\in -k(\chi) (\vartheta-\vartheta_s)-\fc'(\teta-\teta_s) \Psi(\dotu)
|\Reg(\partial \Phi (\uu))| & \text{in $ \Gammac \times (0,T)$,}
\end{cases}
\\
\label{eqtetas} &\partial_t (\ltetas(\vartheta_s)) - \partial_t
(\lambda(\chi))  -\mathrm{div} (\nabla \ftetas( \vartheta_s)) \in
k(\chi) (\vartheta-\vartheta_s)+ \fc'(\teta-\teta_s) \Psi(\dotu)
|\Reg(\partial \Phi (\uu))|
 \quad  \text{ in $\Gammac
\times (0,T)$,}
\\
\label{condtetas} & \nabla \ftetas( \vartheta_s) \cdot  \nn_s =0
\qquad \text{in $\partial \Gammac \times (0,T)$,}
\\
&-  \hbox{div}\ssigma={\bf f}  \ \text{ with } \
\ssigma=K_e\tensore+
 K_v\tensoret+\vartheta \matrid \quad \hbox{in }\Omega\times (0,T),
 \label{eqI}
 \\
&{\bf u}={\bf 0}\quad\hbox{in }\Gdir\times (0,T), \quad
\ssigma{\bf n}={\bf g} \quad \hbox{in }\Gnew\times (0,T),
\label{condIi}
\\
& \label{new-abstract-sigma} \ssigma  {\bf n} +\chi \uu + \partial
\Phi(\uu) + \fc(\teta-\teta_s) \partial \Psi (\dotu)
  | \Reg( \partial \Phi (\uu))|\ber \ni {\bf 0} \edr \quad\hbox{in }{\Gammac} \times (0,T),
 \\
& \partial_t \chi + \partial \pseupot(\partial_t \chi)
-\Delta\chi+\partial \widehat{\beta}(\chi) + {\ele \gamma}'(\chi) \ni
-\lambda'(\chi) (\vartheta_s-\vartheta_{\mathrm{eq}})- \frac 1 2
|\uu|^2\quad\hbox{in }{\Gammac} \times (0,T),\label{eqII}
\\
&\partial_{\nn_s} \chi=0 \quad \text{in $\partial {\Gammac} \times
(0,T)$.} \label{bord-chi}
\end{align}
\end{subequations}
Observe that the temperature equations \eqref{e1} and \eqref{eqtetas} are a generalization of the ``concrete" equations
\eqref{e1-concrete} and \eqref{eqtetas-concrete}.
The logarithms
therein have been replaced by two (possibly different)
maximal monotone (single-valued) operators
\ber $\lteta$ and $\ltetas$ fulfilling suitable
conditions.   The function  $\fteta$  is  a strictly increasing, bi-Lipschitz, and otherwise general. Instead,
$\ftetas$  depends on the choice of $\ltetas$, as it is defined by
\edr
  \begin{equation}
\label{ftetas_intro} \ftetas (\teta_s) := \int_0^{\teta_s}
\frac1{\ltetas'(r)} \dd r\,.
\end{equation}
Admissible choices are, for instance
\[
(\ltetas(\tetas) = \ln(\tetas), \ftetas(\tetas) = \tetas^2), \qquad (\ltetas(\tetas) = \tetas, \ftetas(\tetas) = \tetas).
\]
Nonetheless,  in the second case the strict positivity of the temperature is no longer
directly ensured by the form of $\ltetas$ like in the first case.

\ber Note moreover that the \edr
terms $\partial I_{(-\infty,0]}(\uun)$ and $\mathbf{d}(\partial_t \uu)$ in \ber system \eqref{system-concrete} have been replaced
by the subdifferentials $\partial \Phi (\uu)$ and $\partial \Psi(\partial_t \uu)$ in  \eqref{new-abstract-sigma},
and accordingly in
\eqref{condteta} and \eqref{eqtetas} \edr (with $\Reg$ the regularization operator used in the analysis of frictional problems).
 Here, $\Phi$ and $\Psi$ are (possibly nonsmooth) positive, lower semicontinuous, and   convex functionals, and in addition $\Psi$ is positively homogeneous of degree $1$, i.e.\ $\Psi(l\vv) = l \Psi(\vv) $ for all $l \geq 0$ and $\vv \in \R^3$.  It turns out that the crucial requirement for tackling the simultaneous presence of the two terms  $\partial \Phi (\uu)$ and $\partial \Psi(\partial_t \uu)$
in \ber \eqref{new-abstract-sigma} \edr is that
\[
\partial \Phi (\uu) \text{ and } \partial \Psi(\vv) \text{ are orthogonal for all $\uu, \, \vv \in \R^3$,}
\]
cf.\ Hypothesis \ref{hyp:3},
which is obviously fulfilled in the case
of system \eqref{system-concrete}.

Finally, in \eqref{eqII} $\widehat {\rho}$ and $\widehat \beta$ are two convex and lower semicontinuous functionals,
such that $\mathrm{dom} (\widehat \beta) \subset [0,+\infty)$
in such a way as to ensure the positivity of the internal variable $\chi$, which is also crucial for the analysis of system \eqref{abstract-system} {\ele and guarantees the physical consistency of the phase variable}.
\paragraph{\bf An existence result for system \eqref{abstract-system}.} The main result of this paper, Theorem \ref{thm:main}, states the existence of solutions to the (Cauchy problem for a) variational formulation of
 \eqref{abstract-system}, in an appropriate functional framework which reflects the  energy estimates for
 this system. This variational formulation is given in \underline{Section \ref{s:2}}, where all the hypotheses on the various nonlinearities of the system and on the problem data are collected.

 The proof  follows by passing to the limit in a carefully devised approximate system, where
 some of the multivalued subdifferential terms featured in
 \eqref{abstract-system} are replaced by their Yosida regularizations.
 In particular, the possibly singular functions $\lteta$ and $\ltetas$ are regularized, and in addition the viscous terms $\eps \partial_t\teta$ and $\eps
 \partial_t \tetas$ are \ber included into \edr the left-hand sides of \eqref{e1} and
 \eqref{eqtetas}, thereby enhancing the time-regularity of the $\teta$- and $\tetas$-components of the approximate solutions.
\bec That is why, \eec  this procedure \bec requires \eec a technically delicate  construction of approximate initial data for $\teta$ and $\tetas$.

In
\underline{Section \ref{s:approximation}} we develop  the set-up of the approximate problem,
 \ber we state its variational formulation and prove a \emph{local-in-time} existence result for the approximate system via a
 Schauder
  fixed point argument. Then, we derive in \underline{Section \ref{s:4}} a series of a priori estimates on the approximate solutions. \bec Relying on them, \eec
   in
\underline{Section \ref{s:5}} we perform the passage to the limit with the approximate solutions
via refined compactness  and lower semicontinuity arguments. \edr
We then obtain in Theorem \ref{thm:exist-local} the (still local-in-time)  existence
  of solutions to  \eqref{abstract-system}.
  In this passage to the limit
  we have to deal with a significant
difficulty stemming from the
 coupling between thermal and frictional effects in the model. This is  the
 dependence of the friction coefficient $\fc$ on the thermal gap $\teta-\teta_s$.
 To tackle the passage to the limit in the approximation of the terms
$\fc'(\teta-\teta_s)  |\Reg( \partial\Phi(\uu))|\Psi(\partial_t\uu) $  in \eqref{condteta} and  \eqref{eqtetas},
and
$\fc(\teta-\teta_s)|\Reg( \partial\Phi(\uu))|\partial\Psi(\partial_t\uu)$ in \eqref{new-abstract-sigma},
it is essential to prove strong
 compactness for (the sequences approximating) $\teta$ and $\teta_s$ in suitable spaces.
 The key step   for $\teta$ is to  derive
  an estimate in
 $\BV (0,T;W^{1,3+\epsilon}(\Omega)')$
  for  all $\epsilon>0$, which enables us to apply a suitable version of the Lions-Aubin compactness
 theorem generalized to $\BV$ spaces.
  As for $\tetas$, exploiting condition \eqref{ftetas_intro}, we are indeed able to
  obtain an a priori bound for $\tetas$ in $H^1 (0,T;L^2(\Gammac){\ele )}$. As previously mentioned,
  this enhanced time-regularity estimate for $\tetas$
   plays a crucial role
  in the derivation of estimates for the terms
  $\partial\widehat{\rho}(\partial_t \chi)$ and
   $\partial \widehat{\beta}(\chi)$ in \eqref{eqII}.
   \begin{remark}
   \upshape
   \label{rmk:aside}
   This mismatch in the time-regularity properties of $\teta$ and $\tetas$, as well as the fact that {\ele (as we are dealing with possibly different thermal properties of the substances in the
   domain and on the contact surface)} we can allow for
   different choices of
    the functions $\lteta$ and $\ltetas$,
    highlights 
    that the temperature
   equations \eqref{e1} and \eqref{eqtetas} have a substantially different character. In fact,
   in \eqref{e1} the {\ele function} $\fteta$ can be fairly general, whereas
   in \eqref{eqtetas} the  function $\ftetas$ needs to be chosen of the form \eqref{ftetas_intro}, in order to allow for the enhanced time-regularity estimate for $\tetas$. Such an estimate
   could not be carried out on equation \eqref{e1}, due to the mixed boundary conditions \eqref{condteta}, which are also responsible for the low spatial regularity of $\teta$.
\end{remark}

The last step in our existence proof consists in the extension of the local-in-time solution to the Cauchy problem for system \eqref{system-concrete}.
This prolongation  procedure follows the lines of an extension argument from \cite{bbr1}.
  Indeed, for technical reasons that shall be expounded at the beginning of \underline{Section \ref{s:6}},
   it is necessary to extend the local-in-time solution, whose existence is guaranteed by Theorem \ref{thm:exist-local}, along with its approximability properties.
  This makes the prolongation argument rather complex, that is why we have devoted to it the whole Sec.\ \ref{s:6}.
\ber Finally, in \bec the Appendix \eec
we collect a series of
auxiliary results, among which some lemmas addressing the approximation of the initial data for $\teta$ and $\tetas$ and the properties of the functions regularizing the nonlinearities of the system.  \edr

\section{\bf Main result}
\label{s:2}
{\ele Before stating the analytical problem we are solving and the corresponding  existence result, we first \bec set up the  notation and  the assumptions. \eec}

{\ele \subsection{Setup}}
\label{ss:2.1} \noindent Throughout the paper we shall
 assume that
 \begin{equation}
 \label{assumpt-domain}
 \begin{gathered}
 \text{
 $\Omega$ is a
bounded   Lipschitz  domain in $\R^3$, with }
\\
\partial\Omega= \ber \overline \Gdir \cup \overline \Gnew \cup \overline {\Gammac}, \ \text{ $\Gdir$, $\Gnew$, $\Gammac$,
 open disjoint subsets in the relative topology of $\partial\Omega$,
such that  } \edr
\\
\mathcal{H}^{2}(\Gdir), \,  \mathcal{H}^{2}(\Gammac)>0, \qquad
\text{and ${\Gammac} \subset \R^2$ a  sufficiently smooth
\emph{flat} surface.}
\end{gathered}
\end{equation}
 More precisely,
by \emph{flat} we mean that $\Gammac$ is a subset of a hyperplane of
$\R^3$ and $\mathcal{H}^2(\Gammac)=\mathcal{L}^2(\Gammac)$,
$\mathcal{L}^d$ and $\mathcal{H}^d$ denoting the $d$-dimensional
Lebesgue and  Hausdorff measures{\ele, respectively}. As for smoothness, we require that
$\Gammac$ has a $\mathrm{C}^2$-boundary. 
\begin{notation}
\label{notation-1} \upshape
 Given a Banach space $X$, we denote by
 $\pairing{}{X}{\cdot}{\cdot}$ the duality pairing
between its dual space  $X'$ and $X$ itself, and by
$\Vert\cdot\Vert_X$
 the norm in $X$. In particular, we shall
use the following   short-hand notation for  function spaces
\[
\begin{gathered}
H:= L^2(\Omega), \quad V:=H^1(\Omega),\quad
 \bsH:= L^2 (\Omega;\R^3), \quad \bsV :=H^1 (\Omega;\R^3),
 \\  \Hc: = L^2
({\Gammac}), \quad \Vc: = H^1 ({\Gammac}), \quad \Yc:=
H^{1/2}_{00,\Gdir}({\Gammac}),
 \\
   \bsW:=\{{\bf v}\in \bsV\, : \ {\bf v}={\bf 0}\hbox{ a.e.\ on
}\Gdir\},\quad \bsH_{{\Gammac}} := L^2 ({\Gammac};\R^3), \quad
\bsY_{{\Gammac}}
:= H^{1/2}_{00,\Gdir}(\Gammac;\R^3),  
\\
\end{gathered}
\]
where we recall that
\[
 H^{1/2}_{00,\Gdir}(\Gammac)=
\Big\{ w\in H^{1/2}(\Gammac)\, :  \ \exists\, \tilde{w}\in
H^{1/2}(\Gamma) \text{ with } \tilde{w}=w  \text{ in $\Gammac$,} \
\tilde{w}=0 \text{ in $\Gdir$} \Big\}
\]
 and $H^{1/2}_{00,\Gdir}(\Gammac;\R^3)$ is analogously defined. We will also use the space
 $H^{1/2}_{00,\Gdir}(\Gnew;\R^3).$
The space $\bsW$  is endowed with the natural norm induced by
$\bsV$. We will make use of the operator
\begin{equation}
A:\Vc \to \Vc' \qquad \pairing{}{\Vc}{A\chi}{w}:= \int_\gc \nabla
\chi {\ele \cdot}\nabla w \dd x \  \text{ for all }\chi, \, w \in \Vc
\end{equation}
and of the notation 
\begin{equation}
\label{mean-value} m(w):= \frac1{\mathcal{L}^d(A)} \int_A w \dd x
\quad \text{for } w \in L^1(A).
\end{equation}
\end{notation}
\paragraph{\bf Linear viscoelasticity.}
{\ele We are in the framework of linear viscoelasticity theory (see e.g.\ \cite{bbr6} for some more details). In particular, we prescribe that
the fourth-order tensors $K_e$ and
$K_v$ (denoting the elasticity and the viscosity tensor,
respectively) are symmetric and positive
 definite.  Moreover, we require that they are uniformly bounded,
in such a way that  the following bilinear symmetric forms $a, b : \bsW \times \bsW
\to \R$,   defined~by
$$
\begin{aligned}
a({\bf u},{\bf v}):=\int_{\Omega} \varepsilon({\bf u})K_e\varepsilon(\vv) \dd x
\qquad \qquad
b({\bf u},{\bf v}):=\int_{\Omega} \varepsilon({\bf u}) K_v
\varepsilon({\bf v}) \dd x  \quad  \text{for all }  \uu, \vv
\in \bsW,
\end{aligned}
$$}
are continuous. In particular, we have
\begin{equation}
\label{continuity} \exists \, \bar{C} >0: \ |a(\uu, \vv)| + |b(\uu,
\vv)| \leq  \bar{C}\| \uu\|_{\bsW} \| \vv\|_{\bsW} \quad \text{for
all } \uu, \vv \in \bsW.
\end{equation}
Moreover, since $\Gdir$ has positive measure,
 by Korn's inequality we deduce that $a(\cdot,\cdot)$ and
$b(\cdot,\cdot)$ are $\bsW$-elliptic, i.e., there exist $C_{a},
C_{b}>0 $ such that
\begin{equation}
\label{korn_a}
 a({\bf u},{\bf u})\geq C_a\Vert{\bf u}\Vert^2_{\bsW}\,,
 \qquad \qquad
 b({\bf u},{\bf u})\geq C_b\Vert{\bf u}\Vert^2_{\bsW} \qquad
\text{for all }\uu\in \bsW.
\end{equation}

\subsection{Assumptions}
\label{ss:2.2} In order to tackle the analysis of the PDE system
\eqref{abstract-system}, we require the following.
\begin{hypothesis}\label{hyp:1}
{\ele For the functions $\lteta$ and $\ltetas$  in \eqref{e1} and \eqref{eqtetas}}
 we assume that
\begin{subequations}
 \label{cond-L}
\begin{align}
 & \lteta: \mathrm{D}(\lteta) \ber \subset \R \edr \to \R \text{ maximal monotone, with
 } \mathrm{D}(\lteta)  \text{ a (possibly unbounded) open interval}
  \label{cond-L1}
 \\
&
 \label{cond-L2}
 \lteta \in \mathrm{C}^1(\mathrm{D}(\lteta))  \text{ and } \frac{1}{L'}\in \mathrm{C}^{0,1}(\overline{\mathrm{D}(\lteta)}).
\end{align}
Moreover,  denoting by $J$ a primitive of $\lteta$, we impose that
the Fenchel-Moreau convex conjugate $J^*$  of $J$ (recall that its
derivative coincides with the inverse function $\lteta^{-1}$),
fulfills the following {\em coercivity} condition
\begin{equation}
\label{cond-L3}
 \exists \, C_1, \ C_2   >0  \qquad \forall\, \teta \in
\mathrm{D}(\lteta) \,
  : \quad
  J^*(\lteta(\teta)) \geq C_1 |\teta| -C_2\,.
\end{equation}
\end{subequations}
{\ele Analogously, we assume for $\ltetas$}
\begin{subequations}
\label{cond-ell}
\begin{align}
 \label{cond-ell1}
& \ltetas: \mathrm{D}(\ltetas) \ber \subset \R \to \R \text{ maximal monotone, with
 } \mathrm{D}(\ltetas)  \text{ a (possibly unbounded) open interval} \edr
\\
& \label{cond-ell2}
 \ltetas \in \mathrm{C}^{\ele 1}(\mathrm{D}(\ltetas)) \text{ and } \frac{1}{\ltetas'}\in \mathrm{C}^{0,1}(\overline{\mathrm{D}(\ltetas)}),
\end{align}
 as well as, again, the {\em coercivity} condition
\begin{equation}
\label{cond-ell4}
 \exists \, c_1, \ c_2   >0  \qquad \forall\, \tetas \in \mathrm{D}(\ltetas) \, : \
  j^*(\ltetas(\tetas)) \geq c_1 |\tetas| -c_2\,,
\end{equation}
where $j^*$ is
  the Fenchel-Moreau convex conjugate
of $j$, $j$ denoting a primitive of $\ltetas$.
\end{subequations}
\end{hypothesis}
A straightforward consequence of \eqref{cond-L} and of
\eqref{cond-ell} is that
\[
\left\{
\begin{array}{lll}
\lteta'(x) >0  & \text{for all } x \in  \mathrm{D}(\lteta),
\\
\ltetas'(x) >0  & \text{for all } x \in  \mathrm{D}(\ltetas),
\end{array}
\right.
 \ \
\text{ hence } \lteta \text{ and } \ltetas \text{ are invertible.}
\]
Furthermore,
 it is not restrictive  to
suppose that
\begin{equation}
\label{not-restrictive-prima} 0 \in {\mathrm{D}(J)} \quad \text{with } J(0)=0
\end{equation}
(the latter relation is trivially obtained with a translation argument), and the same for
$j$. Since $\overline{\mathrm{D}(\ltetas)} = \overline{\mathrm{D}(j)} $, this in particular
implies that
\begin{equation}
\label{not-restrictive} 0 \in \overline{\mathrm{D}(\ltetas)}\,,
\end{equation}
which will be
convenient for the definition of $f$ and $f_\eps$ later on.
\begin{example}
\upshape \label{ex-for-L} \upshape An example  for $\lteta$ in
accord with conditions \eqref{cond-L1}--\eqref{cond-L3} is
\begin{equation}
\label{loga}
 \lteta(\teta)=\ln (\teta) \qquad \forall\, \teta
\in \mathrm{D}(\lteta)=(0,+\infty)\,.
\end{equation}
In this case, \bec with, \eec  e.g.\ $J(\teta)= \teta(\ln(\teta) -1)$ for
all $\teta \in (0,+\infty)$ as primitive of $L$, we see that $J^*
(w)=e^w= \lteta^{-1}(w)$ for all $w \in \R$, and \eqref{cond-L3} is
satisfied. This choice of $\lteta$ is particularly meaningful from a
modeling viewpoint, since it enforces that $\teta>0$, in accord with
its interpretation as the absolute temperature. Clearly,
$\ltetas(\teta_s) = \ln(\teta_s) $ is also an admissible choice for
$\ltetas$.

As already mentioned in the introduction, let us point out that, in
fact, for our analysis we do not need the positivity of $\teta$ and
$\teta_s$ (namely, that $\mathrm{D}(\lteta), \, \mathrm{D}(\ltetas)
\subset (0, +\infty)$). Hence, other admissible choices  for \ber $\lteta$ and \edr $\ltetas $
are
\begin{equation}
\label{lteta-id} \lteta(\teta) =\teta, \qquad \ltetas(\tetas) =
\tetas \qquad \text{with } \mathrm{D}(\lteta)=
\mathrm{D}(\ltetas)=\R.
\end{equation}
Let us also stress that, in system \bec  \eqref{abstract-system} \eec
 we can in principle combine  two distinct choices
for $\lteta$ and $\ltetas$.
\end{example}
\begin{hypothesis}
\label{hyp:2} {\ele As far as the  functions $\fteta$ and $\ftetas$
\bec are concerned, \eec
we} \bec impose \eec that
\begin{align}
 \label{cond-g}
& \fteta \in \mathrm{C}^1(\R) \text{ and } \exists\, c_3, c_4 >0\ :
\ \forall\, x \in \R\ \ c_3\leq \fteta'(x)\leq c_4.
\end{align}
\noindent As for $\ftetas$, as previously mentioned we require that
it is  is the primitive of $\frac{1}{\ltetas'}$; in view of
\eqref{not-restrictive}, we set
\begin{align}
 \label{def-f}
& \ftetas(x)=\int_0^x \frac{1}{\ltetas'(s)}\, \dd s\,,\ \forall\, x
\in \overline{\mathrm{D}(\ltetas)}\,.
\end{align}
\end{hypothesis}

\begin{example}
\upshape \label{ex-for-fg} Clearly, the {\ele function} $\ftetas$
depends on the choice for $\ltetas$. For example,
\[
\left\{
\begin{array}{llll}
  \ltetas(\tetas) = \tetas  &  \Rightarrow & \ftetas(\tetas) = \tetas,
 \\
   \ltetas(\tetas) = \ln(\tetas)  &  \Rightarrow & \ftetas(\tetas) = \tetas^2,
  \end{array}
\right.
\]
\end{example}
\begin{hypothesis}[The subdifferential operators in the momentum \ber balance \edr
equation] 
\label{hyp:3}
 We suppose that
\begin{equation}
\label{ass-psi} \Psi : \R^3 \to [0,+\infty) \text{ is  convex,
non-degenerate, and positively $1$-homogeneous}
\end{equation}
i.e.\ $\Psi$ satisfies
\[
\Psi(\vv)>0 \text{ if } \vv \neq 0, \qquad \Psi(l \vv) = l \Psi(\vv)
\text{ for all } l \geq 0 \text{ and } \vv \in \R^3, 
\]
(in fact, under positive homogeneity of degree $1$, sublinearity is
equivalent to convexity).
As for the function $\Phi$, we assume that
\begin{equation}
\label{hyp:phi}
\begin{gathered}
 \Phi : \R^3  \rightarrow [0, + \infty] \, \text{is proper, convex and  lower semicontinuous, with $\Phi(0)=0$}
\end{gathered}
\end{equation}
and effective domain $\mathrm{dom}(\Phi)$.
We impose the following ``compatibility'' condition between the
respective subdifferential operators $\partial \Psi: \R^3
\rightrightarrows \R^3$ and $\partial \Phi: \R^3 \rightrightarrows
\R^3$:
\begin{equation}
\label{ass:orthogonal} \forall\, \uu \in \mathrm{dom} (\Phi) \text{
and } \vv \in \ber \R^3 \edr, \qquad \forall\, \eeta \in
\partial \Phi(\uu) \text{ and } \zz \in \partial \Psi(\vv)\, :
\qquad \eeta \cdot \zz =0.
\end{equation}
\end{hypothesis}
In the variational formulation of system \eqref{abstract-system}
(cf.\ \eqref{system-weak} ahead), 
 in fact we are going to make use of the
abstract realization of $\Phi$ as a functional on
$\bsY_{{\Gammac}}$, viz.\
\begin{equation}
\label{funct-phi} \bvarphi:  \ber \bsY_{{\Gammac}} \edr \to [0,+\infty] \ \text{ defined by }
\ \bvarphi(\uu):= \left\{
\begin{array}{ll}
\int_{\Gammac} \Phi(\uu) \dd x & \text{if $\Phi(\uu) \in L^1
({\Gammac})$,}
\\
+\infty
 & \text{otherwise}
 \end{array}
 \right. \quad \text{for all $\uu  \in \bsY_{{\Gammac}}.$}
 \end{equation}
Since $\bvarphi: \bsY_{{\Gammac}} \to [0,+\infty]$ is a proper,
convex and lower semicontinuous functional on $\bsY_{{\Gammac}}$,
its subdifferential
\[
\partial \bvarphi: \bsY_{{\Gammac}}
\rightrightarrows \bsY_{{\Gammac}}' \text{ is a maximal monotone
operator.}
\]
With a slight abuse of notation, we will  \bec use \eec the symbol $\eeta$
not only for the elements of $\partial \Phi$, but also for those of
 $\partial \bvarphi$.

Instead, in formulation  \eqref{system-weak}  we are  going to stay
with the ``concrete'' subdifferential operator $ \partial \Psi : \bsH_{{\Gammac}} \rightrightarrows \bsH_{{\Gammac}}$,
which with a slight abuse 
\bec we \eec  denote in the same way as the operator $
\partial \Psi: \R^3 \rightrightarrows \R^3
$ inducing it. It follows from \eqref{ass-psi} (observe that $\mathrm{dom}(\Psi)
=\R^3$), that \bec $\partial \Psi: \R^3 \rightrightarrows \R^3 $ satisfies \eec
\begin{equation}
\label{bounded-operator}
\exists\, C_{\Psi}>0 \ \forall\, \zz\in \partial \Psi(\vv) \, :
\qquad |\zz| \leq C_\Psi.
\end{equation}
\begin{example}
\upshape \label{ex:contact-friction} The prototypical example of
functionals $\Phi$ and $\Psi$ complying with Hypothesis
\eqref{hyp:3} comes from the modeling of frictional contact. In this
frame, we have 
\begin{equation}
\label{choice-Signorini}
\begin{aligned}
&
 \Phi(\uu) := I_{(-\infty, 0]}(\uun),
 \\
 &
 \Psi(\vv):= |\vvt|.
 \end{aligned}
\end{equation}
Clearly, the orthogonality condition \eqref{ass:orthogonal} is
fulfilled in this case.  Nonetheless, let us highlight that
\eqref{ass:orthogonal} allows for much more generality:  for
example, $\partial \Phi(\uu)$ and $\partial \Psi(\vv)$ might be of
the form
\[
\left\{
\begin{array}{lllll}
  \eeta \in \partial \Phi(\uu)   & \text{ with }\eeta = \eta \ww_1(u)
 &   \eta \in \R, &   \ww_1(\uu) \in \R^3,
\\
 \zz \in \partial \Psi(\vv)   & \text{ with }\zz = z \ww_2(\vv) &
z\in \R, &   \ww_2(\vv) \in \R^3,
\end{array}
\right.
\]
with $\ww_1(\uu)$ and $\ww_2 (\vv)$ depending on $\uu$ and $\vv$,
respectively, and such that
\[
\ww_1(\uu) \cdot \ww_2 (\vv)=0 \quad \text{for all }\uu, \, \vv \in
\R^3.
\]
\end{example}
\begin{hypothesis}[The regularizing operator $\Reg$]
\label{hyp:4} Following \cite{bbr5,bbr6} we require that there
exists $\nu>0$ such that
\begin{align}
 \label{hyp-r-1}
&
\begin{aligned}
 & \Reg: L^2 (0,T;\bsY_{{\Gammac}}') \to L^\infty
(0,T;L^{2+\nu}(\gc;\R^3))  \text{  is weakly-strongly continuous,
viz.}
\\
  &
  \eeta_n \weakto \eeta  \ \text{ in
$L^2 (0,T;\bsY_{{\Gammac}}')$} \ \  \Rightarrow \ \
 \Reg(\eeta_n) \to
\Reg(\eeta) \ \text{ in $L^\infty (0,T;L^{2+\nu}(\gc;\R^3)$}
\end{aligned}
\end{align}
for all  $(\eeta_n),\,\eeta \in L^2 (0,T;\bsY_{{\Gammac}}')$.
\end{hypothesis}

Observe that \eqref{hyp-r-1} implies that $ \Reg: L^2
(0,T;\bsY_{{\Gammac}}') \to L^\infty (0,T;L^{2+\nu}(\gc;\R^3)) $ is
bounded.
 We
refer to \cite[Example 3.2]{bbr6} for the explicit construction of
an operator $\Reg$ complying with \eqref{hyp-r-1}.
\begin{hypothesis}[The subdifferential operators in the equation for $\chi$]
\label{hyp:5}
 We assume that $\widehat{\beta} $ in \eqref{eqII}
 fulfills
\begin{equation}
\label{hyp-beta}
 \begin{gathered}
 \widehat{\beta}: \R \rightarrow  (-\infty,+\infty] \, \text{is proper, convex and  lower semicontinuous, with
 }  \ \mathrm{dom}(\widehat \beta) \subset [ 0,+\infty).
\end{gathered}
\end{equation}
In what follows, we use the notation $\beta:=
\partial\widehat{\beta}$.

We also require that
\begin{equation}
\label{hyp-rho}
 \begin{gathered}
 \pseupot: \R \rightarrow  [0,+\infty] \, \text{is proper, convex and  lower semicontinuous, with
 }  0 \in \mathrm{dom}(\pseupot).
\end{gathered}
\end{equation}
We use $\rho$ as a placeholder for $\partial \pseupot$.
\end{hypothesis}
Observe that, with a translation we can always confine ourselves to
the case in which $\pseupot(0)=0 = \min_{x\in \R}\pseupot(x)$,
therefore we may also suppose that
\begin{equation}
\label{addition-rho} 0 \in \rho(0).
\end{equation}
The simplest examples for $\widehat{\beta}$ and $\pseupot$ are
$\widehat{\beta} (\chi) = I_{[0,1]}(\chi)$ and $\pseupot(\chi_t) =
I_{(-\infty,0]}(\chi_t)$.
\begin{hypothesis}[The other nonlinearities]
\label{hyp:6} We assume that the functions $k$ in
\eqref{condteta}--\eqref{eqtetas}, $\fc$ in \eqref{condteta},
\eqref{condtetas}, and \eqref{new-abstract-sigma},
 $\lambda$ in
\eqref{eqtetas} and \eqref{eqII},
 and ${\ele \gamma}$ in \eqref{eqII}
 fulfill
\begin{align}
 &  \label{hyp-k}    k \, : \R \to [0,+\infty)\ \  \text{is
Lipschitz
continuous,}  
\\
& \label{hyp-fc}
 \fc \in \mathrm{C}^1(\R), \qquad \exists\, c_5,\, c_6
>0
 \ \forall\, x \in \R\, : \ \fc(x) \geq c_5, \ \
 |\fc'(x)| \leq c_5, \qquad
 \fc'(x) x \geq 0.
\\
&
 \label{cond-landa-enhanc}
 \lambda \in \mathrm{C}^2(\R) \text{ and } \exists\, c_{7}, c_{8} >0
\ \forall\, x \in \R\, :  \ |\lambda'(x)| \leq c_{7}, \
|\lambda{''}(x)| \leq c_{8},
\\
 &
 \label{hyp-sig} {\ele \gamma} \in \bec \mathrm{C}^{2} (\R), \eec  \ \text{with
} {\ele \gamma}': \R \to \R \text{ Lipschitz continuous.}
\end{align}
 \end{hypothesis}
\paragraph{\bf Assumptions on the problem  and on the initial data.}
We  require that
\begin{subequations}
\label{hyp-data}
\begin{align}
\label{hypo-h} &  h \in L^2 (0,T;V')\cap L^1 (0,T;H) \,,
\\
&  \label{hypo-f} \mathbf{f} \in L^2 (0,T;\mathbf{W}')\,,
\\
&  \label{hypo-g}
 \mathbf{g} \in L^2 (0,T;
H_{00,\Gdir}^{1/2}(\Gnew;\R^3)').
\end{align}
\end{subequations}
For later convenience, we remark that, thanks to
\eqref{hypo-f}--\eqref{hypo-g} the function $\mathbf{F}:(0,T) \to
\bfW'$  defined by
$$
 \pairing{}{\bsW}{\mathbf{F}(t)}{\vv}:=\pairing{}{\bsW}{\mathbf{f}(t)}{\vv}
 +\pairing{}{H_{00,\Gdir}^{1/2}(\Gnew;\R^3)}{\mathbf{g}(t)}{\vv}
\quad \text{for all } {\vv} \in  \bsW  \text{ and almost all } t \in
(0,T),
$$
 satisfies
\begin{equation}
\label{effegrande} \mathbf{F} \in L^2(0,T;\bfW') \,.
\end{equation}
For the initial data we impose that
\begin{subequations}
\label{hyp-initial}
\begin{align}
& \label{cond-teta-zero}
J^* (\lteta(\teta_0)) \in L^1(\Omega)
 \ \ \ \text{and} \ \ \
\lteta(\vartheta_0) \in H\,,
\\
& \label{cond-teta-esse-zero}
\teta_s^0 \in \Hc\,, \quad \ltetas(\vartheta_s^0) \in \Hc\,,
 \ \text{and} \ \ \
f(\teta_s^0) \in H^1 (\gc)\,,
\\
 & \label{cond-uu-zero} {\bf u}_0 \in \bsW \ \text{and} \ \uu_{0}
\in \dom (\bvarphi)\,,
\\
& \label{cond-chi-zero} \chi_0 \in H^2(\gc),     \ \partial_{\nn_s}
  \chi_0=0\text{ a.e.\ in } \partial\Gammac, \quad
\widehat{\beta}(\chi_0)\in
 L^1(\Gammac)\,.
\end{align}
\end{subequations}
Concerning the initial data $\vartheta_0$ and $\vartheta_s^0$, we observe that the first of \eqref{cond-teta-zero}
implies $\vartheta_0 \in L^{1}(\Omega)$, in view of \eqref{cond-L3}. Moreover,
the enhanced regularity \eqref{cond-teta-esse-zero} required of
$\teta_s^0$ reflects that we shall obtain a higher temporal
regularity for $\teta_s$ than for $\teta$, see Theorem
\ref{thm:main} ahead.
\subsection{Variational formulation of the problem and main result}
We are now in the position to detail the  formulation for the initial-boundary value problem associated with
system
\bec \eqref{abstract-system}. \eec
Observe that, while the temperature equations \eqref{e1} and \eqref{eqtetas} and the
momentum equation \eqref{eqI} need to be  formulated in dual spaces, the equation \eqref{eqII} for the internal variable $\chi$ can be  given a.e.\ in $\gc \times (0,T)$,
\bec due to the $H^2(\Gammac)$-regularity obtained for $\chi$. \eec
\begin{problem}
 \label{prob:irrev}
 \upshape
 Given a quadruple of 
data
 $(\vartheta_0, \vartheta_s^0 , \uu_0, \chi_0)$
 fulfilling \eqref{hyp-initial}, find
$(\vartheta,  \vartheta_s,   \uu,\chi,\eeta,\mmu, \xi, \zeta)$, with
\begin{subequations}
\begin{align}
& \label{reg-teta}
 \vartheta \in L^2 (0,T; V) \cap L^\infty (0,
T;L^1 (\Omega)), \qquad \fteta(\teta) \in L^2 (0,T; V),
\\
\label{reg-log-teta}
 & \lteta(\vartheta) \in L^\infty (0,T;H) \cap H^1
(0,T;V')\,,
 \\
& \label{reg-teta-s} \vartheta_s \in  L^2 (0,T; \Vc),
\\
& \label{reg-f-tetas} \ftetas(\vartheta_s) \in  L^2 (0,T; \Vc)\,,
 \\
& \label{reg-log-teta-s} \ell(\vartheta_s) \in L^\infty (0,T;\Hc)
\cap H^1 (0,T;\Vc')\,,
\\
& \uu \in H^1(0,T;\bsW) \,,\label{reguI}\\
&\chi \in L^{2}(0,T;H^2 (\Gammac))  \cap L^{\infty}(0,T;\Vc) \cap
H^1(0,T;\Hc) \,, \label{regchiI}
\\
 &
\eeta\in L^2(0,T; \bsY'_{{\Gammac}})\,, \label{etareg}
\\ & \mmu \in L^2(0,T;\bsH_{{\Gammac}})\,,\label{mureg}
\\
& \xi\in L^2(0,T; \Hc)\,, \label{xireg}
\\
& \zeta\in L^2(0,T; \Hc)\,, \label{zetareg}
\end{align}
\end{subequations}
\begin{align}
& \label{iniw} \vartheta(0)=\vartheta_0 \quad \aein \ \Omega\,,
\\
& \label{iniz} \vartheta_s(0)=\vartheta_s^0 \quad \aein \ \Gammac\,,
\\
& \label{iniu} \uu(0)=\uu_0 \quad {\aein \ \Omega}\,,
\\
& \label{inichi} \chi(0)=\chi_0 \quad {\aein \ \Gammac}\,,
\end{align}
  and satisfying
  \begin{subequations}
  \label{system-weak}
 \begin{align}
&
 \label{teta-weak}
\begin{aligned}
\pairing{}{V}{\partial_t \lteta (\vartheta)}{v}
 &  -\int_{\Omega} \dive(\partial_t\mathbf{u}) \, v \dd x
+\int_{\Omega}    \nabla \fteta(\vartheta) \, \nabla v  \dd x +
\int_{\Gammac} k(\chi)  (\vartheta-\vartheta_s) v  \dd x
\\ & +\int_{\Gammac} \fc'(\teta-\teta_s) |\Reg (\eeta)| \Psi(\dotu) v  \dd x
 = \pairing{}{V}{h}{v} \quad
\forall\, v \in V \ \hbox{ a.e.\ in }\, (0,T)\,,
\end{aligned}
\\
& \label{teta-s-weak}
\begin{aligned}
 & \pairing{}{\Vc}{\partial_t \ltetas(\vartheta_s)}{v}
-\int_{\Gammac}
\partial_t \lambda(\chi) \, v  \dd x    +\int_{\Gammac} \nabla \ftetas(\vartheta_s) \, \nabla
v \dd x
\\ &
 = \int_{\Gammac} k(\chi) (\vartheta-\vartheta_s) v  \dd x
 +\int_{\Gammac} \fc'(\teta-\teta_s) |\Reg (\eeta)| \Psi(\dotu) v  \dd x
  \quad
\forall\, v \in \Vc  \ \hbox{ a.e.\ in }\, (0,T)\,,
\end{aligned}
\\
 &
\begin{aligned}
 &  b(\dotu,\vv)  +a(\uu,\vv)+ \int_{\Omega} \vartheta \dive (\vv)  \dd x \\
& +\int_{\Gammac} \chi \uu \cdot \vv \dd x + \sideset{}{_{
\bsY_{{\Gammac}}}} {\mathop{\langle { \mbox{\boldmath$\eta$}} , \vv
\rangle}}
   +
\int_{{{\Gammac}}}\fc(\teta-\vartheta_s)   {\mmu}\cdot {\vv} \dd x =
\pairing{}{\bsW}{\mathbf{F}}{\vv}
 \quad \text{for all } \vv\in \bsW \ \hbox{ a.e.\ in }\, (0,T)\,,
 \end{aligned}
\label{eqIa}
\\
&\eeta\in
\partial \bvarphi(\uu) \quad \text{ in } \bsY_{{\Gammac}}',
\hbox { a.e.\ in }\, (0,T), \label{incl1}
\\
& \mmu \in |\Reg(\eeta)|\partial \Psi(\dotu)\hbox { a.e.\ in }\,
{\Gammac}\times (0,T), \label{incl1-bis}
\\
&
\partial_t{\chi}+\zeta+A\chi+\xi+{\ele \gamma}'(\chi)= -\lambda'(\chi)
\vartheta_s- \frac12 |\uu|^2   \quad\hbox{a.e. in } {\Gammac}\times(0,T),  \label{eqIIa-irr}\\
& \xi\in \beta(\chi)
  \quad\hbox{a.e.\ in } {\Gammac}\times
(0,T),\label{incl-beta-vincolo}
\\
& \zeta \in \rho(\partial_t\chi)
  \quad\hbox{a.e.\ in } {\Gammac}\times
(0,T).\label{incl-rho-vincolo}
\end{align}
\end{subequations}
\end{problem}

\begin{definition}
\label{def-energysol}
\bec We call a solution  to Problem \ref{prob:irrev} \emph{energy solution} if, in addition, it satisfies the energy inequality \eec
 \begin{equation}
\label{en-ineq-local}
 \begin{aligned}
  &\int_\Omega J^* (\lteta (\teta(t))) \dd x+\int_s^t
  \int_\Omega g'(\teta)|\nabla\teta|^2 \dd x \dd r
  + \int_\Gammac j^*(\ltetas(\teta_s(t))) \dd x
   +\int_s^t\int_\Omega f'(\tetas)|\nabla\tetas|^2 \dd x \dd r  \\
  &
 \quad
 +
  \int_s^t\int_{\Gammac}k(\chi)(\teta-\tetas)^2 \dd x \dd r
  +\int_s^t b(\partial_t\uu,\partial_t\uu) \dd r +\frac 1 2 a(\uu(t),\uu(t))
  +\frac 1 2\int_{\gc}\chi(t)|\uu(t)|^2 \dd x
  + \int_\Gammac \Phi (\uu(t)) \dd x
  \\
  & \quad  + \int_s^t \int_{\Gammac} \fc(\teta-\tetas)  \Psi(\partial_t\uu)|{\calR}(\rmD\Phi_\eps)|  \dd x \dd r +
\int_s^t\int_{\Gammac} \fc'(\teta-\tetas)\Psi(\partial_t\uu)|{\calR}(\rmD\Phi_\eps)|(\teta-\tetas) \dd x \dd r
  \\
  & \quad
  +\int_s^t \int_\Gammac |\partial_t\chi|^2 \dd x \dd r
  +\int_s^t \int_\Gammac  \zeta \chi_t \dd x \dd r
  +\frac12 \int_\Gammac |\nabla \chi(t)|^2 \dd x
  + \int_\Gammac \widehat{\beta} (\chi(t)) \dd x
   + \int_\Gammac {\ele \gamma} (\chi(t)) \dd x
   \\
  &\leq  \int_\Omega J^* (\lteta (\teta(s))) \dd x  +  \int_\Gammac j^*(\ltetas(\teta_s(s))) \dd x
  +   \int_0^t\pairing{}{V}{h}{\teta}\dd r +\int_0^t\pairing{}{\WW}{\FF}{\partial_t\uu} \dd r
  + \frac 1 2 a(\uu(s),\uu(s))
  \\
  & \quad
  + \int_\Gammac \Phi (\uu(s)) \dd x
  +
  \frac 1 2\int_{\Gammac}\chi(s)|\uu(s)|^2\dd x +
   \frac12 \int_\Gammac |\nabla \chi(s)|^2 \dd x
  + \int_\Gammac \widehat{\beta}_\eps (\chi(s)) \dd x
   + \int_\Gammac {\ele \gamma} (\chi(s)) \dd x
 \end{aligned}
\end{equation}
for almost all
\bec $0 \leq s \leq t
\leq T$, and
 and for $s=0$. \eec
\end{definition}

With our main  result, Theorem \ref{thm:main}, we state the existence of \bec an \emph{energy solution} \eec
to Problem \ref{prob:irrev}, with  the additional
properties \eqref{additional-regularities} below.  Namely, 
 $ \teta $ has bounded variation, as a function of time, with values in some dual space: this
in particular ensures that $t \mapsto \teta(t)$ is continuous, with values in that space, at almost all points
$t_0 \in (0,T)$. For $\teta_s$ we gain a better regularity, cf.\ \eqref{add-reg-tetas} and \eqref{add-reg-ftetas}, as a result of a
\bec
further regularity estimate on the equation for $\tetas$. \eec
 Such an estimate also implies \eqref{add-reg-chi}--\eqref{add-reg-zeta}.
\begin{theorem}
\label{thm:main} Assume
\eqref{assumpt-domain} and
 Hypotheses \ref{hyp:1}--\ref{hyp:6}. Suppose
that the data $(h,\mathbf{f},\mathbf{g})$ and
$(\teta_0,\tetas^0,\uu_0,\chi_0)$ fulfill \eqref{hyp-data} and
\eqref{hyp-initial}.

Then, Problem \ref{prob:irrev} admits \bec an \emph{energy solution} \eec  $(\vartheta,
\vartheta_s,   \uu,\chi,\eeta,\mmu, \xi, \zeta)$, which in addition
satisfies
\begin{subequations}
\label{additional-regularities}
\begin{align}
& \label{add-reg-teta} \teta \in \mathrm{BV}([0,T];W^{1,q}(\Omega)')
\text{ for every $q>3$,}
\\
& \label{add-reg-tetas}
 \tetas \in H^1 (0,T;\Hc),
\\
 & \label{add-reg-ftetas}
 \ftetas(\vartheta_s) \in   L^\infty (0,T; \Vc),
\\ &
\label{add-reg-chi} \chi \in L^{\infty}(0,T;H^2 (\Gammac))  \cap
H^{1}(0,T;\Vc) \cap W^{1,\infty} (0,T;\Hc),
\\
& \label{add-reg-xi}
 \xi\in L^\infty(0,T; \Hc)\,,
\\
& \label{add-reg-zeta}
 \zeta\in L^\infty(0,T; \Hc)\,.
\end{align}
\end{subequations}
\end{theorem}
\paragraph{\bf Outline of the proof.}
We set up a suitable approximation of system \eqref{system-weak} by regularizing
some of the (maximal monotone) operators featured therein;
we shall  denote the regularization parameter with the symbol $\eps$ and accordingly refer to the approximate problem as $(P_\eps)$.
In Section \ref{s:approximation} we prove the existence of
local-in-time
solutions to Problem  $(P_\eps)$ (cf.\ Proposition \ref{prop:loc-exist-eps} ahead).
Then, we
 show that the approximate solutions fulfill  an energy identity, which serves as the basis for deriving a series of
 a priori estimates, \emph{uniform} w.r.t.\ $\eps$. Relying on them,  in Section\ \ref{s:5} we prove that, along a suitable
 subsequence, the approximate solutions converge to a \emph{local-in-time} solution to Problem \ref{prob:irrev}.
 Its extension to a \emph{global-in-time} solution by means of a careful  prolongation argument is the
 focus of Section \ref{s:6}. {\ele Some useful technical lemma we \bec will \eec use in the proofs are stated and proved in  \bec the Appendix. \eec}


\section{\bf Approximation}
\label{s:approximation}
First,    in Sec.\ \ref{ss:3.1} we introduce and 
\bec explain \eec  our approximation of system \eqref{system-weak}, leading to
 Problem  $(P_\eps)$; in the end, we state
its variational formulation.
The existence of a local-in-time solution to $(P_\eps)$
   is proved in Sec.\ \ref{ss:3.3} via
the Schauder fixed point theorem.
 Most of the calculations for the (uniform w.r.t.\ $\eps$) a priori estimates on the approximate solutions
  which we shall derive in Sec.\ \ref{s:5}
 hinge on  a series of technical results on the functions approximating the nonlinearities of the problem,
   which we have collected {\ele in the Appendix.} 

\subsection{The approximate problem}
\label{ss:3.1}
\bec To
motivate  \eec  the regularization procedures
 we are going to
adopt, \bec we discuss \eec  in advance  some of the a priori estimates we
shall perform  on system \eqref{system-weak} in Sec.\
\ref{s:5}. As we will see, the related calculations cannot be
performed rigorously on system \eqref{system-weak}, and indeed
necessitate of the Yosida-type regularizations by which we are going
to replace some of the maximal monotone nonlinearities in system
\eqref{system-weak}.

\paragraph{\bf Outlook to the approximate problem.} The
basic \emph{energy estimate} for system \eqref{system-weak}   (cf.\
the \emph{First a priori estimate} in Sec.\ \ref{s:4})
follows by testing \eqref{teta-weak} by $\vartheta$,
\eqref{teta-s-weak} by $\vartheta_s$, \eqref{eqIa} by
$\partial_t\uu$,  \eqref{eqIIa-irr} by $\partial_t\chi$, adding the
resulting relations, and integrating in time. The \emph{formal}
calculations
\begin{equation}
\label{e:formal1}
\begin{aligned}
 \int_0^t \pairing{}{V}{\partial_t \lteta(\vartheta)}{ \vartheta } \dd
r  &  = \int_0^t \pairing{}{V}{\partial_t w }{ \lteta^{-1}(w)} \dd r
\\ &
 =
 \| J^*(\lteta(\teta(t))) \|_{L^1 (\Omega)} - \| J^*(\lteta(\vartheta_0)) \|_{L^1
(\Omega)} \geq  C_1 \| \teta(t) \|_{L^1 (\Omega)} -C
\end{aligned}
\end{equation}
where we have  used the auxiliary variable $w:= \lteta(\teta)$, the
\emph{formal identity}
\[
\pairing{}{V}{\partial_t w }{ \lteta^{-1}(w)}  =  \frac{\dd}{\dd t }
\left(  \int_\Omega J^* (w ) \dd x \right) \quad \aein \, (0,T),
\]
and, finally, the coercivity condition \eqref{cond-L3}, lead to a
bound for $\teta$ in $L^\infty (0,T;L^1(\Omega))$. The corresponding
calculations on the level of \eqref{teta-s-weak} yield an estimate
for
  $\teta_s$ in $L^\infty (0,T;L^1(\Gammac))$.

In order to  make \eqref{e:formal1} rigorous, following \cite{bcfg1,
bbr4,bbr6}
\begin{enumerate}
\item we replace $\lteta$ and $\ltetas$  in the equations
\eqref{teta-weak} and \eqref{teta-s-weak} by the following
approximating functions
\begin{align}
&
\label{def-applogteta} \applogteta(r):=\varepsilon
r+\lteta_\varepsilon (r),
\\
&\label{def-applogteta-s} \applogtetas(r):=\varepsilon
r+\ltetas_\varepsilon( r),
\end{align}
where for $\eps>0$  $\lteta_\varepsilon$ and $\ltetas_\varepsilon$
denote the Yosida regularizations of $\lteta$ and $\ltetas$,
respectively, cf.\ \eqref{e:max-mon-2} below.
  \end{enumerate}
Therefore,
 $\applogteta$ ($\applogtetas$, respectively) is differentiable, strictly increasing  and Lipschitz continuous, see also
the upcoming
  Lemma \ref{l:new-lemma2}.
  Nonetheless, this procedure only partially serves to the purpose of rigorously justifying \eqref{e:formal1}, as expounded in Remark
  \ref{rmk:rigour} at the end of Section \ref{s:4}.

In accord with
 \eqref{def-f} and \eqref{def-applogteta-s},
 \begin{enumerate}
 \setcounter{enumi}{1}
\item we thus replace the function $\ftetas$ in \eqref{teta-s-weak} by
\begin{equation}
\label{def-appf} \ftetas_\eps (x)=\int_0^{x}
\frac{1}{\applogtetas'(s)} \dd s\,,\quad \forall x\in \R,
\end{equation}
whose definition reflects the fact that $f(x) = \int_0^x \frac{1}{\ltetas'(s)} \dd s$.
\end{enumerate}

Combining the  aforementioned \emph{energy estimate} and  a
comparison argument in  the momentum equation \eqref{eqIa} leads to
the following estimate
\begin{equation}
\label{comparison-1} \| \fc(\teta-\teta_s)\mmu + \eeta \|_{L^2
(0,T;\bsY_{{\Gammac}}')} \leq C
\end{equation}
with $\eeta\in
\partial \bvarphi(\uu) $ in $\bsY_{{\Gammac}}'$
 a.e.\ in
$(0,T)$ (cf.\  \eqref{incl1}), and $\mmu \in |\Reg(\eeta)|\partial
\Psi(\dotu) $ a.e.\ in ${\Gammac}\times (0,T)$ (cf.\
\eqref{incl1-bis}). From \eqref{comparison-1}, it is crucial to
conclude  the \emph{separate} estimates
\begin{equation}
\label{second-aprio}
 \| \fc(\teta-\teta_s)\mmu\|_{L^2
(0,T;\bsY_{{\Gammac}}')} + \| \eeta \|_{L^2 (0,T;\bsY_{{\Gammac}}')}
\leq C.
\end{equation}
This  follows from the orthogonality condition
\eqref{ass:orthogonal} only on a \emph{formal} level, since
\eqref{ass:orthogonal}  is not, in general, inherited by the
abstract operator
 $\partial \bvarphi: \bsY_{{\Gammac}} \rightrightarrows \bsY_{{\Gammac}}' $ .
In order to justify this argument,
 we need to
 suitably approximate
  $\partial\bvarphi: \bsY_{{\Gammac}} \rightrightarrows
  \bsY_{{\Gammac}}'$
  in such a way as to replace $\eeta \in
  \bsY_{{\Gammac}}'$ in \eqref{eqIa} with a term $\eeta_\eps$
orthogonal to $\partial \Psi(\dotu) $. Along the lines of \cite{bbr5,bbr6}, we
\begin{enumerate}
 \setcounter{enumi}{2}
\item   approximate  the function $\Phi$ \ref{hyp:phi}, which defines the functional $\bvarphi$ through
\eqref{funct-phi},
 by its Yosida
approximation $\Phi_\eps: \R^3 \to [0,+\infty)$.
 \end{enumerate}
 We recall that $\Phi_\eps $ is convex, differentiable, and  such that $\mathrm{D}\Phi_\eps$  is
 the Yosida regularization of the subdifferential $\partial\Phi: \R^3 \rightrightarrows \R^3$.
 As such, it fulfills (cf.\ \eqref{e:max-mon-3-bis} below)
 \begin{equation}
 \label{to-be-cited-orthog}
 \mathrm{D}\Phi_\eps(\uu) \in \partial \Phi( \mathsf{r}_\eps (\uu)),
 \end{equation}
 where $\mathsf{r}_\eps$ denotes the resolvent of the operator $\partial \Phi$.
 Therefore, in view of  \eqref{ass:orthogonal}, any approximate solution $\uue$ satisfies
 \begin{equation}
 \label{crucial-orthogonality}
  \mathrm{D}\Phi_\eps(\uue)  \cdot \zz =0 \qquad \text{for all }
  \zz \in \partial \Psi\ber (\partial_t\uue)\edr,
   \end{equation}
 which will be crucial in order to deduce \eqref{second-aprio}, cf.\ the \emph{Third a priori estimate} in Sec.\ \ref{s:4}.

Finally, \bec along on the lines of \cite{bfl} \eec we will  also perform on the \emph{doubly nonlinear}
equation \eqref{eqIIa-irr} the test by  (the formal quantity)
$\partial_t (A\chi +\beta(\chi))$. Let us
 mention that such  an estimate is by now classical for {\ele this kind of doubly nonlinear  diffusive } evolutionary differential inclusions.
  It allows one to
 estimate the terms $A\chi$
 and $\xi \in \beta(\chi)$,  \emph{separately},  in $L^\infty (0,T; L^2(\gc))$. {\ele \bec Let us stress that \eec
  this estimate requires ad hoc  time-regularity for the
  \bec right-hand side terms. Once the computations have been carried out,  \eec} an estimate for $\zeta \in \rho(\partial_t
 \chi)$ in  $L^2  (0,T; L^2(\gc))$ \bec then \eec  follows from a comparison in \eqref{eqIIa-irr}. In order to perform all the calculations in a rigorous way (cf.\ the \emph{Seventh a priori estimate} in Sec.\
 \ref{s:4}), it is necessary to
 \begin{enumerate}
 \setcounter{enumi}{3}
\item
 replace $\rho$ and $\beta$ by their Yosida
approximations $\rho_\eps$ and $\beta_\eps$.
 \end{enumerate}



Furthermore, it will be convenient to \bec use \eec the functions
$\mathcal{I}_\eps, {\it i}_\eps: \R \to \R$
 \begin{align}
 &
\label{mathcal-i-eps} \mathcal{I}_\eps(x):= \int_0^x s\,
\applogteta'(s) \dd s
\\
&
\label{ipiccolo-eps} {\it i}_\eps(x):= \int_0^x s\, \applogtetas'(s) \dd s
\end{align}
(cf.\ in particular the derivation of the approximate energy identity \eqref{enid0} later on),
and  the function
$H_\eps: \R \to \R$
\begin{equation}
\label{def-accaapp} H_\eps(x):=\int_0^x \applogtetas'(s)\,
\ftetas_\eps(s) \dd s \quad\forall x\in \R
\end{equation}
(cf. \ the derivation of \emph{Sixth and Seventh a priori estimate} in Sec.\ \ref{s:4}).

\ber Finally, we will supplement our approximate Problem $(P_\eps)$ by \edr the initial data $(\vartheta_0^\eps,\vartheta_s^{0,\eps},
\uu_0,\chi_0)$, where the family
\begin{equation}
\label{datitetaapp}
 (\vartheta_0^\eps, \vartheta_s^{0,\eps})_\eps\in H\times H_{\Gammac}
\end{equation}
 approximates the data
 $(\teta_0,\teta_s^0)$  from \eqref{cond-teta-zero}--\eqref{cond-teta-esse-zero}
in the sense that
\begin{align}
&\label{convdatiteta}
(\vartheta_0^\eps, \vartheta_s^{0,\eps}) \to (\teta_0,\teta_s^0) \text{ in } L^1(\Omega)\times H_{\Gammac}
\text{ as $\eps\down0$\,,}
\\
&\label{bounddatiteta-bis}
\ber \exists\, {\bar S_0}>0 \ \text {depending on } \teta_0  \text { and } \teta_s^0 \,:\ \
 \|\lteta_\eps(\vartheta_0^\eps)\|_H + \|\ltetas_\eps(\vartheta_s^{0,\eps})\|_{\Hc}\leq {\bar S_0}\edr \quad \text{for all} \, \eps>0\,,
 \\
 &
  \label{bounddatiteta-1lemma}
\berc \bec \int_\Omega \mathcal{I}_\eps(\teta_0^\eps)\dd x \to \int_\Omega J^*(L(\teta_0)) \dd x \qquad \text{as } \eps \downarrow 0\,, \eec \eerc
\\
&
\label{bounddatiteta-2lemma}
\berc \bec \int_{\Gammac}   i_\eps(\teta_s^{0,\eps})\dd x \to   \int_{\Gammac} j^*(\ltetas(\teta_s^0)) \dd x  \qquad \text{as } \eps \downarrow 0\,.  \eec \eerc
\end{align}
\berc Observe that  \eqref{bounddatiteta-1lemma} and \eqref{bounddatiteta-2lemma} guarantee that
 \begin{align}
&
\label{bounddatiteta-1}
\exists\, {\bar S_1}>0 \ \text {depending on } \teta_0\,: \ \ \forall\, \eps \in (0,1) \quad
\int_\Omega \mathcal{I}_\eps(\teta_0^\eps)\dd x \leq {\bar S_1}\,,
\\
&
\label{bounddatiteta-2}
\exists\, {\bar S_2}>0 \ \text {depending on } \teta_s^0\,: \ \ \forall\, \eps \in (0,1) \quad
\quad \int_\gc i_\eps(\teta_s^{0,\eps})\dd x \leq {\bar S_2}\,.
\end{align}
Finally, we also require that \eerc the family $(\vartheta_s^{0,\eps})_\eps$ fulfills
\begin{align}
&
\label{bounddatiteta-3}
\exists\, {\bar S_3}>0\ \text {depending on } \teta_s^0\,: \ \ \forall\, \eps \in (0,1) \quad
\|H_\eps(\vartheta_s^{0,\eps})\|_{L^1(\Gamma_c)} \leq {\bar S_3}\,,
\\
&
\label{bounddatiteta-4}
\exists\, {\bar S_4}>0\ \text {depending on } \teta_s^0\,: \ \ \forall\, \eps \in (0,1) \quad
\quad \| \ftetas_\eps(\vartheta_s^{0,\eps})\|_{\Vc} \leq {\bar S_4}.
\end{align}
\bec Observe that, since $f_\eps$ is bi-Lipschitz (cf.\ Lemma  \ref{l:new-lemma3} later on), \eqref{bounddatiteta-4} in fact
implies that $\vartheta_s^{0,\eps}$ is also in $\Vc$. \eec
{\ele In the Appendix 
we state \bec a series of Lemmas, \eec  in which } we will  construct sequences of initial data
$(\vartheta_0^\eps, \vartheta_s^{0,\eps})_\eps$ complying with the properties \eqref{datitetaapp}--\eqref{bounddatiteta-4} and
 detail
how the constants $\bar S_i\,, \ i=1,\cdots, 4$ \ber  may \edr depend on the data $\teta_0$ and $\teta_s^0$.

\smallskip
\noindent
All in all, the variational formulation of the approximate problem
reads: 
\begin{problem}[$P_\eps$]
 \label{prob:irrev-app}
 \upshape
Let {\ele a quadruple} of initial data
 $(\vartheta_0^\eps, \vartheta_s^{0,\eps} , \uu_0, \chi_0)$
   {\ele satisfy} \eqref{cond-uu-zero}--\eqref{cond-chi-zero} and \eqref{datitetaapp}--\eqref{bounddatiteta-4}.
Find
 a
 quintuple
$(\vartheta,  \vartheta_s,   \uu,\chi,\mmu)$ fulfilling
\begin{align}
& \label{reg-teta-app} \teta \in  L^2 (0,T;\V) \cap \mathrm{C}^0
([0,T];H),
\\
& \label{reg-log-teta-app} \applogteta(\teta) \in L^2 (0,T;\V) \cap
\mathrm{C}^0 ([0,T];H) \cap H^1 (0,T;\V'),
\\
& \label{reg-teta-s-app} \teta_s \in  L^\infty (0,T;\Vc) \cap H^1
(0,T;\Hc),
\\
&
 f_\eps(\teta_s) \in L^\infty (0,T;\Vc), 
\\
& \label{reg-log-teta-s-app} \applogtetas(\teta_s) \in L^\infty
(0,T;\Vc) \cap H^1 (0,T;\Hc),
\end{align}
and such that $(\uu,\chi,\mmu)$ comply with \eqref{reguI},
\bec \eqref{add-reg-chi}, \eec
  and \eqref{mureg}, satisfying  the initial
 conditions
 \begin{align}
& \label{iniw-better} \vartheta(0)=\vartheta_0^\eps \quad \aein \
\Omega\,,
\\
& \label{iniz-better} \vartheta_s(0)=\vartheta_s^{0,\eps} \quad
\aein \ \Gammac\,,
\end{align}
 as well as  \eqref{iniu}--\eqref{inichi},
  and the equations
  \begin{subequations}
  \label{var-formu-abstract}
 \begin{align}
&
 \label{teta-weak-app}
\begin{aligned}
 &  \pairing{}{\V}{\partial_t\applogteta(\teta)}{ v}
   -\int_{\Omega} \dive(\partial_t\mathbf{u}) \, v  \dd x
+\int_{\Omega}   \nabla g(\vartheta) \, \nabla v  \dd x
\\ & + \int_{\Gammac}
k(\chi)  (\vartheta-\vartheta_s) v   \dd x+\int_{\Gammac}
\fc'(\teta-\teta_s)  \Psi(\dotu) |\Reg (\mathrm{D}\Phi_\eps(\uu))| v
\dd x
 = \pairing{}{V}{h}{v} \quad
\forall\, v \in V \ \hbox{ a.e.\ in }\, (0,T)\,,
\end{aligned}
\\
& \label{teta-s-weak-app}
\begin{aligned}
 &  \pairing{}{\Vc}{\partial_t\applogtetas(\teta_s)}{ v}  -\int_{\Gammac}
\partial_t \lambda(\chi) \, v   \dd x     +\int_{\Gammac} \nabla  f_\eps(\vartheta_s)  \, \nabla
v \dd x
\\ &
 = \int_{\Gammac} k(\chi) (\vartheta-\vartheta_s) v\dd x
 +\int_{\Gammac} \fc'(\teta-\teta_s)   \Psi(\dotu) |\Reg (\mathrm{D}\Phi_\eps(\uu)) | v \dd x
  \quad
\forall\, v \in \Vc  \ \hbox{ a.e.\ in }\, (0,T)\,,
\end{aligned}
\\
 &
 \label{eqIa-app}
\begin{aligned}
 &  b(\dotu,\vv)  +a(\uu,\vv)+ \int_{\Omega} \vartheta \dive (\vv)\dd x
 +\int_{\Gammac} \chi \uu \cdot \vv \dd x\\
& + \int_\gc  \mathrm{D}\Phi_\eps(\uu) \cdot  \vv\dd x
   +
\int_{{{\Gammac}}}\fc(\teta-\vartheta_s)   {\mmu}\cdot {\vv}\dd x =
\pairing{}{\bsW}{\mathbf{F}}{\vv}
 \quad \text{for all } \vv\in \bsW \ \hbox{ a.e.\ in }\, (0,T)\,,
  \end{aligned}
 \\
 &
 \label{to-be-quoted-also}
 \mmu = |\Reg (\mathrm{D}\Phi_\eps(\uu)) |  \zz  \quad \text{with} \quad \zz\in \partial \Psi(\dotu)\hbox { a.e.\ in }\,
{\Gammac}\times (0,T),
\\
&
\partial_t{\chi}+\rho_\eps(\partial_t \chi )+A\chi+\beta_\eps (\chi)+{\ele \gamma}'(\chi)= -\lambda'(\chi)
\vartheta_s- \frac12 |\uu|^2   \quad\hbox{a.e.\ in }
{\Gammac}\times(0,T)\,.  \label{eqIIa-irr-app}
\end{align}
\end{subequations}
\end{problem}
Let us only briefly comment on the enhanced regularity
properties
\eqref{reg-teta-app}--\eqref{reg-log-teta-s-app}. Since, for $\eps>0$ fixed,
$\applogteta$ is bi-Lipschitz (cf.\ Lemma \ref{l:new-lemma2}),
$\teta \in L^2 (0,T; V)$ implies  that
$\applogteta(\teta)$ is in the same space. Therefore, by interpolation with
$H^1 (0,T; V')$ we conclude that
$\applogteta (\teta) \in \rmC^0 ([0,T]; H)$, whence $\teta \in \rmC^0 ([0,T]; H)$.
Analogous arguments \bec apply to \eec
\eqref{reg-teta-s-app}--\eqref{reg-log-teta-s-app}. Observe that $\teta_s$
\bec also \eec
inherits the regularity
of $\ftetas_\eps (\teta_s)$, since  $f_\eps$ is  bi-Lipschitz as well, cf.\ Lemma \ref{l:new-lemma3}.

\subsection{Local existence of approximate solutions}
\label{ss:3.3}

 The main result of this section is the forthcoming
Proposition \ref{prop:loc-exist-eps}, stating the
 existence of a local-in-time solution to Problem $(P_\eps)$.
The latter features  a  structure   very similar to the \bec
approximate problem for the \eec
{\ele  PDE system  analyzed} in \cite{bbr6}, cf.\ Problem 4.4
therein. Indeed, {\ele some of } the arguments from \cite{bbr6} may
be easily adapted to the present setting. Therefore,
 we only sketch the fixed point procedure  yielding
 local existence. In particular, we
only hint to the most relevant steps in the construction of the
fixed point operator and in the proof of its continuity and
compactness, referring to   \cite[Sect.
4.2]{bbr6} for all details.

\paragraph{\bf Fixed point setup.}
In view of hypothesis \eqref{hyp-r-1} on the regularizing operator
$\Reg$,
we may choose $\delta \in (0,1)$ such that
\begin{equation}\label{hyp-r-nuova}
\Reg:L^2(0,T;\bsY'_{\Gammac})\rightarrow L^\infty(0,T;L^{\frac
2{1-\delta}}(\Gammac;\R^3)) \text{ is weakly-strongly continuous}
\end{equation}
(and therefore bounded). For a fixed $\tau>0$ and a fixed constant
$M>0$, we consider the set
\begin{equation}
\label{fixed-point-set}
\begin{aligned}
\! \!\!\!\! \!\!
 \!{\calY}_\tau=\{ & (\teta,
\teta_s,\uu,\chi)\in L^2(0,\tau;H^{1-\delta}(\Omega))\times
L^2(0,\tau;H^{1-\delta}(\Gammac)) \times
L^2(0,\tau;H^{1-\delta}(\Omega;\R^3))
\times L^2(0,\tau;\Hc)\,: \\
&
\|\teta\|_{L^2(0,\tau;H^{1-\delta}(\Omega))}+\|\teta_s\|_{L^2(0,\tau;H^{1-\delta}(\Gammac))}
+ \|\uu\|_{L^2(0,\tau;H^{1-\delta}(\Omega;\R^3))}+
\|\chi\|_{L^2(0,\tau;\Hc)} \leq M\},
\end{aligned}
\end{equation}
 with the topology induced by $L^2(0,\tau;H^{1-\delta}(\Omega))\times L^2(0,\tau;H^{1-\delta}(\Gammac))
\times L^2(0,\tau;H^{1-\delta}(\Omega;\R^3)) \times
L^2(0,\tau;\Hc)$. We are going to construct an operator
$\mathcal{T}$ mapping ${\calY}_{\widehat{T}}$ into itself for a
suitable time $0\leq \widehat{T} \leq T$, depending on $M$,  in such a way that any
fixed point of $\mathcal{T}$ yields a solution to Problem $(P_\eps)$
on the interval $(0,\widehat{T})$.
\begin{notation}
\label{not-4.2} \upshape
 In the following lines,
we will denote by $S_i$, $i=1,...,5$, {\ele positive constants} depending
on the problem data and on
$M>0$ in \eqref{fixed-point-set}, but \emph{independent} of
$\eps>0$,
 and by $S_6(\eps)$ a
constant depending on the above quantities and on $\eps>0$ as well.
Furthermore,  with the symbols $\pi_i (A)$, $\pi_{i,j}(A), \ldots,$
we will denote the projection of a set $A$ on its $i$-, or
$(i,j)$-component.
\end{notation}

\paragraph{\bf Step $1$:}
As a first step in the construction of $\mathcal{T}$, we fix
$(\widehat\teta,\widehat{\teta}_s,\widehat{\uu},\widehat\chi) \in
\mathcal{Y}_\tau$ and consider
 (the Cauchy problem for)  the
  system
(\ref{eqIa-app}--\ref{to-be-quoted-also}), with
$(\widehat\teta,\widehat\teta_s,\widehat{\chi})$ in place of
$(\teta,\teta_s,\chi)$, and \bec $
\fc(\widehat{\teta}-\widehat{\vartheta}_s)
|\Reg(\mathrm{D}\Phi_\eps(\widehat{\uu}))|$   replacing  $
\fc(\teta-\vartheta_s) |\Reg(\mathrm{D}\Phi_\eps(\uu))$, \eec that
is
\begin{equation}
\label{eqS1}
\begin{aligned}
 &  b(\dotu,\vv)  +a(\uu,\vv)+ \int_{\Omega} \widehat\teta \dive (\vv)\dd x
 +\int_{\Gammac} \widehat\chi \uu \cdot \vv \dd x\\
& + \int_\gc \mathrm{D}\Phi_\eps(\uu)\cdot  \vv\dd x
   +
\int_{{{\Gammac}}} \fc(\widehat{\teta}-\widehat{\vartheta}_s)  \mmu
\cdot {\vv}\dd x = \pairing{}{\bsW}{\mathbf{F}}{\vv}
 \quad \text{for all } \vv\in \bsW \ \hbox{ a.e.\ in }\, (0,T)\,,
 \\
 &  \mmu = |\Reg(\mathrm{D}\Phi_\eps(\widehat{\uu}))|
  {\zz}  \ \text{ with } \ \zz\in \partial \Psi(\bec \partial_t \uu \eec) \hbox { a.e.\ in }\,
{\Gammac}\times (0,T)\,.
 \end{aligned}
 \end{equation}
A well-posedness result for such a problem can be obtained easily
adapting the arguments of the proof of \cite[Lemma\ 4.6]{bbr6},
observing that \eqref{eqS1} has the very same structure of the
corresponding momentum equation tackled in \cite {bbr6} (cf.\
Hypothesis 2.3 on the subdifferential operators). Then, there exists
a constant $S_1>0$ and a unique  pair $(\uu,\mmu) \in
H^1(0,\tau;{\bf W}) \times  L^\infty (0,\tau;L^{2+\nu}(\gc;\R^3))$
fulfilling the initial condition \eqref{iniu},  equation
\eqref{eqS1} and the estimate
\begin{equation}\label{boundS1}
\|\uu\|_{H^1(0,\tau;{\bf W})}+ \| \zz \|_{L^\infty (\gc \times
(0,\tau))} + \| \mmu \|_{L^\infty (0,\tau;L^{2+\nu}(\gc;\R^3)) }
\leq S_1.
\end{equation}

For later convenience, let us detail the proof of the estimate
for $\| \uu \|_{H^1(0,\tau;\WW)}$; the estimate for $\zz$ simply derives from
\eqref{bounded-operator}, while the bound for $\mmu$ follows from the calculations
for the forthcoming \emph{Third a priori estimate}, cf.\ Sec.\ \ref{s:4}.
We choose $\vv=\partial_t\uu$ in \eqref{eqS1} and
integrate in time over $(0,t)$. In particular, we exploit the ellipticity properties
\eqref{korn_a} and integrate by parts. It follows from
the H\"older inequality
\begin{align}\label{eqstimaS1-1}
 &  C_b\int_0^t\|\partial_t\uu\|^2_\WW \dd s+\frac {C_a}2\|\uu(t)\|^2_\WW+\int_\gc\Phi_\eps(\uu(t)) \dd x \\
 \no
 & 
 \begin{aligned}
  \leq
  c (\| \uu_0 \|_{\WW}^2 + \bvarphi(\uu_0))
 +
  \int_0^t\|\widehat\teta\|_H\|\hbox{div }\partial_t\uu\|_H \dd s &  +\int_0^t\|\widehat\chi\|_{\Hc}\|\uu\|_{L^4(\gc)}\|\partial_t\uu\|_{L^4(\gc)} \dd s
  \\ & +
  \int_0^t\|\FF\|_{\WW'}\|\partial_t\uu\|_\WW \dd s+c.
  \end{aligned}
\end{align}
Note that here we have exploited the fact that
\begin{equation}
\int_{{{\Gammac}}} \fc(\widehat{\teta}-\widehat{\vartheta}_s)  \mmu
\cdot {\partial_t\uu}\dd x\geq0
\end{equation}
in view of \eqref{hyp-fc}.
Hence, the right-hand side of \eqref{eqstimaS1-1} can be handled by
Young's inequality combined with trace theorems and Sobolev
embeddings, \bec which give \eec
\begin{align}\label{eqstimaS1-2}
&\int_0^t\|\widehat\teta\|_H\|\hbox{div
}\partial_t\uu\|_H \dd s
+\int_0^t\|\widehat\chi\|_{\Hc}\|\uu\|_{L^4(\gc)}\|\partial_t\uu\|_{L^4(\gc)} \dd s+\int_0^t\|\FF\|_{\WW'}\|\partial_t\uu\|_\WW \dd s+c\\\no
&\leq
\frac{C_b}2\int_0^t\|\partial_t\uu\|^2_\WW \dd s+c\left(\|\widehat\theta\|^2_{L^2(0,\tau;H)}+\|\FF\|^2_{L^2(0,\tau;\WW')}
+\int_0^t\|\widehat\chi\|^2_{\Hc}\|\uu\|^2_\WW \dd s \right).
\end{align}
Exploiting  the Gronwall lemma, recalling \eqref{effegrande} \bec
for $\mathbf{F} $ \eec  and \eqref{fixed-point-set} (in particular
\bec that \eec $\|\widehat\chi\|^2_{\Hc}$ belongs to $L^1(0,\tau)$),
we infer that $\uu$ is bounded in $H^1(0,\tau;\WW)$ by some constant
$S_1$
 depending on the data of the
problem (and on $M$) but not on $\eps$. As a consequence, we may
define an operator
 \begin{equation}
 \label{ope-1}
 \mathcal{T}_1: \mathcal{Y}_\tau \to \mathcal{U}_\tau:= \{\uu \in H^1 (0,\tau;\bsW)\, : \ \|\uu\|_{H^1(0,\tau;{\bf W})}\leq S_1 \}
 \end{equation}
 which maps every quadruple $(\widehat\teta,\widehat{\teta}_s,\widehat{\uu},\widehat\chi) \in \mathcal{Y}_\tau$
 into the unique solution $\uu$
of the Cauchy problem  for \eqref{eqS1} \bec (with associated
   $\mmu\in |\Reg(\mathrm{D}\Phi_\eps(\widehat{\uu}))| \partial \Psi(\dotut)
   $). \eec

\paragraph{\bf Step $2$:}
As a second step, we consider (the Cauchy problem for)
 \eqref{eqIIa-irr-app}, with   $\widehat\teta_s \in \pi_2(\mathcal{Y}_\tau) $ and $\uu=\mathcal{T}_1(\widehat\teta,\widehat{\teta}_s,\widehat{\uu},\widehat\chi) $ in  $ \mathcal{U}_\tau$
on the right-hand side, that is
\begin{equation}
\label{eqS1/2}
\partial_t{\chi}+\rho_\eps(\partial_t \chi )+A\chi+\beta_\eps (\chi)+{\ele \gamma}'(\chi)=-\lambda'(\chi)
\widehat{\vartheta}_s- \frac12 |\uu|^2 \quad\hbox{a.e.\ in }
{\Gammac}\times
(0,T)\,.\\
 \end{equation}
Standard results in the theory of parabolic equations (recall the
Lipschitz continuity of $\beta_\epsi$ and $\rho_\epsi$, as $\epsi$
is fixed) ensure that  there exists a constant $S_2>0$ (depending on
$M$ via $S_1$),
 and a unique
function $\chi \in {\ele L^2} (0,\tau; H^2(\gc)) \cap L^\infty
(0,\tau; \Vc) \cap  H^1(0,\tau;\hc)$, fulfilling the initial
condition \eqref{inichi},  equation \eqref{eqS1/2} and
\begin{equation}\label{boundS1/2}
\|\chi\|_{L^2 (0,\tau; H^2(\gc)) \cap L^\infty (0,\tau; \Vc) \cap
H^1(0,\tau;\hc)}  \leq S_2.
\end{equation}
\noindent It follows that we may define an operator
 \begin{equation}
 \label{ope-2}
 \begin{aligned}
 \mathcal{T}_2: \pi_2(\mathcal{Y}_\tau) \times \mathcal{U}_\tau  \to \mathcal{X}_\tau:=
 \{ & \chi \in
    L^2 (0,\tau; H^2(\gc)) \cap L^\infty (0,\tau; H^1(\gc)) \cap  H^1(0,\tau;\hc)\, :
    \\
    &
    \ \|\chi\|_{L^2 (0,\tau; H^2(\gc))\cap L^\infty (0,\tau; H^1(\gc)) \cap  H^1(0,\tau;\hc)} \leq S_2
   \}
    \end{aligned}
 \end{equation}
 mapping  $(\widehat{\teta}_s,\uu) \in \pi_2(\mathcal{Y}_\tau) \times \mathcal{U}_\tau$
 into the unique solution $\chi$ of the Cauchy problem  for \eqref{eqS1/2}.

\paragraph{\bf Step $3$:}
Finally, we consider the Cauchy problem for the system
(\ref{teta-weak-app}, \ref{teta-s-weak-app}) with fixed
  $(\widehat
\teta,\widehat\teta_s) \in \pi_{1,2}(\mathcal{Y}_\tau)$ and
 $\uu=\mathcal{T}_1(\widehat\teta,\widehat{\teta}_s,\widehat{\uu},\widehat\chi), $
 $\chi=\mathcal{T}_2(\widehat{\teta}_s,\mathcal{T}_1(\widehat\teta,\widehat{\teta}_s,\widehat{\uu},\widehat\chi))$ from the previous steps. In particular,
 we set
 \begin{equation}
 \label{not-mathcalF}
 \widehat{\mathcal{F}}:= k(\chi)  (\widehat\teta-\widehat\teta_s)+\fc'(\widehat\teta-\widehat\teta_s) |\Reg (\mathrm{D}\Phi_\eps(\uu))| \Psi(\uu_t)
 \end{equation}
and plug it into the boundary integral on the left-hand side of
 \eqref{teta-weak-app} and on the right-hand side of \eqref{teta-s-weak-app}.
 Observe that, due to \eqref{fixed-point-set}, \eqref{boundS1}, \eqref{boundS1/2},
  to the Lipschitz continuity of $\fc$ and $k$, {\ele  as well as \eqref{hyp-fc} and Sobolev embeddings, and trace theorems, }
 there holds
 \begin{equation}
 \label{est-math-cal-F}
 \widehat{\mathcal{F}} \in L^2 (0,\tau;L^{4/3+s}(\gc)) \quad \text{for some } s=s(\delta)>0.
 \end{equation}
Now
we consider the Cauchy problem for the system
\begin{align}
& \label{eqapptheta1}
\begin{aligned}
 &  \pairing{}{\V}{\partial_t\applogteta(\teta)}{ v}
   -\int_{\Omega} \dive(\partial_t\mathbf{u}) \, v  \dd x
+\int_{\Omega}   \nabla g(\vartheta) \, \nabla v  \dd x +
\int_{\Gammac} \widehat{\mathcal{F}}  v  \dd x
 = \pairing{}{V}{h}{v} \quad
\forall\, v \in V\,,
\end{aligned}
\\
& \label{eqappthetas1}
\begin{aligned}
   \pairing{}{\Vc}{\partial_t \applogtetas(\teta_s)}{ v }  -\int_{\Gammac}
\partial_t \lambda(\chi) \, v   \dd x     +\int_{\Gammac} \nabla f_\eps(\vartheta_s)  \, \nabla
v \dd x
 = \int_{\Gammac} \widehat{\mathcal{F}} v \dd x
  \quad
\forall\, v \in \Vc\,,
\end{aligned}
\end{align}
a.e.\ in $(0,T)$. Observe that system (\ref{eqapptheta1},
\ref{eqappthetas1}) is decoupled, hence we will handle equations
\eqref{eqapptheta1} and \eqref{eqappthetas1} separately.

 The
well-posedness for  the Cauchy problem for the doubly nonlinear
equation \eqref{eqapptheta1} follows from standard results (cf.\ e.g.\ \cite[Chap.\ 3, Thm.\ 4,1]{visintin}), taking into account that both
$\applogteta$ and $g$ are bi-Lipschitz.
 We only sketch here the uniqueness proof for
\eqref{eqapptheta1}. We subtract the equation for $\teta_2$ from the
equation for $\teta_1$ and integrate on  $(0,t)$, with $0 \leq t
\leq\tau$. Hence, we choose $v=g(\teta_1)-g(\teta_2)$ as test
function and integrate again in time. Using that
$\applogteta(r) = \eps r + \lteta_\eps (r)$ and
  that the
functions $\lteta_\epsi$ and $g$ are strictly increasing (cf.\ also
\eqref{cond-L1}--\eqref{cond-L2}, \eqref{cond-g}), we obtain
\begin{align}
\label{u:1} & c_3 \, \eps\|\teta_1-\teta_2\|^2_{L^2(0,\tau;\hc)}
\leq \eps \int_0^t\!\!\int_\Omega (\teta_1-\teta_2)(g(\teta_1)-g(\teta_2)) \dd x \dd s\nonumber\\
& + \int_0^t\!\!\int_\Omega
(\lteta_\epsi(\teta_1)-\lteta_\epsi(\teta_2))(g(\teta_1)-g(\teta_2)) \dd x \dd s
 +\frac12 \int_\Omega  |(1 * \nabla(g(\teta_1)-g(\teta_2))(t)|^2 \dd x  \leq 0,
\end{align}
whence the desired uniqueness. Next, we test \eqref{eqapptheta1}
by $\teta$ and integrate on $(0,t)$, with $0 \leq t \leq\tau$.
Arguing in a very similar way as in the derivation of the subsequent
{\it First a priori estimate} in Sec.\ \ref{s:4}, we deduce that there exists a
positive constant $S_3$ such that
\begin{align}
& \label{stimaS3}
 \|\teta\|_{L^2(0,\tau;V) \cap L^\infty(0,\tau;L^1(\Omega))}\leq S_3,
 \\
 &
 \label{stimaS3bis}
 \eps^{1/2}\|\teta\|_{L^\infty(0,\tau;H)}\leq S_3
 \end{align}
 {\ele so that also $\fteta(\teta)\in L^2(0,\tau;V)$. Thus,  recalling  \eqref{ope-1} as well, }
by comparison in \eqref{eqapptheta1}, {\ele we get}
\begin{align}
&
 \label{stimaS3ter}
 \|\partial_t \applogteta(\teta)\|_{L^2(0,\tau;\V')}\leq S_3.
 \end{align}

Analogously, we handle equation \eqref{eqappthetas1} taking into
account the monotonicity, the bi-Lipschitz continuity of
$\applogtetas$ and $f_\eps$ and the coercivity of $j^*_\eps$ (cf.\
Lemma \ref{l:new-lemma1} and Lemma \ref{l:new-lemma2}). In
particular, in order to conclude suitable estimates for $\teta_s$,
we test \eqref{eqappthetas1} by $\teta_s$  and argue as in the
derivation of the forthcoming {\it First a priori estimate}. We
obtain
 \begin{align}
 \label{stimaS4}
 \|\teta_s\|_{L^\infty(0,\tau;L^1({\Gammac}))} +  \|\teta_s\|_{L^\infty(0,\tau;\hc)}  \leq S_4
 \end{align}
 for some constant $S_4>0$.
 Moreover, we test \eqref{eqappthetas1} by $\applogtetas(\teta_s)$. Proceeding as in the upcoming {\it Second a priori estimate},
 we deduce
  \begin{align}
& \label{stimaS5}
 \|\teta_s\|_{L^2(0,\tau;V_{\Gammac})} \leq S_5,
 \end{align}
 for some constant $S_5>0$. Observe that the latter estimate holds uniformly w.r.t\ $\epsi>0$
 and this property will be crucial to deduce that the local-existence time does
  not depend on
 $\widehat T>0$.
 Finally, by comparison in \eqref{eqappthetas1}, taking into
account that $f'_\eps$  is  bounded by  a constant depending on
$\eps$ (cf.\ \eqref{def-appf} and \eqref{bi-Lip-tetas}), we find
\begin{align}
 &
 \label{stimaS5bis}
  \|\partial_t  \applogtetas(\teta_s)\|_{L^2(0,\tau;\Vc')}\leq S_6(\eps),
  \end{align}
for some constant $S_6(\eps)>0$  depending on $\eps>0$ as well.

\medskip
\noindent
 Therefore we may  define an operator
 \begin{equation}
 \label{ope-3}
 \begin{aligned}
  & \mathcal{T}_3: \pi_{1,2}(\mathcal{Y}_\tau) \times \mathcal{U}_\tau  \times \mathcal{X}_\tau \to
  \\ &
  \begin{aligned}
 \mathcal{W}_\tau:=
 \{  (\teta,\teta_s) \in
     & (L^2 (0,\tau; V) \cap L^\infty (0,\tau;H)) \times (L^2 (0,\tau; \Vc) \cap L^\infty (0,\tau;\hc)) :
    \\
    &
    \ \|\teta\|_{L^2 (0,\tau; V) \cap L^\infty (0,\tau;L^1(\Omega))} + \eps^{1/2}\|\teta\|_{L^\infty (0,\tau;H)} \leq S_3,
    \\
    & \ \|\teta_s\|_{L^2 (0,\tau; \Vc) \cap L^\infty (0,\tau;L^1(\gc))}  + \eps^{1/2}\|\teta_s\|_{L^\infty (0,\tau; \hc)} \leq S_4
   \}
   \end{aligned}
    \end{aligned}
 \end{equation}
 mapping $(\widehat\teta,\widehat{\teta}_s,\uu,\chi) \in \pi_{1,2}(\mathcal{Y}_\tau) \times \mathcal{U}_\tau  \times  \mathcal{X}_\tau$
 into the unique solution $(\teta,\teta_s)$ of the Cauchy problem  for system \eqref{eqapptheta1}--\eqref{eqappthetas1}).
We are now in the position to prove the existence of local-in-time
solutions to Problem $(P_\eps)$, defined on some interval
$[0,\widehat T]$ with $0<\widehat T \leq T$. We stress that
$\widehat T$ in fact will not depend on the parameter $\eps>0$, and
such a property will be crucial in the forthcoming passage to the
limit procedure.
\begin{proposition}[Local existence for Problem $(P_\eps)$]
\label{prop:loc-exist-eps} Assume  \eqref{assumpt-domain},
 Hypotheses \ref{hyp:1}--\ref{hyp:6}, and conditions
\eqref{hyp-data} on the data $h$, $\mathbf{f}$,
$\mathbf{g}$, \eqref{cond-uu-zero}--\eqref{cond-chi-zero} on
$\uu_0,\,\chi_0$, and \eqref{datitetaapp}--\eqref{bounddatiteta-4} on
 $\teta_0^\eps,$
$\teta_s^{0,\eps}$.

Then,  there exists $\widehat{T}\in (0,T]$
 such that for every $\eps>0$  Problem   $(P_\eps)$
admits a solution $(\teta,\teta_s,\uu,\chi,\mmu)$ on the interval
$(0,\widehat{T})$,
fulfilling 
 the \emph{(approximate) energy identity}
\begin{equation}
\label{enid0}
 \begin{aligned}
  &\int_\Omega \calI_\eps (\teta(t)) \dd x+\int_s^t
  \int_\Omega g'(\teta)|\nabla\teta|^2 \dd x \dd r
  + \int_\Gammac i_\eps (\teta_s(t)) \dd x
   +\int_s^t\int_\Omega\nabla f_\eps(\tetas)\nabla\tetas \dd x \dd r  \\
  &
 \quad
 +
  \int_s^t\int_{\Gammac}k(\chi)(\teta-\tetas)^2 \dd x \dd r
  +\int_s^t b(\partial_t\uu,\partial_t\uu) \dd r +\frac 1 2 a(\uu(t),\uu(t))
  +\frac 1 2\int_{\gc}\chi(t)|\uu(t)|^2 \dd x
  + \int_\Gammac \Phi_\eps (\uu(t)) \dd x
  \\
  & \quad  + \int_s^t \int_{\Gammac} \fc(\teta-\tetas)  \Psi(\partial_t\uu)|{\calR}(\rmD\Phi_\eps(\uu))|  \dd x \dd r +
\int_s^t\int_{\Gammac} \fc'(\teta-\tetas)\Psi(\partial_t\uu)|{\calR}(\rmD\Phi_\eps(\uu))|(\teta-\tetas) \dd x \dd r
  \\
  & \quad
  +\int_s^t \int_\Gammac |\partial_t\chi|^2 \dd x \dd r
  +\int_s^t \int_\Gammac  \rho_\eps(\chi_t) \chi_t \dd x \dd r
  +\frac12 \int_\Gammac |\nabla \chi(t)|^2 \dd x
  + \int_\Gammac \widehat{\beta}_\eps (\chi(t)) \dd x
   + \int_\Gammac {\ele \gamma} (\chi(t)) \dd x
   \\
  &= \int_\Omega \calI_\eps (\teta(s)) \dd x
  +  \int_\Gammac i_\eps (\teta_s(s)) \dd x
  +   \bec \int_s^t\pairing{}{V}{h}{\teta}\dd r \eec + \bec \int_s^t\pairing{}{\WW}{\FF}{\partial_t\uu} \dd
  r \eec
  + \frac 1 2 a(\uu(s),\uu(s))
  \\
  & \quad
  + \int_\Gammac \Phi_\eps (\uu(s)) \dd x
  +
  \frac 1 2\int_{\Gammac}\chi(s)|\uu(s)|^2\dd x +
   \frac12 \int_\Gammac |\nabla \chi(s)|^2 \dd x
  + \int_\Gammac \widehat{\beta}_\eps (\chi(s)) \dd x
   + \int_\Gammac {\ele \gamma} (\chi(s)) \dd x
 \end{aligned}
\end{equation}
for all $0 \leq s \leq t \leq \widehat T$.
\end{proposition}
\begin{proof}
Let the operator $\mathcal{T}: \mathcal{Y}_\tau \to \mathcal{W}_\tau
\times \mathcal{U}_\tau \times \mathcal{X}_\tau$
 be defined by
\begin{equation}
\label{def-T-ope} \mathcal{T}(\widehat{\teta}, \widehat{\teta}_s,
\widehat{\uu},\widehat{\chi} ) := (\teta, \teta_s,\uu,\chi) \ \
\text{with } \ \
\begin{cases}
\uu:= \mathcal{T}_1(\widehat{\teta}, \widehat{\teta}_s,
\widehat{\uu},\widehat{\chi} ),
\\
\chi:= \mathcal{T}_2 (\widehat{\teta}_s,\uu),
\\
(\teta,\teta_s):= \mathcal{T}_3 ( \widehat\teta,\widehat{\teta}_s,
\uu,\chi).
\end{cases}
\end{equation}
Our aim is now to show that
 there exists $\widehat{T} \in (0,T]$ such that for every $\eps >0$
\begin{gather}
\label{itself} \text{
 $\mathcal{T}$ maps
 $\mathcal{Y}_{\widehat{T}}$ into itself,}
 \\
 \label{compact-conti}
 \begin{gathered}
\mathcal{T} : \mathcal{Y}_{\widehat{T}} \to
\mathcal{Y}_{\widehat{T}}  \ \ \text{ is compact and
 continuous w.r.t. the topology of } \\ \text{ $L^2(0,\tau;H^{1-\delta}(\Omega))\times L^2(0,\tau;H^{1-\delta}(\Gammac))
\times L^2(0,\tau;H^{1-\delta}(\Omega;\R^3)) \times
L^2(0,\tau;\Hc)$.}
\end{gathered}
\end{gather}
We will prove \eqref{itself} and  \eqref{compact-conti} following the
lines of the proof of \cite[Proposition 4.9]{bbr6}.

 We shall not repeat
the arguments leading to \eqref{compact-conti}, as they are
completely analogous to those in \cite{bbr6}. We only observe that,
in the proof of the continuity of the operator $\mathcal{T}_2$ \bec
from \eec \eqref{ope-2}, providing the solution of the Cauchy
problem for
 \eqref{eqS1/2}, the limit passage in
the term $\rho_\epsi (\dt\chi)$ can be easily handled  in the very
same way as in the forthcoming Sec.\ \ref{s:5}, cf.\
\eqref{limsup-rho} \bec later on. \eec
 Arguing as in
\cite[Proposition 4.9]{bbr6}, compactness and continuity of the
operator $\mathcal{T}$ \eqref{def-T-ope} in the
$(\teta,\teta_s)$-component can be proved by first
deriving compactness for $(\applogteta(\teta),\applogtetas(\teta_s))$
 (exploiting also
estimates \eqref{stimaS3ter} and \eqref{stimaS5bis}).
Since, for $\eps>0$ fixed,  $\applogteta$ and
$\applogtetas$ are bi-Lipschitz
 (cf.\ Lemma \ref{l:new-lemma2}),
 we can then infer compactness in $(\teta,\teta_s)$.
%

In what follows we  will detail the proof of \eqref{itself}, highlighting
that the final time $\widehat T$ for which \eqref{itself} holds is
independent of $\epsi$.
Let $(\widehat\teta,\widehat{\teta}_s,\widehat\uu, \widehat\chi)\in
\mathcal{Y}_\tau$ be fixed, and let
$(\teta,\teta_s,\uu,\chi):=\mathcal{T}(\widehat\teta,\widehat{\teta}_s,\widehat\uu\,\widehat\chi)$.
 We use the interpolation inequality
 \begin{equation}
\label{interpolation-brezis}
 \|\teta(t)\|_{H^{1-\delta}(\Omega)}\leq c \|\teta(t)\|^{1-\delta}_{H^1(\Omega)}\|\teta(t)\|^\delta_{L^2(\Omega)} \qquad \foraa\, t \in (0,\tau)
\end{equation}
(cf.\ e.g.\ \cite[Cor. 3.2]{brezis-mironescu}). Now, a further
interpolation between the spaces $L^2 (0,\tau;V)$ and
$L^\infty(0,\tau;L^1(\Omega))$ and estimate \eqref{stimaS3} also
yield the bound $\| \teta\|_{L^{10/3}(0,\tau;L^2(\Omega))} \leq
\bar{C} S_3$ for some interpolation constant $\bar{C}$. Integrating
\eqref{interpolation-brezis} in time and using H\"older's inequality
we therefore have
\begin{equation}
\label{inte-teta}
\begin{aligned}
\|\teta\|_{L^2(0,t;H^{1-\delta}(\Omega))}^2 & \leq c
 \int_0^t  \|\teta(s)\|^{2(1-\delta)}_{H^1 (\Omega)} \|\teta(s) \|^{2\delta}_{L^2 (\Omega)}\dd s
 \\ & \leq c  \| \teta \|_{L^{2}(0,\tau;H^1(\Omega))}^{2(1-\delta)} t^{(2 \delta)/5}
   \| \teta \|_{L^{10/3}(0,\tau;L^2(\Omega))}^{2\delta}
\leq  C   S_3^{2} t^{(2 \delta)/5}\,.
\end{aligned}
\end{equation}
 We use \bec the analogues of \eec \eqref{interpolation-brezis}
for $\teta_s$ and $\uu$ to estimate the norms $\|
\tetas\|_{L^2(0,t;H^1(\Gammac))}$ and $\| \uu \|_{L^2(0,t;
H^{1-\delta}(\Omega;\R^3))}$, in 
 \bec the same way as in \eec
\eqref{inte-teta}.
 For $\chi$ we trivially have
$\|\chi\|_{L^2(0,t;\hc)}^2 \leq t \|\chi\|_{L^\infty(0,t;\hc)}^2
\leq t S_2^2$. Combining all of these estimates,   which hold
uniformly w.r.t.\ $\epsi$, we \bec infer \eec that there exists a
sufficiently small $\widehat{T}>0$ for which \eqref{itself} holds.

 Thus we conclude that the operator $\mathcal{T}$ admits
a fixed point in ${\calY}_{\widehat T}$. \bec Hence, \eec  there
exists a solution $(\teta,\teta_s,\uu,\chi,\mmu)$ to the Cauchy
problem for system \eqref{teta-weak-app}--\eqref{eqIIa-irr-app} on
the interval $(0,\widehat T)$, with the regularity
\eqref{reg-teta-app}, \eqref{reg-log-teta-app}, \eqref{reguI},
\eqref{mureg} and, in addition,
\begin{align}
& \label{reg-teta-s-app-fix} \teta_s \in  L^2 (0,\widehat T;\Vc) \cap
L^\infty (0,\widehat T;\Hc),
\\
& \label{reg-log-teta-s-app-fix} \applogtetas(\teta_s) \in L^2
(0,\widehat T;\Vc)\cap \rmC^0 ([0,\widehat T];\Hc) \cap H^1(0,\widehat T;\Vc'),
\\
& \label{regchi-fix} \chi \in L^{2}(0,\widehat T;H^2 (\Gammac))  \cap
L^{\infty}(0,\widehat T;\Vc) \cap H^1(0,\widehat T;\Hc).
\end{align}
Moreover, to obtain the enhanced regularity properties
\eqref{reg-teta-s-app}--\eqref{reg-log-teta-s-app} and
\eqref{add-reg-chi}, we have to perform a further priori estimate
 on the $(\teta_s,\chi)$-component of
the solution. First, we readily deduce
\begin{align}
& \label{reg-fteta-s-app} \ftetas_\eps(\teta_s) \in  L^2 (0,\widehat T;\Vc)
\cap L^\infty (0,\widehat T;L^1(\gc)),
\end{align}
by \eqref{reg-teta-s-app-fix} and the bi-Lipschitz continuity of
$\ftetas_\eps$.
Moreover, we test \eqref{teta-s-weak-app} by $\partial_t
\ftetas_\eps(\teta_s)$, \eqref{eqIIa-irr-app} by $\partial_t
(A\chi+\beta_\eps(\chi))$, add the resulting relations, and
integrate in $(0,t)$, $t\in (0,\widehat T)$. Arguing as in the
derivation of the upcoming {\it Seventh a priori estimate}, we
obtain
\begin{align}
& \label{reg-teta-s-app-ult} \dt\teta_s \in  L^2 (0,\widehat T;\hc),
\\
& \label{reg-fteta-s-app-ult} \ftetas_\eps(\teta_s) \in  L^\infty
(0,\widehat T;\Vc),
\\
& \label{reg-chi-app-ult} \dt\chi \in  L^2 (0,\widehat T;\Vc) {\text { and }}
A(\chi) \in L^\infty (0,\widehat T;\hc),
\end{align}
and, by comparison in \eqref{eqIIa-irr-app},
\begin{align}
& \label{reg-chi-app-ult-bis} \dt\chi \in  L^\infty (0,\widehat T;\hc).
 \end{align}
Then, also  taking into account the bi-Lipschitz continuity of $\ftetas_\eps$
and $\applogtetas$, we
conclude
\eqref{reg-teta-s-app}--\eqref{reg-log-teta-s-app} and
\eqref{add-reg-chi}.
\paragraph{\bf Proof of  the energy identity  \eqref{enid0}.}  We  test \eqref{teta-weak-app} by
$\vartheta$, \eqref{teta-s-weak-app} by $\vartheta_s$,
\eqref{eqIa-app} by $\partial_t \uu$, and \eqref{eqIIa-irr-app} by
$\partial_t\chi$, add the resulting relations, and integrate on
$(s,t)$, $t \in (0,T]$.
The thermal expansion term in \eqref{eqIa-app}
cancels out with the one from \eqref{teta-weak-app}, and so does
$-\int_s^t \int_\Gammac \partial_t \lambda (\chi) \tetas \dd x \dd r$
with the corresponding term from \eqref{eqIIa-irr-app}.
 Integrating by parts in time
 $\int_s^t \int_\Gammac \chi \uu \cdot \uu_t \dd x \dd r$ we also have a cancellation with
 the term $-\frac12 \int_s^t \int_\Gammac \chi_t |\uu|^2 \dd x \dd r $ from \eqref{eqIIa-irr-app}.
We also use the \emph{formal} chain rule (cf.\ Remark
\ref{rmk:rigour})
\begin{equation}
\label{enidI} \pairing{}{V}{\partial_t \applogteta(\teta)}{\teta}  =
\int_\Omega\applogteta'(\teta)\partial_t\teta\teta \dd x
  =\frac{\dd}{\dd t}\int_\Omega{\calI}_\eps(\teta)  \dd x
  \end{equation}
  yielding
  \[
  \int_s^t   \pairing{}{V}{\partial_t \applogteta(\teta)}{\teta}  \dd r =
\int_\Omega {\calI}_\eps(\teta(t))  \dd x-
\int_\Omega{\calI}_\eps(\teta(s)) \dd x
  \]
and the same for the term $ \int_s^t   \pairing{}{\hunoc}{\partial_t \applogtetas(\tetas)}{\tetas}  \dd r$,
giving
$ \int_\Gammac {i}_\eps(\tetas(t))  \dd x-
\int_\Gammac{i}_\eps(\tetas(s)) \dd x.
$
Instead, the chain rule identity
\[
\int_s^t \int_\Gammac \mathrm{D}\Phi_\eps (\uu) \cdot \uu_t \dd x \dd r = \int_\Gammac \Phi_\eps(\uu(t)) \dd x -  \int_\Gammac \Phi_\eps(\uu(s)) \dd x
\]
holds rigorously.
Furthermore, we exploit that $\zz \cdot \partial_t \uu = \Psi (\partial_t \uu)$ a.e.\ in $\Gammac \times (0,\widehat T)$ by the
$1$-homogeneity of the functional $\Psi$. This gives rise to the  tenth term on the left-hand side of \eqref{enid0}.
Then, we conclude \eqref{enid0}.
\end{proof}

\section{\bf Uniform w.r.t. $\eps$ a priori estimates}
\label{s:4}

In this section,
we perform  a series of priori estimates on the solutions to Problem
\ref{prob:irrev-app} \bec (i.e.\ Problem $(P_\eps)$). \eec  From
them,
 we derive bounds on suitable norms of the local solutions, which hold uniformly w.r.t.\ the parameter
 \ber $\eps\in(0,1)$. \edr Exploiting them, we  will pass to the limit with $\eps$ in Problem \ref{prob:irrev-app}
 {\ele $(P_\varepsilon)$} and conclude \bec in \eec Theorem
 \ref{thm:exist-local} the existence of a local-in-time solution to Problem \ref{prob:irrev}.

 Let us mention in advance that the forthcoming a priori estimates are not, however, global in time.
  This \emph{local character}
 manifests itself already with the \emph{First a priori estimate} (i.e.\ the \emph{energy estimate}), which derives from the approximate
 energy identity \eqref{enid0}. More precisely, the problem is to estimate the left-hand side term
 \begin{equation}
 \label{trouble-maker}
 \frac 1 2\int_{\gc}\chi(t)|\uu(t)|^2 \dd x
 \end{equation}
 therein. {\ele Indeed, since}  in Problem \bec  $(P_\eps)$ \eec
\bec the maximal monotone operator \eec  $\beta$ has been replaced
by its Yosida regularization $\beta_\eps
 $, the approximate solution $\chi$ is no longer guaranteed to be positive a.e.\ in $\Gammac \times (0,\widehat T)$. Therefore
 the term in \eqref{trouble-maker} is not, a priori, estimated from below by a constant. On the other hand, it cannot be moved to the right-hand side
 of  \eqref{enid0} and absorbed into the left-hand side by H\"older and Young inequalities, {\ele
 or by use of  the Gronwall lemma, mainly due to a lack of regularity of the terms in the left-hand side}.

 That is why, in order to estimate \eqref{trouble-maker} we will resort to the following argument.
It follows from the  fixed
point procedure in Sec.\ \ref{ss:3.3} that the local solution
$(\teta,\tetas,\uu,\chi)$
 whose existence we have proved in Prop.\ \ref{prop:loc-exist-eps}
 belongs to ${\calY}_{\widehat
T}$ \bec from \eqref{fixed-point-set}, \eec
whence
\begin{equation}\label{dalpuntofisso}
 \|\teta\|_{L^2(0,\widehat T;H^{1-\delta}(\Omega))}+\|\teta_s\|_{L^2(0,\widehat T;H^{1-\delta}(\Gammac))}
+ \|\uu\|_{L^2(0,\widehat T;H^{1-\delta}(\Omega;\R^3))}+ \|\chi\|_{L^2(0,\widehat T;\Hc)}
\leq M,
\end{equation}
where $M$ does not depend on $\eps$ (see \eqref{fixed-point-set}).
In addition, as  estimates \eqref{boundS1} \bec and \eec
\eqref{boundS1/2} do not depend on $\eps$, we infer
\begin{equation}\label{dalpuntofissodue}
\|\uu\|_{H^1(0,\widehat T;\WW)}+\|\chi\|_{H^1(0,\widehat T;\Hc)}\leq c.
\end{equation}
Exploiting \eqref{dalpuntofissodue} as well as well-known
embedding and trace results, we then have
\begin{equation}\label{stimadiffiI}
\frac 1 2\int_{\gc}\chi(t)|\uu(t)|^2 \dd x
\leq
c\|\chi\|_{L^\infty(0,{\ele\widehat T};\Hc)}\|\uu\|^2_{L^\infty(0,{\ele\widehat T};L^4(\Gammac))}
\leq c\|\chi\|_{L^\infty(0,{\ele\widehat T};\Hc)}\|\uu\|^2_{L^\infty(0,{\ele\widehat T};\WW)}
{\ele\leq c},
\end{equation}
where $c$ depends here on $S_1,S_2$, {\ele on $\widehat T$, } but
not on $\eps>0$.

Clearly, since the \emph{First a priori estimate} is not global in
time, neither of the subsequent estimates is. Nonetheless, for
notational simplicity in the following calculations we shall write
$T$ in place of the local-existence time $\widehat T$. Moreover, we
shall omit to indicate the dependence on $\eps$ for the approximate
solutions. We also mention that with the symbols $c$ and $C$ we will
denote possibly different positive constants, depending on the data
of the problem, \bec but not \eec on $\eps$.

\paragraph{\bf First a priori estimate.}
We consider the approximate energy identity \eqref{enid0} on the
interval $(0,t)$, with $t \in (0,T)$ {\ele (i.e.\ in $(0, \widehat
T)$)}. Thanks to Lemma \ref{l:new-lemma1} (cf.\
\eqref{ci-serve-anticipata-1} and \eqref{ci-serve-anticipata-2}) we
have
$$
\begin{aligned}
&
\int_\Omega {\calI}_\eps(\teta(t))  \dd x \geq  \frac{\eps}{2} \| \teta(t)
\|_{H}^2  + c \| \teta(t) \|_{L^1(\Omega)}  - C,
\\
&
\int_\Gammac {\it i}_\eps(\tetas(t)) \dd x
 \geq  \frac{\eps}{2} \| \tetas(t) \|_{\Hc}^2 +
c \| \tetas(t) \|_{L^1(\Gammac)} -C,
\end{aligned}
$$
while \eqref{bounddatiteta-1} and \eqref{bounddatiteta-2}
 \bec ensure \eec  that
\[
\int_\Omega{\calI}_\eps(\bec \teta_0^\eps \eec) \dd x \leq C\,,
\qquad \int_\Omega{\it i}_\eps(\bec {\tetas}^{0,\eps} \eec) \dd x
\leq C\,.
\]

 Since $g'$  is bounded from below by a strictly positive constant {\ele (see \eqref{cond-g})}, and so is $\ftetas_\eps'$ (cf.\ Lemma \ber \ref{l:new-lemma2})\edr,
 we also have 
\[
\begin{aligned}
&
\int_0^t\int_\Omega g'(\teta)|\nabla\teta|^2 \dd x \dd s \geq {\ele c_3}\int_0^t\int_\Omega |\nabla\teta|^2 \dd x \dd s,
\\
&
 \int_0^t \int_\gc \nabla f_\eps(\teta_s) \nabla\teta_s \dd x \dd r=
 \int_0^t \int_\gc \frac {1}{\applogtetas'(\teta_s)} |\nabla\teta_s|^2 \dd x \dd r
\geq\frac{\eps}3\int_0^t \int_\gc |\nabla \teta_s|^2 \dd x \dd r.
\end{aligned}
\]
We observe that the term $\int_0^t \int_\Gammac k(\chi) (\teta-\tetas)^2 \dd x \dd s $ is positive thanks to
\eqref{hyp-k}, and  so are the ninth, the tenth and the eleventh terms on the left-hand side of
\eqref{enid0} by
\eqref{hyp:phi} and
\eqref{hyp-fc}, respectively
{\ele, i.e.
\begin{align}
&\int_\Gammac \Phi_\eps (\uu(t)) \dd x + \bec \int_0^t \eec
\int_{\Gammac} \fc(\teta-\tetas)
\Psi(\partial_t\uu)|{\calR}(\rmD\Phi_\eps(\uu))| \dd x \dd r\\\no & +
\bec \int_0^t \eec \int_{\Gammac}
\fc'(\teta-\tetas)\Psi(\partial_t\uu)|{\calR}(\rmD\Phi_\eps(\uu))|(\teta-\tetas)
\dd x \dd r\geq 0.
\end{align}
}
Since $\rho_\eps (0) =0$, we also infer that
$$
\int_0^t \int_\Gammac \rho_\eps(\chi_t) \chi_t \dd x \dd s \geq 0
$$
by monotonicity of $\rho_\eps$.
 Furthermore,
 we estimate (cf.\ \eqref{e:max-mon-5}, here $\mathfrak{r}_\eps$ denotes the resolvent of
$\beta$)
\[
\int_\Gammac \widehat{\beta}_\eps (\chi(t)) \dd x \geq \int_\Gammac \widehat{\beta} (\mathfrak{r}_\eps (\chi(t)))  \dd x
\geq -c \int_\Gammac |\mathfrak{r}_\eps (\chi(t))| \dd x - C
\geq - c \int_\Gammac |\chi(t)| \dd x  -C
\]
since $\mathfrak{r}_\eps$ is a contraction, yielding
\[
 |\mathfrak{r}_\eps (\chi(t))| \leq    |\mathfrak{r}_\eps (\chi(t))- \mathfrak{r}_\eps(x_0)| + | \mathfrak{r}_\eps(x_0)| \leq |\chi(t)| +    | \mathfrak{r}_\eps(x_0)|+ |x_0|
 \leq  |\chi(t)| +c
\]
(where $x_0$ is a  point in $\mathrm{dom} (\widehat \beta)$).
Moreover, since ${\ele \gamma}'$ is Lipschitz by  \eqref{hyp-sig},
${\ele \gamma}$ has at most quadratic growth, therefore
\[
\int_\Gammac {\ele \gamma}(\chi(t)) \dd x \geq - c \int_\Gammac |\chi(t)|^2 \dd x - C.
\]
Finally,  we estimate
\[
\begin{aligned}
&
\int_0^t \pairing{}{\WW}{\mathbf{F}}{\uu_t} \dd s \leq \int_0^t \| \mathbf{F} \|_{\WW'} \|\uu_t \|_{\WW} \dd s,
\\
&
 \int_0^t\pairing{}{V}{h}{\teta} \dd s  \leq {\ele c}\int_0^t \| h \|_{V'} (\| \teta \|_{L^1(\Omega)} + \| \nabla \teta \|_{H}) \dd s.
\end{aligned}
\]

All in all, taking into account \eqref{korn_a} 
 we conclude
\begin{equation}
\label{Istima}
\begin{aligned}
  & \frac{\eps}{2} \ber \| \teta (t)\|_{H}^2  \edr
    {\ele+ c}\| \teta(t) \|_{L^1(\Omega)} +
    \ber {c_3}\edr\int_0^t \int_\Omega |\nabla \teta|^2 \dd x \dd s
    +  \ber \frac{\eps}{2} \| \tetas (t) \|_{\Hc}^2 \edr  +  {\ele c}\| \tetas(t) \|_{L^1(\Gammac)}
    +{\ele\frac \eps 3}\int_0^t \int_\gc |\nabla \teta_s|^2 \dd x \dd s
    \\
  & \quad
  +\frac {C_a} 2\|\uu(t)\|^2_\WW+C_b\int_0^t\|\partial_t\uu\|^2_\WW \dd  s
  +\int_0^t\|\partial_t\chi\|^2_{\Hc}   \dd s \ber + \frac 12\|\nabla\chi (t)\|^2_{\Gammac}\edr
  \\
  & \leq
     C_0  + c \left(\int_0^t  \|\mathbf{F}\|^2_{\WW'} \dd s + \int_0^t  \| h \|^2_{V'} \dd s \right)
     +  \frac{C_b}4 \int_0^t \| \uu_t \|_{\WW}^2 \dd s
      + \frac{c_3}{4}\int_0^t   \int_\Omega |\nabla \teta|^2 \dd x \dd s
      \\
      & \quad
      + \bec c \eec \int_0^t  \| h \|_{V'} \| \teta\|_{L^1(\Omega)} \dd s
  + C' \| \chi(t) \|_{\Hc}^2 + C\,.
   \end{aligned}
\end{equation}
Here, the constant
$C_0$ depends on the initial data, in view of \ber \eqref{bounddatiteta-1}--\eqref{bounddatiteta-2}\edr
 and of conditions  \eqref{cond-uu-zero}--\eqref{cond-chi-zero} on $\uu_0$ and $\chi_0$, and we have also
used \eqref{stimadiffiI} to estimate the term
in~\eqref{trouble-maker}.

Hence, applying the Gronwall lemma
we conclude {\ele (in addition to \eqref{dalpuntofissodue})}
\begin{align}
\label{first-aprio-teta} & \bec \eps^{1/2} \| \teta \|_{L^\infty
(0,T;H)} + \eec \|\teta \|_{L^2(0,T;V)\cap
L^\infty(0,T;L^1(\Omega))} \leq c,\\
\label{first-aprio-teta-s} &\|\teta_s \|_{L^\infty(0,T;L^1(\gc))}
\leq c,
\\
\label{first-aprio-u-chi}
&
 \|\chi\|_{L^\infty(0,T;\Vc)} \leq c.
\end{align}
%
\paragraph{\bf Second a priori estimate.}
We test \eqref{teta-weak-app} by $\applogteta(\teta)$ and
\eqref{teta-s-weak-app} by $\applogtetas(\teta_s)$, add the
resulting relations and integrate in time. In particular, we use
that \bec the term $\int_0^t \int_\gc \nabla f_\eps(\teta_s) \,
\nabla \applogtetas(\teta_s) \dd x \dd r$ on the left-hand side of
the resulting inequality fulfills \eec
\begin{equation}
\label{costruita-cosi} \int_0^t \int_\gc \nabla f_\eps(\teta_s) \,
\nabla \applogtetas(\teta_s) \dd x \dd r=
 \int_0^t \int_\gc \frac {1}{\applogtetas'(\teta_s)} \nabla\teta_s\, \applogtetas'(\teta_s) \, \nabla\teta_s\dd x \dd r
=\int_0^t \int_\gc |\nabla \teta_s|^2 \dd x \dd r,
\end{equation}
\bec thanks to  \eqref{def-appf}. \eec Then, we observe that (see
\eqref{bi-Lip} and \eqref{cond-g})
\begin{equation}
\label{costruita-cosi-2} \int_0^t \int_\Omega \nabla g(\teta) \,
\nabla \applogteta(\teta) \dd x \dd r\geq \eps c_3\int_0^t
\int_\Omega|\nabla\teta|^2 \dd x \dd r\geq 0.
\end{equation}
Hence, integrating by parts in time and exploiting \bec
\eqref{dalpuntofissodue},
\eqref{first-aprio-teta}--\eqref{first-aprio-u-chi}, \eec
 using the Young inequality and the
Gronwall lemma it is a standard matter to get
\begin{equation} \label{second-aprio-ln-0} \|
\applogteta(\teta)\|_{L^\infty (0,T;H)}+ \|
\applogtetas(\teta_s)\|_{L^\infty (0,T;\Hc)} \leq c,
\end{equation}
\bec whence \eec
\begin{equation} \label{second-aprio-ln} \|
\lteta_\eps(\teta)\|_{L^\infty (0,T;H)}+ \|
\ltetas_\eps(\teta_s)\|_{L^\infty (0,T;\Hc)} \leq c.
\end{equation}
Moreover, \bec \eqref{first-aprio-teta-s}  and \eec  \eqref{costruita-cosi} 
yield
\begin{equation}
\label{second-aprio-teta-s}
 \|\teta_s \|_{L^2(0,T;\Vc)}\leq c.
\end{equation}
\paragraph{\bf Third a priori estimate.}
We are now in the position to estimate  $\mmu$ and
$ \mathrm{D}\Phi_\eps(\uu) $ in \eqref{eqIa-app}. Indeed, by a comparison in the
equation, and the previous a priori estimates,  there holds that
$$
\|\mathrm{D}\Phi_\eps(\uu)+\fc(\teta-\tetas)\mmu\|_{L^2(0,T; \bsY_{{\Gammac}}')}\leq c.
$$
Then, after observing that the two addenda are orthogonal
thanks to assumption
\eqref{ass:orthogonal} combined with \eqref{to-be-cited-orthog},
it is
straightforward to deduce that  each of  them  is bounded in
$L^2(0,T;\bsY_{{\Gammac}}')$ arguing in the same way as in \cite[Sec.\ 4]{bbr5}
 In particular, we get
\begin{equation}
\label{added-3d}
\|\mathrm{D} \Phi_\eps(\uu)\|_{L^2(0,T;\bsY_{{\Gammac}}')}\leq c.
\end{equation}
Now, taking into account that
$\mmu = |\Reg (\mathrm{D} \Phi_\eps(\uu))| \zz $, {\ele and the fact }
 that $\| \zz \|_{L^\infty (\Gammac \times (0,T))} \leq C$ by \eqref{bounded-operator},
and exploiting \bec that $\Reg: L^2(0,T;\bsY_{{\Gammac}}') \to
L^\infty (0,T;L^{2+\nu}(\gc;\R^3))$ is bounded by \eqref{hyp-r-1},
\eec  we ultimately conclude
\begin{equation}
\label{third-aprio}
 \|\mmu\|_{L^\infty
(0,T;L^{2+\nu}(\gc;\R^3))}
\leq c.
\end{equation}

\paragraph{\bf Fourth a priori estimate.}
We perform a comparison argument in \eqref{teta-weak-app}.
 We take into account the previously proved
estimates
 \eqref{dalpuntofissodue}, \eqref{first-aprio-teta}, \eqref{second-aprio-teta-s},
\eqref{added-3d}, and combine them with \bec \eqref{ass-psi} and
\eqref{hyp-k}--\eqref{hyp-fc}. \eec In particular, we infer (cf.\
\eqref{not-mathcalF}--\eqref{est-math-cal-F}) that \bec $
 k(\chi)  {\ele (\teta-\teta_s)}+\fc'{\ele (\teta-\teta_s)} |\Reg (\mathrm{D}\Phi_\eps(\uu))| \Psi(\uu_t)
$ is bounded in $L^2 (0,T;L^{4/3+s}(\gc))$ \eec
 for some $s>0$.
Then, also taking into account
\eqref{hypo-h}, we conclude that
 \begin{equation}
\label{fourth-aprio} \|\partial_t\applogteta(\teta)  \|_{L^2
(0,T;V')} \leq c.
\end{equation}
\paragraph{\bf Fifth a priori estimate - BV estimate for $\teta$.}
Now, we formally test \eqref{teta-weak-app} by
$\frac{1}{\applogteta'(\teta)} v$, with $v\in W^{1,q}(\Omega)$, $q>3$.
Observe that this guarantees  $v \in L^\infty (\Omega) $ and that its trace is in $L^\infty (\gc)$.
The formal identity
 \begin{equation}
\label{e:formal1-5} \pairing{}{\V}{\partial_t
\applogteta(\vartheta)}{\frac{1}{\applogteta'(\teta)} v}=
\int_{\Omega} \bec \partial_t\teta \eec  {\ele v} \dd x
\end{equation}
holds. Next, we  exploit the Lipschitz continuity of
$\frac{1}{\applogteta'}$ (cf.\ Lemma \ref{l:new-lemma2}) to deduce
that 
\begin{equation}
\label{crucial-log} \left
\|\frac{1}{\applogteta'(\teta)}\right\|_H\leq c \|\teta\|_H+C,
\quad\left \|\nabla\frac{1}{\applogteta'(\teta)} \right\|_H\leq
c\|\nabla\teta\|_H, \quad \text{whence }
\left\|\frac{1}{\applogteta'(\teta)}\right\|_{\ele V} \leq c \|
\teta\|_V + C\,.
\end{equation}
\bec Thus, \eec
 by \eqref{first-aprio-teta}
$$
\left\|\frac{1}{\applogteta'(\teta)}\right\|_{L^2(0,T;V)}\leq c.
$$
Therefore,
also taking into account \eqref{cond-g},
 we estimate
\begin{align}
&
\label{bv-1}
 \left|
 \int_{\Omega} \dive(\uu_t)  \frac{1}{\applogteta'(\teta)} v      \dd x
 \right| \leq  {\ele c}\| \uu_t \|_{\WW} \left\|\frac{1}{\applogteta'(\teta)}\right\|_{V} \| v\|_{L^3(\Omega)}
 \leq  {\ele c}\| \uu_t \|_{\WW}  {\ele ( \|
\teta\|_V + 1
)}\| v\|_{L^3(\Omega)}  \doteq f_1 \in L^1 (0,T),
 \\
 &
 \label{bv-2}
 \begin{aligned}
 &
 \left|\int_{\Omega}   \nabla g(\vartheta) \, \nabla \left( \frac{1}{\applogteta'(\teta)} v \right)  \dd x \right|
  \\
  &
   \leq  c \| \nabla \teta\|_{H}  \left\|\nabla\frac{1}{\applogteta'(\teta)}\right\|_H \| v \|_{L^\infty (\Omega)}
   +   {\ele c}\| \nabla \teta\|_{H}\left \| \frac{1}{\applogteta'(\teta)}\right\|_{L^6 (\Omega)}  \| \nabla v \|_{L^3(\Omega)}
   \doteq f_2 \in L^1 (0,T),
   \end{aligned}
\\ &
\label{bv-3}
 \left| \int_{\Gammac}
k(\chi)  (\vartheta-\vartheta_s)   \frac{1}{\applogteta'(\teta)} v \dd x \right|
\leq \| k(\chi)  (\vartheta-\vartheta_s)   \|_{\hc} \left\|  \frac{1}{\applogteta'(\teta)} \right\|_{L^4 (\gc)} \| v \|_{L^4 (\gc)}
\doteq f_3 \in L^1 (0,T),
   \\
   &
   \label{bv-4}
   \begin{aligned}
   &
 \left|   \int_{\Gammac}
\fc'(\teta-\teta_s)  \Psi(\dotu) |\Reg (\mathrm{D}\Phi_\eps(\uu))|   \frac{1}{\applogteta'(\teta)}  v \dd x
\right|
\\
&
\leq \|\fc'(\teta-\teta_s)  \Psi(\dotu) |\Reg (\mathrm{D}\Phi_\eps(\uu))| \|_{L^{4/3}(\gc)}  \left\|  \frac{1}{\applogteta'(\teta)} \right\|_{L^4 (\gc)}
\| v \|_{L^\infty (\gc)} \doteq f_4 \in L^1 (0,T),
\end{aligned}
\\
&
\label{bv-5}
 \left| \pairing{}{V}{h}{\frac{1}{\applogteta'(\teta)} v}  \right|
 \leq  {\ele c}\| h \|_{V'} \left( \left\|  \frac1{\applogteta'(\teta)} \right \|_{V} \| v \|_{W^{1,q}(\Omega)}\right)  \doteq f_5 \in L^1 (0,T),
\end{align}
where \eqref{bv-1}--\eqref{bv-5} follow from the previously proved estimates. Thus we conclude that
\[
\exists\, {\ele {\mathcal F}} \in L^1 (0,T)\, \quad \text{for a.a. }
t \in (0,T) \ \forall\, v \in W^{1,q}(\Omega)\, : \quad \left|
\int_\Omega \partial_t\teta v \dd x \right| \leq{\ele {\mathcal
F}}(t) \|v \|_{ W^{1,q}(\Omega)}
\]
 and this implies
 \begin{equation}
\label{stimaBV} \| \partial_t \teta\|_{L^1 (0,T;W^{1,q}(\Omega)')}
\leq c.
\end{equation}
\paragraph{\bf Sixth a priori estimate.}
We test \eqref{teta-s-weak-app} by $\ftetas_\eps(\teta_s)$ and we
integrate in time. Taking into account \bec the definition
\eqref{def-accaapp} of $H_\eps$, we have  the formal identity \eec
 \begin{equation}
\label{e:formal2}
 \int_0^t  \pairing{}{\Vc}{\partial_t \ltetas_\eps(\vartheta_s)}{
\ftetas_\eps(\vartheta_s)} \dd r=  \int_\gc
H_\eps(\vartheta_s(t))\dd x - \int_\gc
H_\eps(\bec\vartheta_s^{0,\eps}\eec) \dd x.
\end{equation}
\bec Thus,  estimate \eqref{accaapp} for $H_\eps$ (see Lemma
\ref{l:new-lemma3}) as well as estimate \eqref{bounddatiteta-3} for
$\int_\gc H_\eps(\bec \vartheta_s^{0,\eps} \eec) \dd x$
 yield \eec
\begin{equation}
\| \ftetas_\eps(\vartheta_s)(t) \|_{L^1(\gc)}+ \|\nabla
\ftetas_\eps(\vartheta_s) \|^2_{L^2 (0,t;\Hc)}\leq C +
I_1+I_2+I_3\,,
\end{equation}
where {\ele (here we use H\"older's inequality and, in particular, \eqref{dalpuntofissodue}, \eqref{first-aprio-teta}, \eqref{second-aprio-teta-s})}
\begin{align}
&
\label{est-7-1}
\begin{aligned}
 I_{1}  = \int_0^t\int_\gc \lambda'(\chi) \partial_t\chi \ftetas_\eps(\vartheta_s)
\dd x \dd r  &  \leq c \int_0^t \| \partial_t \chi  \|_{\Hc}
\|\ftetas_\eps(\vartheta_s)  \|_{\Vc}
 \dd r \\ &
 \leq c \int_0^t \| \partial_t \chi \|_{\Hc} \left(\|\ftetas_\eps(\vartheta_s)-m(\ftetas_\eps(\vartheta_s))\|_{\Vc}+\|m(\ftetas_\eps(\vartheta_s))\|_{\Vc}\right)
\dd r \\&
 \leq c \int_0^t \|
\partial_t \chi \|_{\Hc}\left(\|\nabla
\ftetas_\eps(\vartheta_s)\|_{\Hc}+\|\ftetas_\eps(\vartheta_s)\|_{L^1(\gc)}\right)
\dd r
\\& \leq \delta \|\nabla \ftetas_\eps(\vartheta_s) \|^2_{\Hc}+ c \int_0^t \|
\partial_t \chi \|_{\Hc}\, \|\ftetas_\eps(\vartheta_s)\|_{L^1(\gc)} \dd r
+c\,,
\end{aligned}
\\
&
\label{est-7-2}
\begin{aligned}
 I_{2}  = \int_0^t\int_\gc k(\chi) (\teta-\teta_s) \ftetas_\eps(\vartheta_s)
\dd x \dd r  &  \leq c \int_0^t (\|\chi\|_{\Vc}+1)
\|\teta-\teta_s\|_{\Hc} \|\ftetas_\eps(\vartheta_s)  \|_{\Vc}
 \dd r \\ &
 \leq c \int_0^t \left( \| \teta\|_{\Hc}+ \| \teta_s\|_{\Hc} \right) \|\ftetas_\eps(\vartheta_s) \|_{\Vc}
\dd r \\&\leq \delta \|\nabla \ftetas_\eps(\vartheta_s) \|^2_{\Hc}+
c \int_0^t \| \left( \| \teta\|_{\Hc}+ \| \teta_s\|_{\Hc} \right)
\|\ftetas_\eps(\vartheta_s)\|_{L^1(\gc)} \dd r +c\,,
\end{aligned}
 \\
  &
  \label{est-7-3}
  \begin{aligned}
  I_{3}  = \int_0^t\int_\gc \fc'(\teta-\teta_s) |\Reg (\bec \mathrm{D}\Phi_\eps(\uu) \eec )| |\Psi (\uu_t)|
\ftetas_\eps(\vartheta_s)  \dd x \dd r & \leq c \int_0^t \|\Reg
(\bec \mathrm{D}\Phi_\eps(\uu) \eec)\|_{\Hc} \| \uu_t \|_{L^4(\gc)}
\|\ftetas_\eps(\vartheta_s)\|_{\Vc}\dd r
\\
& \leq c \int_0^t \| \uu_t \|_{L^4(\gc)}
\|\ftetas_\eps(\vartheta_s)\|_{\Vc}\dd r
\\ & \leq
\delta \|\nabla \ftetas_\eps(\vartheta_s) \|^2_{\Hc}+ c \int_0^t
\| \uu_t \|_{\bsW} \|\ftetas_\eps(\vartheta_s)\|_{L^1(\gc)} \dd r
+c\,,
 \end{aligned}
 \end{align}
for some sufficiently small $\delta>0$; note that  in \eqref{est-7-1} and \eqref{est-7-2} we have used that $\lambda $  and
$k$ are
 Lipschitz, while \eqref{est-7-3} follows from the fact that $|\Psi(\uu_t)|\leq C |\uu_t|$ thanks to \eqref{bounded-operator}.
 We apply the Gronwall lemma and we
conclude
\begin{equation}
\label{fifth-aprio-teta-s} \|\ftetas(\vartheta_s)\|_{L^\infty
(0,T;L^1(\gc))}+ \| \ftetas(\vartheta_s)\|_{L^2(0,T;\Vc)} \leq c.
\end{equation}
\paragraph{\bf Seventh estimate.}
We test \eqref{teta-s-weak-app} by $\partial_t
\ftetas_\eps(\teta_s)$ \bec (observe that, since $\ftetas_\eps$ is
bi-Lipschitz, $\teta_s \in H^1 (0,T; \Hc)$ implies $\ftetas_\eps(\teta_s)
\in H^1 (0,T; \Hc)$), \eec \eqref{eqIIa-irr-app} by $\partial_t
(A\chi+\beta_\eps(\chi))$, add the resulting relations, and
integrate in time. In particular, we (formally) have
$$
\int_0^t \int_\gc
\partial_t(\applogtetas(\teta_s))\, \partial_t (\ftetas_\eps(\teta_s))
\dd x \dd r=
 \int_0^t \int_\gc \applogtetas'(\teta_s) \partial_t\teta_s\, \frac {1}{\applogtetas'(\teta_s)} \partial_t\teta_s \dd x \dd r
=\int_0^t \int_\gc |\partial_t \teta_s|^2 \dd x \dd r\,.$$
Thus, \bec taking into account the monotonicity of $\beta_\eps$ and $\rho_\eps$, \eec we obtain
\begin{equation}
\label{v-formal} \|\partial_t
\vartheta_s\|^2_{L^2(0,t;\Hc)}+\frac12\|\nabla
\ftetas_\eps(\teta_s  (t) )\|^2_{\Hc}+\|\partial_t\chi\|^2_{L^2(0,t;\Vc)}
+\frac12\|A(\chi(t))+\xi(t)\|^2_{\Hc}\leq C +
I_4+I_5+I_6+I_7+I_8,
\end{equation}
\bec where we have
 used the place-holder $\xi:= \beta_\eps(\chi)$. Now, \eec 
systematically using estimate \eqref{fprimoapp}, we have  (cf.\ also
\eqref{hyp-k}--\eqref{hyp-sig}), 
\[
\begin{aligned}
&
\begin{aligned}
 I_{4}  = \int_0^t\int_\gc \lambda'(\chi)\partial_t \chi  \partial_t
\ftetas_\eps(\teta_s)  \dd x \dd r   &  \leq \delta \int_0^t \|
\partial_t\teta_s \|^2_{\Hc} \dd r + c \int_0^t\int_\gc |\partial_t \chi|^2
|\ftetas_\eps'(\teta_s)|^2 \dd x \dd r
\\ &  \leq \delta \int_0^t
\| \partial_t\teta_s \|^2_{\Hc} \dd r + c \int_0^t \|\partial_t
\chi\|_{\Hc} \|\partial_t \chi\|_{L^4(\Gammac)}
(\| \ftetas_\eps(\teta_s) \|_{\Vc} +c) \dd r
\\ &
\begin{aligned}
  \leq \delta \int_0^t
\| \partial_t\teta_s \|^2_{\Hc} \dd r &  + \delta \int_0^t \|\partial_t
\chi\|^2_{\Vc} \dd r \\ & + c \int_0^t \|\partial_t
\chi\|^2_{\Hc} \left( \|\ftetas_\eps(\teta_s)  \|^2_{\Vc}  +1\right) \dd r\,,
\end{aligned}
\end{aligned}
\\
&
\begin{aligned}
 I_{5}  = \int_0^t\int_\gc k(\chi) (\teta-\teta_s)  \partial_t \ftetas_\eps(\teta_s)
\dd x \dd r   &  \leq \delta \int_0^t \| \partial_t\teta_s
\|^2_{\Hc} \dd r + c \int_0^t\int_\gc (\chi ^2+1)
(\teta^2+\teta_s^2) |f_\eps'(\tetas)|^2 \dd x \dd r
\\ &
\begin{aligned}
\leq  &  \ \delta \int_0^t \| \partial_t\teta_s \|^2_{\Hc} \dd r \\ & + c
\int_0^t
\left(\|\teta\|^2_{L^4(\Gammac)}+\|\teta_s\|^2_{L^4(\Gammac)}\right)
\left( \|f_\eps(\tetas) \|_{\Vc} +1\right)
 \dd r\,,
 \end{aligned}
\end{aligned}
 \\
  &
  \begin{aligned}
  I_{6} & = \int_0^t\int_\gc \fc'(\teta-\teta_s) |\Reg (\eeta)| \Psi(\uu_t)
 \partial_t f_\eps(\teta_s)  \dd x \dd r
\\ & \leq
\delta \int_0^t \| \partial_t\teta_s \|^2_{\Hc} \dd r + c \int_0^t
\|\Reg (\eeta)\|^2_{L^{2+\nu}(\gc;\R^3)} \| \uu_t \|^2_{L^4(\Gammac{\ele ;\R^3})}\left(
\|f_\eps(\teta_s)\|_{\Vc} +1 \right) \dd r
\\
& \leq \delta \int_0^t \| \partial_t\teta_s \|^2_{\Hc} \dd r + c
\int_0^t
 \| \uu_t
\|^2_{\bsW} \left(
\|f_\eps(\teta_s)\|_{\Vc} +1 \right)  \dd r \,,
 \end{aligned}
\\
&
\begin{aligned}
 I_{7} &=-\int_0^t\int_\gc \left( \gamma'(\chi) + \lambda'(\chi) \teta_s \right) \partial_t
  (A\chi+\xi) \dd x \dd r
 \\
&  = \bec \int_0^t\int_\gc  (\gamma{''}(\chi)+\lambda{''}(\chi)  \teta_s ) \partial_t \chi
 (A\chi+\xi) \dd x \dd r
  +\int_0^t\int_\gc \lambda'(\chi)\partial_t\teta_s (A\chi+\xi) \dd x \dd r \eec 
  \\ & \qquad \bec  - \int_\gc (\gamma'(\chi(t)) +  \lambda'(\chi(t))\teta_s(t) )  (A\chi(t)+\xi(t)) \dd x
  +  \int_\gc (\gamma'(\chi_0) +\lambda'(\chi_0)\teta_s^0 ) (A\chi_0+\xi(0)) \dd x \eec 
\\
& \bec \leq c \int_0^t \| \partial_t \chi \|_{\Vc} (  \| \teta_s\|_{\Vc} +1)  \|
A\chi+\xi \|_{\Hc} \dd r+ c\int_0^t \|
\partial_t\teta_s\|_{\Hc}
 \| A\chi+\xi \|_{\Hc} \dd r \eec 
 \\ &
 \bec + c(\|\chi(t) \|_{\Hc} + \|\teta_s(t)\|_{\Hc} +1)  \|A\chi(t)+\xi(t)\|_{\Hc} + c \eec 
\\
& \leq \frac14 \|A\chi(t)+\xi(t)\|_{\Hc}^2 + \delta \int_0^t \|
\partial_t\teta_s \|^2_{\Hc} \dd r + \delta \int_0^t \| \partial_t
\chi \|^2_{\Vc} \dd r + c \int_0^t \left( \|\teta_s\|^2_{\Vc}+1
\right)  \|A\chi+\xi \|^2_{\Hc} \dd r
\\
& \bec +c \|\teta_s\|^2_{L^\infty(0,T;\Hc)} +c \|\chi\|^2_{L^\infty(0,T;\Hc)}  + c\,, \eec 
\end{aligned}
 \end{aligned}
 \]
for some small $\delta>0$, and
\[
\begin{aligned}
&
\begin{aligned}
  I_{8} &=-\frac12 \int_0^t\int_\gc |\uu|^2\partial_t
  (A\chi+\xi) \dd x \dd r\\ & = \int_0^t\int_\gc \uu \uu_t (A\chi+\xi) \dd x \dd
  r- \frac12 \int_\gc |\uu(t)|^2
  (A\chi(t)+\xi(t)) \dd x+ \frac12 \int_\gc |\uu_0|^2
  (A\chi_0+\xi(0)) \dd x
  \\
  & \leq {\ele c} \| \uu\|_{L^\infty(0,T;\bsW)}^2 \int_0^t \| \uu_t
  \|_{\bsW}^2 \dd r +\frac12\int_0^t \|A\chi+\xi \|^2_{\Hc} \dd r + \frac18 \|A\chi(t)+\xi(t)\|_{\Hc}^2
+ c \|\uu\|_{L^\infty(0,T;\bsW)}^4+ c.
\end{aligned}
 \end{aligned}
 \]
We plug the above estimates for $I_{4}, \ldots, I_{8}$ into
\eqref{v-formal}, exploit the previously obtained bounds  and apply
the Gronwall Lemma. In this way, we conclude that
\begin{equation}
\label{seventh-esti} \|\teta_s\|_{H^1(0,T;\Hc)}+ \| f_\eps (\teta_s)
\|_{L^\infty(0,T;\Vc)}+ \| \chi \|_{H^1(0,T;\Vc)}+ \| A\chi+\xi
\|_{L^\infty(0,T;\Hc)} \leq c.
\end{equation}
We deduce moreover
\begin{equation}
\label{seventh-esti-bis} \| \chi\|_{L^\infty(0,T;H^2(\gc))} +\|\xi
\|_{L^\infty(0,T;\Hc)} \leq c,
\end{equation}
yielding (by a comparison in \eqref{eqIIa-irr-app})
\begin{equation}
\label{seventh-esti-tris}
\|\rho_\eps(\partial_t\chi)\|_{L^\infty(0,T;\Hc)}\leq c.
\end{equation}
\begin{remark}[A fully rigorous derivation of the a priori estimates]
\label{rmk:rigour} \upshape As already mentioned, the \emph{First}
and \emph{Fifth estimates} are not yet  rigorously justified in the
framework of the approximate Problem \ref{prob:irrev-app}. This is
due to a lack of regularity for the term
 $\partial_t \applogteta(\teta_\eps)$, which is only in $L^2 (0,\widehat
 T;\V')$.

 In order fully justify them, we should add a further
 viscosity contribution to the equation for $\teta$, modulated by a
 second parameter $\nu>0$, hence perform a double passage to the
 limit, first as $\nu \down 0$ and secondly as $\eps \down 0$.

  In the present paper we have chosen not to explore this, to avoid
  overburdening the analysis. We refer the reader to \cite{bbr3},
 where this
 procedure has been carried out in  detail.
\end{remark}


\section{\bf Local existence for Problem \ref{prob:irrev}}\label{s:5}
\label{ss:local-2}
\noindent
 In this section, we pass to the limit as $\eps
\down 0$ in Problem \ref{prob:irrev-app} and deduce
from the local existence result in Proposition \ref{prop:loc-exist-eps}
 the following
\begin{theorem}
\label{thm:exist-local}
Assume
\eqref{assumpt-domain} and
 Hypotheses \ref{hyp:1}--\ref{hyp:6}. Suppose
that the data $(h,\mathbf{f},\mathbf{g})$ and
$(\teta_0,\tetas^0,\uu_0,\chi_0)$ fulfill \eqref{hyp-data} and
\eqref{hyp-initial}.

Then, there exists $\widehat T>0$ such that  Problem
\ref{prob:irrev} admits an \bec \emph{energy solution}  $(\vartheta,
\vartheta_s, \uu,\chi,\eeta,\mmu, \xi, \zeta)$ on $(0,\widehat T)$
(in the sense of Definition \ref{def-energysol}), \eec  having in
addition the regularity properties \eqref{additional-regularities}
on $(0,\widehat T)$.
\end{theorem}
\begin{proof}
Let 
 $(\teta_\eps,\tetase, \uu_\eps,\chi_\eps,\mmu_\eps)_\eps$
 be a family of solutions to Problem $(P_\eps)$ with $\zz_\eps$ as in \eqref{to-be-quoted-also}; in what follows, we shall use the place-holders
 $\eeta_\eps:= \mathrm{D} \Phi_\eps (\uu_\eps)$ and $\xi_\eps = \beta_\eps (\chi_\eps)$.
  Relying on the (uniform w.r.t.\ $\eps$) a priori estimates
\eqref{dalpuntofissodue}, \eqref{first-aprio-teta}--\eqref{first-aprio-u-chi},
 \eqref{second-aprio-ln}--\eqref{second-aprio-teta-s},
 \eqref{third-aprio}--\eqref{fourth-aprio},
 \eqref{fifth-aprio-teta-s},
 \eqref{seventh-esti}--\eqref{seventh-esti-tris}, by weak and weak$^*$ compactness arguments we deduce that,
 along a suitable subsequence (which we do not relabel) \bec the
 following weak and weak$^*$ convergences hold
 \eec as $\eps\searrow0$
\begin{align}\label{convI}
&\uu_\eps\weakto \uu\quad\hbox{ in }H^1(0,\widehat
T;\WW),\\\label{convII} &\chi_\eps \weaksto\chi\quad\hbox{ in }
L^\infty(0,\widehat T;H^2(\gc)) \cap H^1(0,\widehat
T;\Vc),\\\label{convIII} &\teta_\eps\weakto \teta\quad\hbox{ in
}L^2(0,\widehat T;V),\\\label{convIV}
&\teta_{s,\eps}\weakto\tetas\quad\hbox{ in }L^2(0,\widehat
T;\Vc),\\\label{convV} &\xi_\eps\weaksto \xi\quad\hbox{ in
}L^\infty(0,\widehat T;\Hc),\\\label{convVI} &\eeta_\eps\weakto
\eeta\quad\hbox{ in }L^2(0,\widehat
T;\bsY_{{\Gammac}}'),\\\label{convVII} &\mmu_\eps\weaksto
\mmu\quad\hbox{ in }L^\infty(0,\widehat
T;L^{2+\nu}(\gc;\R^3)),\\\label{convVIII} &\zz_\eps\weaksto
\zz\quad\hbox{ in }L^\infty(\gc\times(0,\widehat
T)),\\\label{convIX}
&\applogteta(\teta_\eps)\weaksto\omega\quad\hbox{ in }
L^\infty(0,\widehat T;H) \cap H^1(0,\widehat T;V'),\\\label{convX}
&\applogtetas(\teta_{s,\eps})\weaksto \omega_s\quad\hbox{ in }
L^\infty(0,\widehat T;\Hc) \cap H^1(0,\widehat
T;\Vc'),\\\label{convXbis} &f_\eps(\teta_{s,\eps})\weakto
\alpha\quad\hbox{ in }L^2(0,\widehat T;\Vc),\\\label{convXtris}
&\bec \rho_\eps(\partial_t\chi)\weaksto \zeta\quad\hbox{ in
}L^\infty(0,\widehat T;\Hc). \eec 
\end{align}
The
 compactness results from \cite{simon} (cf.\ Thm.\ 4 and Cor.\ 5)  yield in addition
 the
following strong convergences
\begin{align}\label{convXI}
&\uu_\eps\rightarrow\uu\quad\hbox{in
}\rmC^0([0,\widehat T];H^{1-\delta}(\Omega;\R^3)) \quad \text{for all }\delta \in (0,1),\\\label{convXII}
&\chi_\eps\rightarrow\chi\quad\hbox{ in
}\rmC^0([0,\widehat T];H^{2-\delta}(\gc)) \quad \text{for all }\delta \in (0,2),
\\
&\teta_{s,\eps}\rightarrow\tetas\quad\hbox{in
}L^2(0,\widehat T;L^\delta(\gc)) \text{ for all } \delta\in[1,+\infty),
\label{convXVII}
\\\label{convXIII}
&\applogteta(\teta_\eps)\rightarrow\omega\quad\hbox{in
}\rmC^0([0,\widehat T];V'),\\\label{convXIV}
&\applogtetas(\teta_{s,\eps})\rightarrow \omega_s\quad\hbox{in
}\rmC^0([0,\widehat T];\Vc').
\end{align}
Note that by virtue of \eqref{convVI} and Hypothesis \ref{hyp:4} on
${\calR}$, it follows that
\begin{equation}\label{identificoR}
{\calR}(\eeta_\eps)\rightarrow\Reg(\eeta)\quad\hbox{ in }L^\infty(0,
\widehat T;L^{2+\nu}(\gc;\R^3)).
\end{equation}
Hence, combining \eqref{convVII} and \eqref{convVIII} {\ele (with \eqref{identificoR})} we end up with
\begin{equation}\label{identificoR-2}
\mmu=|{\calR}(\eeta)|\zz.
\end{equation}
Finally, we  combine estimate \eqref{first-aprio-teta} with \eqref{stimaBV}
and resort to the $\mathrm{BV}$-version of the Aubin-Lions compactnes theorem (cf.\ again \cite[Cor.\ 4]{simon}),
to deduce that
\begin{align}\label{convXV}
&\teta_\eps\rightarrow\teta\quad\hbox{in
}L^2(0,\widehat T;H^{1-\delta}(\Omega)) \quad \text{for all } \delta  \in (0,1), \\
\label{convXVI}
&\teta_\eps\rightarrow\teta\quad\hbox{in
}L^2(0,\widehat T;L^{4-\delta}(\gc)) \quad \text{for all }\delta \in (0,4)
\end{align}
the latter convergence ensuing from \eqref{convXV} via trace and
embedding theorems. {\ele Let us also point out that the  strong
convergences \eqref{convXI}, \eqref{convXII}, \eqref{convXVII}, and
\eqref{convXV} imply  (for suitable subsequences) a.e.\
convergence.} Now, by \eqref{convXVII} and  and \eqref{hyp-fc} we
can identify the limit of $\fc(\teta_\eps-\teta_{s,\eps})$ and
$\fc'(\teta_\eps-\teta_{s,\eps})$ in $L^2(0, \widehat
T;L^\delta(\gc))$ for any $\delta\in[1,4)$ and \bec $L^q(\gc \times
(0, \widehat T))$ \eec for all $q\in[1,+\infty)$, respectively \bec
(the latter convergence is guaranteed by   a generalization of the
Lebesgue theorem and by  \eqref{cond-g}). \eec
\paragraph{\bf Passage to the limit in the momentum equation.}
We can pass to the limit in \eqref{eqIa-app} as $\eps \to 0$ by virtue of the
previous convergences
and conclude that  the sextuple $(\teta,\tetas,\uu,\chi,\mmu,\eeta)$ satisfies \eqref{eqIa} on $(0,\widehat T)$,
with $\mmu= |\Reg(\eeta)| \zz$ (cf.\ \eqref{identificoR-2}).
It remains to identify $\zz$
as an element in $\partial \Psi (\uu_t)$ (thus obtaining  \eqref{incl1-bis} for $\mmu$),
and $\eeta$ (cf.\
\eqref{convVI})
as an element in $\partial \bvarphi (\uu)$.

For the latter task,
we proceed as in \cite{bbr6} and test
\eqref{eqIa-app} by $\uu_\eps$. By the previous convergences,
lower semicontinuity arguments,
and the fact that the limiting sextuple  $(\teta,\tetas,\uu,\chi,\mmu,\eeta)$ fulfills \eqref{eqIa},
 it is straightforward to check that
 \begin{equation}
 \label{lsc}
\limsup_{\eps\searrow0}\int_0^t\int_\gc\eeta_\eps\cdot\uu_\eps \dd x \dd r  \leq\int_0^t \pairing{}{\bsY_{{\Gammac}}}{\eeta}{\uu}   \dd r
\end{equation}
for all $t \in (0,\widehat T)$.
We
use \eqref{lsc} to deduce that for all $v \in \bsY_{{\Gammac}}$ there holds
\[
\begin{aligned}
\int_0^t \pairing{}{\bsY_{{\Gammac}}}{\eeta}{\vv -\uu}   \dd r\leq
\liminf_{\eps \to 0}  \int_0^t\int_\gc\eeta_\eps\cdot(\vv -\uu_\eps) \dd x \dd r
 & \leq  \liminf_{\eps \to 0}  \int_0^t\int_\gc\left( \Phi_\eps (\vv) - \Phi_\eps(\uu_\eps) \right)
 \dd x \dd r
 \\ & \leq  \int_0^t\int_\gc (\Phi(\vv) - \Phi(\uu))  \dd x \dd r  
 \\
  & = \int_0^t (\bvarphi(\vv) - \bvarphi(\uu)) \dd r
 \end{aligned}
\]
where the  third inequality is due to the fact that $\Phi_\eps$
Mosco converges to $\Phi$. We have thus shown that
\begin{equation}
\label{conseq}
\eeta\in\bvarphi(\uu) \quad \text{ in $\bsY_{{\Gammac}}'$ a.e.\ in $(0,\widehat T)$.}
\end{equation}

Instead of directly proving that
$\zz \in \partial\Psi (\uu_t)$ a.e.\ in $\gc \times (0,\widehat T)$, we will show that
\begin{equation}
\label{prove} \funeta{(\teta,\teta_s,\eeta)}(\ww) -
\funeta{(\teta,\teta_s,\eeta)}(\uu_t) \geq \int_0^{\widehat T}
\int_{\gc} \fc(\teta-\vartheta_s)|\Reg(\eeta) | \zz \cdot(\ww
-\uu_t) \dd x \dd t
\end{equation}
for all $\ww \in L^2 (0,{\widehat T};L^4(\gc;\R^3))$, where the functional
$\funeta{(\teta,\teta_s,\eeta)}: L^2 (0,{\widehat T};L^4(\gc;\R^3)) \to [0,+\infty)$
is defined for all $\vv \in L^2 (0,{\widehat T};L^4(\gc;\R^3))$ by
\begin{equation}
 \funeta{(\teta,\teta_s,\eeta)}(\vv): = \int_0^{\widehat T}
\int_{{\Gammac}} \fc(\teta(x,t)-\vartheta_s(x,t))|\Reg(\eeta)(x,t)|
\Psi(\vv (x,t)) \dd x \dd t.
\end{equation}
Clearly, $\funeta{(\teta,\tetas,\eeta)}$ is a convex and lower
semicontinuous functional on $L^2 (0,{\widehat T};L^4(\gc;\R^3))$.
 It can be easily verified
 that the subdifferential $\partial\funeta{(\teta,\teta_s,\eeta)} : L^2 (0,{\widehat T};L^4(\gc;\R^3)) \rightrightarrows L^2 (0,{\widehat T};L^{4/3}(\gc;\R^3))
 $ of $\funeta{(\teta,\tetas,\eeta)}$ is given  at every $\vv \in L^2
 (0,{\widehat T};L^4(\gc;\R^3))$ by
\begin{equation}
\label{repre-funeta} {\bf h} \in
\partial\funeta{(\teta,\teta_s,\eeta)}(\vv) \ \Leftrightarrow \quad
\begin{cases} {\bf h} \in  L^2
 (0,{\widehat T};L^{4/3}(\gc)),  \\ {\bf h}(x,t) \in \fc(\teta(x,t)-\vartheta_s(x,t))|\Reg(\eeta)(x,t)|
 \partial \Psi(\vv(x,t))\end{cases}
\end{equation}
for almost all $(x,t) \in {\Gammac} \times (0,{\widehat T})$.
Therefore, \eqref{prove} \bec will yield \eec
\[
\fc(\teta-\vartheta_s)|\Reg(\eeta) | \zz  \in \fc(\teta(x,t)-\vartheta_s(x,t))|\Reg(\eeta)(x,t)|
 \partial \Psi(\uu_t(x,t))
\]
whence \eqref{incl1-bis} for $\mmu=|\Reg(\eeta) | \zz$, also in view
of \eqref{hyp-fc}. In order to show \eqref{prove}, we first observe
that
\begin{equation}
\label{ehsicivuole} \limsup_{\eps\to 0} \int_0^{\widehat T} \int_{\gc} \fc
(\teta_\epsi-\teta_{s,\epsi})|\Reg(\eeta_\eps)| \zz_ \eps   \cdot
\partial_t \uu_\eps \dd x \dd t \leq  \int_0^{\widehat T} \int_{\gc} \fc
(\teta-\teta_{s})|\Reg(\eeta) | \zz \cdot  \uu_t \dd x \dd
t,
\end{equation}
which can be  checked   by testing \eqref{eqIa-app} by
$\partial_t\uu_\eps$ and passing to the limit via the above
convergences and  lower semicontinuity arguments, and again the
Mosco convergence of $\Phi_\eps$ to $\Phi$.
 Therefore, by \eqref{ehsicivuole},
\bec the previously obtained  \eec weak-strong convergence we have
\begin{equation}
\begin{aligned}
\label{prove1} \int_0^{\widehat T} \int_{\gc} \fc
(\teta-\teta_{s})|\Reg(\eeta) | \zz \cdot(\ww -\uu_t)  \dd x \dd t &
 \leq \liminf_{\eps \to 0}
\int_0^{\widehat T} \int_{\gc} \fc
(\teta_\eps-\teta_{s,\eps})|\Reg(\eeta_\eps) | \zz_\eps \cdot(\ww
-\partial_t\uu_\eps) \dd x \dd t
\\
 & \leq \liminf_{\eps \to 0} \int_0^{\widehat T} \int_{{\Gammac}}\fc
(\teta_\epsi-\teta_{s,\epsi}) |\Reg(\eeta_\eps)| ( \Psi(\ww) -
\Psi(\partial_t \uu_\eps))  \dd x \dd t   \\ & \leq \int_0^{\widehat
T} \int_{{\Gammac}} \fc (\teta-\teta_{s})|\Reg(\eeta)| ( \Psi(\ww) -
\Psi (\uu_t) )  \dd x \dd t.
\end{aligned}
\end{equation}
  Then,
\eqref{prove} ensues. Furthermore, arguing as in the derivation of
\eqref{prove1},
 a.e.\ in $\gc \times (0,{\widehat T})$,   we deduce
\begin{equation}
\label{ehsicivuole1} \liminf_{\eps\to 0} \int_0^{\widehat T} \int_{\gc} \fc
(\teta_\epsi-\teta_{s,\epsi})|\Reg(\eeta_\eps)| \zz_ \eps   \cdot
\partial_t \uu_\eps  \dd x \dd t \geq \int_0^{\widehat T} \int_{\gc} \fc
(\teta-\teta_{s})|\Reg(\eeta) | \zz \cdot  \uu_t  \dd x \dd
t \,.
\end{equation}
Ultimately, from \eqref{ehsicivuole} and \eqref{ehsicivuole1}, taking into account that
$\zz_\eps \cdot \partial_t \uu_\eps = \Psi(\partial_t \uu_\eps)$ and the same for $\zz$  and $\uu_t$,
 we
conclude
\begin{equation}
\label{ehsicivuole2} \lim_{\eps\to 0} \int_0^{\widehat T} \int_{\gc} \fc
(\teta_\epsi-\teta_{s,\epsi})|\Reg(\eeta_\eps)|  \Psi(\partial_t \uu_\eps)
 \dd x \dd t  = \int_0^{\widehat T} \int_{\gc} \fc
(\teta-\teta_{s})|\Reg(\eeta) | \Psi(\partial_t \uu)   \dd x \dd
t \,.
\end{equation}

We can develop a $\limsup$ argument completely analogous to the one leading to \eqref{lsc} and conclude that
\begin{equation}
\label{primo} \limsup_{\eps \to 0} \int_0^{\widehat T}
b(\partial_t\uu_\eps,\partial_t\uu_\eps)\dd t \leq \int_0^{\widehat T}
b(\partial_t\uu,\partial_t\uu) \dd  t
\end{equation}
so that it follows
\[
\lim_{\eps \to 0} \int_0^{\widehat T} b(\partial_t\uu_\eps,\partial_t\uu_\eps)
\dd t =\int_0^{\widehat T} b(\partial_t\uu,\partial_t\uu) \dd t.
\]
This gives, by the $\bf W$-ellipticity of $b$ (cf.\ \eqref{korn_a}),
the following strong convergence
\begin{equation}
\label{forteu-t}
 \partial_t\uu_\eps\rightarrow
\partial_t\uu\quad\text{in }L^2(0,{\widehat T};\bsW).
\end{equation}
\paragraph{\bf Passage to the limit in the equation for $\teta$.}
To pass to the limit in \eqref{teta-weak-app}
we combine the bi-Lipschitz continuity \eqref{cond-g}  of $g$ with  convergence
\eqref{convXV} to conclude that
\begin{equation}
\label{g-strong} g'(\teta_\eps)\rightarrow \fteta'(\teta)\quad\hbox{
in }L^q(\Omega\times(0,{\widehat T})) \quad \text{for all } q \in
(1,\infty).
\end{equation}
Taking into account that
$(\nabla g(\teta_\eps))_\eps$ is bounded in
$L^2(0,{\widehat T};H)$
(by \eqref{first-aprio-teta}
and \eqref{cond-g}), we therefore
conclude
\begin{equation}
\label{nabla-g-weak} \nabla g(\teta_\eps)\weakto \nabla
g(\teta)\quad\hbox{ in }L^2(0,{\widehat T};H).
\end{equation}
It
 follows from
convergence
 \eqref{convXII} for $\chi_\eps$,  \eqref{convXVII} for $\tetase$, \eqref{convXVI} for $\teta_\eps$, and
 \eqref{hyp-k} on $k$, that
 \begin{equation}
 \label{strong-k}
 k (\chi_\eps) (\teta_\eps - \tetase) \to k(\chi)(\teta-\tetas) \quad \text{in } L^2 (0,\widehat T; \Hc)\,.
 \end{equation}
Relying on the strong convergence
\eqref{forteu-t}
 of $\uu_\eps$
 and on the previously proved convergences for $\teta_\eps$,  $\tetase$, and $\eeta_\eps$,
  we show for
the frictional contribution that
{\ele \[
\lim_{\eps \to 0}
\int_0^{\widehat T}
\int_{\Gammac}
\fc'(\teta_\eps-\teta_{s,\eps}) |\Reg(\eeta_\eps )| \Psi(\partial_t \uu_\eps)  v  \dd x \dd t
=  \int_0^{\widehat T} \int_{\Gammac}
\fc'(\teta-\teta_{s}) |\Reg(\eeta)| \Psi( \uu_t)  v  \dd x \dd t  \quad \text{for all } v \in V.
\]}
Ultimately, on account of convergence \eqref{convIX} for $\applogteta(\teta_\eps)$ we conclude that
the functions $(\omega, \uu, \chi,\teta,\tetas,\eeta)$ fulfill the weak formulation \eqref{teta-weak}  of the equation for $\teta$.
It then remains to prove that
\begin{equation}
\label{omega-identif}
\omega =  \lteta(\teta) \quad \aein\, \Omega \times (0,\widehat T).
\end{equation}
This follows from the $\limsup$ inequality
\begin{equation}
\label{e:limsup} \limsup_{\eps \searrow 0} \int_0^T
\int_{\Omega}\applogteta (\teta_\eps) \teta_\eps \dd  x \dd t \leq
\int_0^T \int_{\Omega} \omega \teta \dd  x \dd t\,,
\end{equation}
(which in turn ensues from combining the weak convergence  \eqref{convIX} of  $\applogteta (\teta_\eps)$
and the strong \eqref{convXV} of $\teta_\eps$), taking into account that $\applogteta $ converges in the sense of graphs to $\lteta$.
In this way we conclude  that $( \uu, \chi,\teta,\tetas,\eeta)$  fulfill  \eqref{teta-weak}.
\paragraph{\bf Passage to the limit in the equation for $\tetas$.}
It proceeds exactly along the same lines as the limit passage  to  \eqref{teta-weak} (cf.\ also the proof of
\cite[Thm.\ 1]{bbr6}). Let us only comment on the proof of
\begin{equation}
\label{nabla-f-weak} \nabla \ftetas_\eps(\tetase)\weakto \nabla
\ftetas(\tetas)\quad\hbox{ in }L^2(0,{\widehat T};\Hc).
\end{equation}
Indeed, in view of convergence \eqref{convXVII} for $\tetase$ and of Lemma \eqref{l:pass-limeps}, we easily conclude, e.g., that
$\ftetas_\eps(\tetase) \to \ftetas(\tetas)$ in $L^1 (\Gammac \times (0,\widehat T))$, {\ele and thus a.e.}. This is enough to identify the weak limit  of
$\ftetas_\eps(\tetase)$ {\ele in \eqref{convXbis}}  and conclude \eqref{nabla-f-weak}.
\paragraph{\bf Passage to the limit in the equation for $\chi$.}
 Finally, we pass to the limit in
\eqref{eqIIa-irr-app} by virtue of  convergences
\eqref{convII}, \eqref{convV},
\eqref{convXtris}--\eqref{convXII},  and \eqref{convXVII}, also taking into account the properties
 \eqref{cond-landa-enhanc} of $\lambda$.
 By the strong-weak closedness of the graph of (the maximal monotone  operator induced by $\beta$ on $\hc$),
  we can directly
identify
$\xi$ as an element of $\beta(\chi) $ a.e.\ in $\Gammac \times (0,\widehat T)$. It remains to show that
 \begin{equation}
 \label{identif-zeta}
 \zeta \in \rho(\chi_t)  \quad \text{  a.e.\ in $\Gammac \times (0,\widehat T)$.}
\end{equation}
 To this
aim, we test \eqref{eqIIa-irr-app} by $\partial_t\chi_\eps$ and prove that
\begin{equation}
\label{limsup-rho}
\limsup_{\eps\searrow0}\int_0^t \int_\Omega\rho_\eps(\partial_t\chi_\eps)\partial_t\chi_\eps
\dd x \dd r
\leq\int_0^t\int_\Omega\zeta\partial_t\chi \dd x \dd r.
\end{equation}
Therefore, \eqref{identif-zeta} follows. As a byproduct, with  the same arguments as in the previous lines we have
\begin{equation}
\label{limite-rho}
\lim_{\eps\searrow0}\int_0^t \int_\Omega\rho_\eps(\partial_t\chi_\eps)\partial_t\chi_\eps
\dd x \dd r
=\int_0^t\int_\Omega\zeta\partial_t\chi \dd x \dd r.
\end{equation}
\paragraph{\bf Proof of the energy inequality \eqref{en-ineq-local}.} We take the  limit as $\eps\to 0$
in the \bec approximate energy inequality \eec \eqref{enid0}: let us
only justify the passage to the limit in some of the terms on the
left- and on the right-hand side. First of all, combining
convergence \eqref{convXV} for $(\teta_\eps)$ with \bec the upcoming
\eec  Lemma \ref{l:pass-lim-logtetaeps} we obtain that  for almost
all $t \in (0,\widehat T)$
\[
\lim_{\eps \to 0}\int_\Omega \calI_\eps (\teta_\eps(t)) \dd x =
\int_\Omega J^* (\lteta(\teta(t)))  \dd x\,.
\]
\bec The convergence for $t=0$ follows from  \berc condition \eqref{bounddatiteta-1lemma}. \eerc  \eec

 Analogously, we pass to the limit in the term
$\int_\Gammac {\it i}_\eps (\tetase(r)) \dd x $ for $r=t,\, s$ and for $r=0$. The weak convergence
\eqref{convIII} of $\nabla \teta_\eps$
and the strong convergence
 \eqref{g-strong} of $g'(\teta_\eps)$ allow us to conclude, {\ele via the Ioffe} theorem {\ele\cite{ioffe}},  that
 \[
 \liminf_{\eps \to 0} \int_s^t \int_\Omega g'(\teta_\eps)|\nabla \teta_\eps|^2 \dd x \dd r \geq \int_s^t \int_\Omega g'(\teta)|\nabla \teta|^2 \dd x \dd r.
 \]
 To take the limit in the term $\int_s^t \int_\Omega f_\eps'(\tetase)|\nabla \tetase|^2 \dd x \dd r =  \int_s^t \int_\Omega \nabla (f_\eps(\tetase))
 \cdot \nabla \tetase \dd x \dd r$ we proceed in a completely analogous way,
 taking into account Lemma
\ref{l:pass-limeps}.  The passage to the limit
 in the fifth term on the left-hand side of \eqref{enid0} results from \eqref{strong-k}, and for
 the term  $\int_s^t \int_{\Gammac} \fc(\teta_\eps-\tetase)  \Psi(\partial_t\uu_\eps)|{\calR}(\eeta_\eps)|  \dd x \dd r $  it follows from \eqref{ehsicivuole2}.
 We use   the strong convergences \eqref{convXVII} and \eqref{convXV} for $\tetase$ and $\teta_\eps$ combined with the properties
 \eqref{hyp-fc} of $\mathfrak{c}$, \eqref{identificoR} for $\Reg(\eeta_\eps)$,  the fact that $\Psi(\uu_t) \to \Psi(\uu)$ in $L^2 (0,T; \bsH_{{\Gammac}})$, and the Ioffe theorem,  to infer that
 \[
 \liminf_{\eps \to 0}   \int_s^t \int_{\Gammac} \fc'(\teta_\eps-\tetase)  \Psi(\partial_t\uu_\eps)|{\calR}(\eeta_\eps)|  \dd x \dd r
 \geq
  \int_s^t \int_{\Gammac} \fc'(\teta-\tetas)  \Psi(\uu_t)|{\calR}(\eeta)|  \dd x \dd r.
 \]
 Finally, observe that  for almost all $t \in (0,T)$
 \[
 \lim_{\eps \to 0} \int_\Gammac \widehat{\beta}_\eps(\chi_\eps(t)) \dd x = \int_\Gammac \widehat{\beta}(\chi(t)) \dd
 x.
 \]
 This can be checked by observing that, on the one hand, by Mosco
 convergence of  $\widehat{\beta}_\eps$ to $\widehat \beta$,
\begin{equation}\label{serveest}
 \liminf_{\eps \to 0} \int_\Gammac \widehat{\beta}_\eps(\chi_\eps(t)) \dd x \geq \int_\Gammac \widehat{\beta}(\chi(t)) \dd
 x.
\end{equation}
For the $\limsup$ inequality we use that
$\widehat{\beta}_\eps(\chi_\eps(t)) \leq \widehat{\beta}(\chi(t))
+\beta_\eps(\chi_\eps(t)) (\chi_\eps(t) - \chi(t)) $ a.e.\ in
$\Gamma_c$, and combine convergences \eqref{convV} and
\eqref{convXII}.
 For $t=0$ we have $\widehat{\beta}_\eps(\chi_0) \to \widehat{\beta}(\chi_0) $ in
$L^1(\Gammac)$ by the dominated convergence theorem.

 All the remaining terms on the left- and on the right-hand side of \eqref{enid0} can be dealt with exploiting the previously
 proved convergences.
This concludes the  proof \eqref{en-ineq-local} on  the interval
$(s,t)$ for almost all $s,\,t \in (0,\widehat T)$ and for $s=0$.
\end{proof}

\section{\bf Extension to a global-in-time solution and proof of Theorem \ref{thm:main}}
\label{s:6}

In this Section we show that \bec  the  local solution to   Problem \ref{prob:irrev}
found in Theorem \ref{thm:exist-local}  
(hereafter, we shall denote it by $(\widehat\teta,\widehat{\teta}_s,\widehat\uu,\widehat\chi, \widehat \eeta,\widehat\mu,\widehat \xi, \widehat \zeta)$),
actually extends from the interval $(0,\widehat T)$ to the whole  $(0,T)$. 

To this aim, we first of all  observe that the ``energy 
estimates'' (cf.\ the \emph{First estimate}) derived from the energy 
inequality
\eqref{en-ineq-local} have a \emph{global-in-time} character. 
Nonetheless, we cannot derive from such global bounds the other estimates (i.e.\ the \emph{Second}--\emph{Seventh estimates}), and therefore we cannot directly extend
the local solution   $(\widehat\teta,\widehat{\teta}_s,\widehat\uu,\widehat\chi)$, along with  $(\widehat \eeta,\widehat\mu,\widehat \xi, \widehat \zeta)$, to the whole interval $(0,T)$. The reason for this is that, as expounded in Sec.\ \ref{ss:3.1} and shown in Sec.\ \ref{s:4}, these estimates involve calculations which are only formal 
on the level of 
 the limit problem. Thus, we need to perform
them on the regularized system from Problem $(P_\eps)$. However, the energy 
estimates have only a local character for the approximate problem, since the term  \ber $\int_\Gammac\chi(t)|\uu(t)|^2 \dd x $ \edr is estimated  locally in time, cf.\  
the discussion at the beginning of Sec.\ \ref{s:4} and \eqref{dalpuntofisso}. 

Therefore, along the lines of the prolongation argument from \cite{bbr1}, 
 we shall proceed in the following way. We will extend the local solution 
  $(\widehat\teta,\widehat{\teta}_s,\widehat\uu,\widehat\chi)$ 
  together with its \emph{approximability properties} (cf.\ the notion of \emph{approximable solution} in Definition \ref{defappross} below). In this way, the approximate solutions
  shall ``inherit'' the global-in-time energy 
estimates from $(\widehat\teta,\widehat{\teta}_s,\widehat\uu,\widehat\chi)$ (cf.\ \eqref{dacitare} and \eqref{fondamentale}). Building on this, 
we will be able to perform rigorously all the estimates necessary for the extension procedure  on the approximate level,  and use them to conclude that 
 $(\widehat\teta,\widehat{\teta}_s,\widehat\uu,\widehat\chi, \widehat \eeta,\widehat\mu,\widehat \xi, \widehat \zeta)$ is defined on  the whole interval $(0,T)$.
 More precisely,   we will consider the maximal extension of  our (approximable) solution
 $(\widehat\teta,\widehat{\teta}_s,\widehat\uu,\widehat\chi, \widehat \eeta,\widehat\mu,\widehat \xi, \widehat \zeta)$
  and show that it is  defined  $(0,T)$ with a  standard contradiction argument (cf.\ Step $4$ below). 
  In doing so, we will meet with the technical difficulty that the $(\teta,\tetas)$-components of our solution need not be continuous w.r.t.\ time and therefore we will not
  be in the position to extend them by continuity. Indeed,  in accord with the notion of \emph{approximable} solution,  we will argue on the level of  the approximate solutions and rely 
  on their time-regularity to carry out this procedure rigorously. 
\eec


\bec In what follows, $(\teta_\eps,\tetase, \uu_\eps,\chi_\eps)_\eps$ 
(with associated $\mmu_\eps$)
will be the family of solutions to Problem $(P_\eps)$ which converge, along a not-relabeled subsequence (cf.\ \eqref{convI}--\eqref{convXVI}),
to the local solution $(\widehat\teta,\widehat{\teta}_s,\widehat\uu,\widehat\chi)$ from Theorem 
\ref{thm:exist-local}. 
For simplicity, hereafter we shall omit the functions  $(\widehat \eeta,\widehat\mu,\widehat \xi, \widehat \zeta)$ and refer to the quadruple 
$(\widehat\teta,\widehat{\teta}_s,\widehat\uu,\widehat\chi)$  as ``the'' local solution to our problem. Accordingly, we will give the definition of 
approximable solution only in terms of the $(\teta,\tetas,\uu,\chi)$-components.  \eec 
We are now in the position to introduce the notion of ``approximable solution''.

\begin{definition}\label{defappross}
Let $\tau\in(0,T]$. We say that a quadruple $(\tetat,\tetast,\uutt,\chitt)$  is an \emph{approximable solution} on $(0,\tau)$ 
to Problem \ref{prob:irrev}
if the following conditions are verified
\begin{itemize}
\item it is an energy solution on $(0,\tau)$;
\item there exists a subsequence $\varepsilon_n \downarrow 0 $ such that the related solutions of  problem $(P_{\varepsilon_n})$  on $(0,\tau)$  (here the dependence on $\tau$ is omitted in the notation) fulfill  as $n\rightarrow\infty$
\begin{align}\label{c1est}
&\|\uu_{\varepsilon_n}-\uu\|_{\rmC^0([0,\tau];H^{1-\delta}(\Omega))} \ber \rightarrow 0 \edr \qquad \text{for all } \delta \in (0,1],\\\label{c2est}
&\|\chi_{\varepsilon_n}-\chi\|_{\rmC^0([0,\tau];H^{1-\delta}(\Gammac))} \ber \rightarrow 0 \edr  \qquad \text{for all } \delta \in (0,1],\\\label{c3est}
&\|\teta_{\varepsilon_n}-\teta\|_{L^2(0,\tau;L^2(\Omega))} \ber \rightarrow 0 \edr,\\\label{c4est}
&\|{\tetas}_{\varepsilon_n}-\tetas\|_{L^2(0,\tau;L^2(\Gammac))} \ber \rightarrow 0 \edr.
\end{align}
\end{itemize}
\end{definition}

Note in particular that, by virtue of the above definition and  \bec convergences \eec   \eqref{convXI}, \eqref{convXII}, \eqref{convXV}, \eqref{convXVI},
\bec for $\tau\geq\widehat T$ the quadruple \eec \ber $(\tetat,\tetast,\uutt,\chitt)$ is a \edr proper
extension on $(0,\tau)$ of the local solution 
$(\widehat\teta,\widehat{\teta}_s,\widehat\uu,\widehat\chi)$.
 More precisely, we have 
{\ele
\begin{equation}
\label{id1}
(\uutt,\chitt)=(\widehat\uu,\widehat\chi)\quad\hbox{ for all }t\in[0,\widehat T],
\qquad \qquad 
(\tetat,\tetast)=(\widehat\teta,\widehat{\teta}_s) \quad\hbox{ for a.a. }t\in(0,\widehat T).
\end{equation} }
We introduce the set
$$
{\mathcal T}:=\{\tau\in(0,T]\hbox{ such that there exists an approximable solution on }(0,\tau)\}.
$$
\bec It follows from the passage to the limit argument in  Theorem \ref{thm:exist-local} \eec
 that ${\mathcal T}$ is not an empty set, as at least $\widehat T\in{\mathcal T}$. As a consequence, letting $$T^*=\sup {\mathcal T}$$ we have $0\ber <\edr T^*\leq T$. Hence, {\ele proving that the local solution
$(\widehat\teta,\widehat{\teta}_s,\widehat\uu,\widehat\chi)$  extends to the whole $(0,T)$ \bec reduces to showing \eec that it extends to an approximable solution on $(0,T^*)$, and that } 
 $T^*=T$. To this aim let us outline the sketch of the proof. 
 \bec
 First, we prove that  the 
``energy estimates'' for  an approximable solution hold with a constant independent \eec of $\tau$ (Step $1$). Then, we deduce that an approximable solution  extends to $(0,T^*)$ (Steps $2$ and $3$). 
 Finally, we show that $T=T^*$ by a contradiction argument (Step $4$).
\bigskip

\noindent{\bf Step $1$.}\quad Let us prove the following  Lemma, stating global estimates on the energy solutions to {\ele Problem \ref{prob:irrev}}. Observe that the constant $C$ below does not depend on $\tau$.
\begin{lemma}\label{lemma1est}
\bec Assume
\eqref{assumpt-domain} and
 Hypotheses \ref{hyp:1}--\ref{hyp:6}. Suppose
that the data $(h,\mathbf{f},\mathbf{g})$ and
$(\teta_0,\tetas^0,\uu_0,\chi_0)$ fulfill \eqref{hyp-data} and
\eqref{hyp-initial}.

Then, there exists a constant \eec $C>0$ depending on the data of the problem such that for any $\tau>0$ and for any energy solution $(\teta,\tetas,\uu,\chi)$ to  {\ele Problem \ref{prob:irrev}} on $(0,\tau)$, there holds
\begin{align}\label{eqlemmaI}
\|\uu\|_{H^1(0,\tau;\WW)}+\|\chi\|_{L^\infty(0,\tau;H^1(\Gammac)\cap H^1(0,\tau;L^2(\Gammac))}+\|\teta\|_{L^2(0,\tau;V)\cap L^\infty(0,\tau;L^1(\Omega))}+\|\tetas\|_{L^\infty(0,\tau;L^1(\Gammac))}\leq C.
\end{align}
\end{lemma}

The proof directly follows from the energy inequality \eqref{en-ineq-local}, written on $(0,\tau)$: \bec we develop the very same calculations as in the derivation of the 
\emph{First estimate},  \eec  after observing that here 
  $\int_\Gammac\chi(t)|\uu(t)|^2 \dd x \geq  0$, as $\chi\in \hbox{dom }\widehat\beta$, 
and thus $\chi\geq0$, a.e.\  on 
$\Gammac$,
 e.g.\ due to \eqref{serveest}.
 \bigskip
 
 \noindent{\bf Step $2$.}\quad  Let $(\tetat,\tetast, \uutt,\chitt)_\tau$ be a family of approximable solutions  depending on $\tau$, with $\tau\in{\mathcal T}$. \bec In view of the 
 regularity required of 
 \eec an approximable solution,
 \bec we have that \eec 
 $
 (\uutt,\chitt)\in \rmC^0([0,\tau];\WW)\times \rmC^0([0,\tau];H^1(\Gammac)).
 $
Hence we can consider the extension to $(0,T^*)$ by continuity of these functions. More precisely, we define
\[
 \begin{aligned}
 \widetildeuutt(t):=
 \begin{cases}
 \uutt{\ele(t)} & \hbox{ if }t\in[0,\tau],
 \\  \uutt(\tau) & \hbox{ if }t\in(\tau,T^*\ber ]\edr,
 \end{cases}
 \qquad \qquad 
 \widetildechitt(t):=
\begin{cases}
 {\ele \chitt(t)} & \hbox{ if }t\in[0,\tau],
 \\ \chitt(\tau) & \hbox{ if }t\in(\tau,T^*\ber ]\edr.
 \end{cases}
 \end{aligned}
 \]
Due to \eqref{eqlemmaI} (where the constant $C$ does not depend on $\tau$) there holds (independently of $\tau$)
\begin{equation}\label{stimapprossimable}
\|\widetildeuutt\|_{H^1(0,T^*;\WW)}+\|\widetildechitt\|_{H^1(0,T^*;L^2(\Gammac))\cap L^\infty(0,T^*;H^1(\Gammac))}\leq C.
\end{equation}
Thus, after fixing a sequence $\tau_m\uparrow T^*$, by (weak, weak$^*$, and strong) compactness results
we can conclude that there exists  a pair
$$
(\uu^*,\chi^*)\in H^1(0,T^*;\WW)\times H^1(0,T^*;L^2(\Gammac))\cap L^\infty(0,T^*;H^1(\Gammac))
$$
such that, at least along some not relabeled subsequence,
\begin{equation}\label{estII:lim}
\|\widetildeuutm-\uu^*\|_{\rmC^0([0,T^*];H^{1-\delta}(\Omega;\R^3))}+\|\widetildechitm-\chi^*\|_{\rmC^0([0,T^*];H^{1-\delta}(\Gammac))}\rightarrow0.
\end{equation}
Hence, by construction {\ele of $\widetildeuutt$ and $\widetildechitt$}, we can infer that $(\uu^*(t),\chi^*(t))=(\widehat\uu(t),\widehat\chi(t))$ for all $t\in[0,\widehat T]$ {\ele (see \eqref{id1})}.

\bigskip

\noindent{\bf Step $3$.}\quad
\ber
Now, we will prove that there exists $(\teta^*,\tetas^*)$ such that $(\teta^*,\tetas^*,\uu^*,\chi^*)$ is an approximable solution to Problem \ref{prob:irrev} on $(0,T^*)$.
To this aim\edr, let $\tau_m\uparrow T^*$ and  \eqref{estII:lim}  hold. By definition of approximable solution (see Def. \ref{defappross}),
for any $m\in \N$ there exists a (sub)sequence $(\varepsilon_n^{\tau_m})_n$ such that   $\varepsilon_n^{\tau_m}\down 0$ and 
the corresponding approximating sequence of solutions to Problem $(P_{\varepsilon_n^{\tau_m}})$ on $(0,\tau_m)$ satisfies
\begin{equation}
\|\uu_{\varepsilon_n^{\tau_m}}^{\tau_m}-\uutm\|_{\rmC^0([0,\tau_m];H^{1-\delta}(\Omega;\R^3))}+\|\chi_{\varepsilon_n^{\tau_m}}^{\tau_m}-\chitm\|_{\rmC^0([0,\tau_m];H^{1-\delta}(\Gammac))}\rightarrow0.
\end{equation}
Thus, by diagonalization, we find a further subsequence, 
\bec which we will denote by  \eec
$\varepsilon_m$, and,  correspondingly,  \bec a sequence \eec   $(\tetatem,\tetastem,\uutem,\chitem)$ of 
 solutions to Problem $(P_{\varepsilon_m})$ on $(0,\tau_m)$, such that for every  $m \in \N$ 
\begin{equation}\label{eststimaII}
\|\uutem-\uutm\|_{\rmC^0([0,\tau_m];H^{1-\delta}(\Omega;\R^3))}+\|\chitem-\chitm\|_{\rmC^0([0,\tau_m];H^{1-\delta}(\Gammac))}\leq\frac 1 m.
\end{equation}
\bec Ultimately, we have for some $ m^*\in\N$ that  \eec
\begin{equation}
\label{dacitare}
\|\uutem-\uu^*\|_{\rmC^0([0,\tau_m];H^{1-\delta}(\Omega;\R^3))}+\|\chitem-\chi^*\|_{\rmC^0([0,\tau_m];H^{1-\delta}(\Gammac))}\leq\frac 2 m \quad\forall m\geq m^*.
\end{equation}
We exploit \ber \eqref{eqlemmaI} \edr \bec and \eqref{dacitare}, combined with  trace theorems, \eec to deduce
\begin{equation}\label{fondamentale}
\int_\Gammac\chitem(t)|\uutem(t)|^2 \dd x \leq c
\end{equation}
independently of $\tau_m$. 
Now, we use
\eqref{fondamentale} in  \bec the approximate energy identity \eqref{enid0} \eec
 and get the analogue of estimates \eqref{eqlemmaI} 
  for the \bec approximate solutions, with a constant  independent \eec  of $\tau_m$. As a consequence we can  
  perform  the \bec same  a priori estimates as in Section \ref{s:4}, and we can now conclude that 
  they hold \emph{globally in time}.  \eec
In particular, we get  that  \eqref{dalpuntofissodue}, \eqref{first-aprio-teta}--\eqref{first-aprio-u-chi},
 \eqref{second-aprio-ln}--\eqref{second-aprio-teta-s},
 \eqref{third-aprio}--\eqref{fourth-aprio},
 \eqref{fifth-aprio-teta-s},
 \eqref{seventh-esti}--\eqref{seventh-esti-tris} hold  on $[0,\tau_m]$, \bec uniformly w.r.t.\ $m \in \N$. \eec

\bec We now
extend the $(\teta,\tetas)$-components of the solution (along with $(\eeta,\mmu,\xi,\zeta)$),  on the interval $(0,T^*)$, together with their approximability properties. To this aim, we proceed with a diagonalization argument, \bec which we sketch here for the sake of completeness, referring to \cite{bbr1} for all details. \eec 
Let us  take $T_k:=T^*(2^k-1)/2^k$, $k\geq 1$. For $k=1$ and  for $m$ sufficiently large $m\geq \bar m_1$ we have $[0,T_1]\subset[0,\tau_m]$. As a consequence we can apply  compactness arguments \bec analogous to the ones \eec in the previous sections. In particular, we show that for some suitable subsequence
\bec $(m_j^1)_j$ with $m_j^1 \to \infty$ and 
  $m_j^1\geq \bar m_1$  for all $j \in \N$, \eec 
 the analogues of  \eec \eqref{convI}--\eqref{convXIV}  hold on $[0,T_1]$.  In particular, 
 we deduce
\begin{align}
&\tetatemj\rightarrow\teta^*\quad\hbox{in }L^2(0,T_1;H),\\
&\tetastemj\rightarrow\tetas^*\quad\hbox{in }L^2(0,T_1;L^2(\Gammac)),
\end{align}
\bec (as well as the existence of a limit quadruple $(\eeta^*,\mmu^*,\xi^*,\zeta^*)$). \eec  
As a consequence it is straightforward to observe that  $\teta^*,\tetas^*$ can be identified (a.e.) with $\widehat\teta,\widehat{\teta}_s$ on $(0,\widehat T)$ {\ele (see 
the second of
\eqref{id1})}.
The \bec aforementioned \eec  convergences  allow us to apply a similar passage to the limit procedure 
in the approximate problem, and conclude that $(\teta^*,\tetas^*,\uu^*,\chi^*)$ is a solution on $(0,T_1)$ to {\ele Problem \ref{prob:irrev}} (we omit details as they follow the already detailed argument  in the proof  Theorem \ref{thm:exist-local}). We  can now proceed \bec repeating \eec  the argument for $T_k$ with $k=2$, and 
  extending the above convergences  to  the interval $[0,T_2]$ along a subsequence   $(m_j^2)_j$ larger than some $\bar m_2\geq \bar m_1$. 
  \bec Iterating this construction for any $k \in \N$ (cf.\ \cite[pag. 1061]{bbr1} for details), 
 we get that the limit functions $(\uu^*,\chi^*,\teta^*,\tetas^*)$ solve the limit {\ele Problem \ref{prob:irrev}} on the interval $(0,T^*)$ (along with some  $(\eeta^*,\mmu^*,\xi^*,\zeta^*)$),
 and conclude indeed that $(\uu^*,\chi^*,\teta^*,\tetas^*)$ is an approximable solution on  $(0,T^*)$. \eec 

\bigskip

\noindent{\bf Step $4.$}\quad We now prove that 
$T=T^*$, hence that
$(\teta^*,\tetas^*,\uu^*,\chi^*)$ is \ber an approximable \edr solution to {\ele Problem \ref{prob:irrev}} on the whole $(0,T)$.
 We proceed by contradiction, assuming $T^*<T$ and show that  actually $(\teta^*,\tetas^*,\uu^*,\chi^*)$ can be extended to some approximable solution $(\widetilde\teta,\widetildetetas,\widetilde\uu,\widetilde\chi)$ on $(0,T^*+\delta)$, with $\delta>0$. Indeed, let us consider the sequence $(\tetatem,\tetastem,\uutem,\chitem)$ of the approximate solutions solving
 Problem
  $(P_{\varepsilon_m})$ on $[0,\tau_m]$ constructed above. We can extend it to some larger interval $[0,\tau_m+\delta]$ with $\delta>0$ by applying  our local existence result,
 \bec Proposition \ref{prop:loc-exist-eps}, to Problem $(P_{\eps_m})$, supplemented \eec
   with the  initial data $(\tetatem(\tau_m),\tetastem(\tau_m),\uutem(\tau_m),\chitem(\tau_m))$ 
 \bec   (observe that $\tetatem(\tau_m)$ and $\tetastem(\tau_m)$ are well defined by  
    the time-regularity \eqref{reg-teta-app} and \eqref{reg-teta-s-app} of the approximate solutions). Since the local existence time does $\delta$
    does not depend on $m$, we conclude that there exists a local solution to  $(P_{\eps_m})$ with the aforementioned initial data on \eec 
     on $[\tau_m,\tau_m+\delta]$. Now, we have that $T^*+\delta/2<\tau_m+\delta$ for $m$ sufficiently large (as $\tau_m\rightarrow T^*$),
       therefore 
       \bec $(\tetatem,\tetastem,\uutem,\chitem)$ turns out to be extended on the interval 
        $(0,T^*+\delta/2)\subset(0,\tau_m+\delta)$. Proceeding as in Step 3, by compactness and passage to the limit procedures, we can prove that
        {\ele Problem \ref{prob:irrev}} admits an approximable solution
    $(\widetilde\teta,\widetildetetas,\widetilde\uu,\widetilde\chi)$   on $(0,T^*+\delta/2)$, against the definition $T^*$. Therefore, $T^*=T$, which concludes the proof of Theorem
     \ref{thm:main}.  \QED \eec


\appendix
\section{}
We develop \bec here \eec  a series of technical results,
collecting useful properties and  estimates for the functions $\applogteta$,
$\applogtetas$, $f_\eps$, and related quantities, which {\ele play} a
crucial role in deriving a priori estimates for Problem $(P_\eps)$.
We also detail the construction of a  family of approximate initial data complying with the properties \eqref{datitetaapp}--\eqref{bounddatiteta-4}.

In what follows, we shall rely on  some   definitions and results
from the theory of maximal monotone operators, for which we refer
 to the classical monographs \cite{barbu76, brezis73}.
 Preliminarily, we fix some
notation.
\begin{notation}
\upshape
\bec Hereafter,  \eec
 for fixed $\eps>0$ we \bec will \eec  denote by
\begin{equation}
\label{e:max-mon-1} \resteta:= (\Id + \eps \lteta)^ {-1}: \R \to \R,
\qquad \restetas: = (\Id + \eps \ltetas)^ {-1}: \R \to \R
\end{equation}
 the resolvent operators
associated with $\lteta $ and $\ltetas$, respectively. We recall
that $ \resteta $  and  $\restetas $ are   contractions, and that
 the Yosida regularizations of  $\lteta$ and $\ltetas$ are  defined,
 respectively,
by
\begin{equation}
\label{e:max-mon-2} \lteta_\eps:=\frac{\Id -\resteta}{\eps}: \R \to
\R\,,\qquad  \ltetas_\eps:=\frac{\Id -\restetas}{\eps}: \R \to \R\,
\end{equation}
and fulfill
\begin{equation}
\label{e:max-mon-3-bis} \lteta_\eps (x) = \lteta (\resteta (x)),
\qquad \ltetas_\eps (x) = \ltetas (\restetas (x))
  \qquad
\forall\, x \in \R\,.
\end{equation}
We also introduce the Yosida approximations $J_\eps$ and $j_\eps$ of
the primitives $J$ and $j$ of $\lteta$ and $\ltetas$, respectively
defined for all $\eps>0$ by
\begin{equation}
\label{jeps} J_\eps(x):= \min_{y \in \R} \left\{
\frac{|y-x|^2}{2\eps} + J(y) \right\},  \qquad
 j_\eps (x):= \min_{y \in \R} \left\{
\frac{|y-x|^2}{2\eps} + j(y) \right\} \quad \forall\, x \in \R\,.
\end{equation}
By~\cite[Prop.~2.11]{brezis73},    for all $\eps >0$
\bec the functions \eec
 $J_\eps$ and
$j_\eps$ are convex and differentiable on $\R$, with $J_\eps'(x) =
\lteta_\eps(x)$ and  $j_\eps'(x) = \ltetas_\eps(x)$ for all $x \in
\R$,  and the following identities hold
(see~\cite[Prop.~2.11]{brezis73})
\begin{equation}
\label{e:max-mon-5} J_\eps (x) = \frac{\eps}2 |\lteta_\eps(x)|^2 +
{\ele J}(\resteta (x)) \quad \forall\, x \in \R, \qquad
 j_\eps (x) = \frac{\eps}2 |\ltetas_\eps(x)|^2 +
j(\restetas (x)) \quad \forall\, x \in \R\,.
\end{equation}
 Moreover, the Fenchel-Moreau convex conjugates $
J^*_\eps:=(J_\eps)^* $ of $J_\eps$ and $ j^*_\eps:=(j_\eps)^* $ of
$j_\eps$
 (which  in general differ
from the Yosida approximations of $J^*$ and $j^*$, respectively),
 fulfill
\begin{align}
&
 \label{convex-analysis-1}
\lteta_\eps^{-1} \equiv  {J^*_\eps}'\,, \qquad \ltetas_\eps^{-1}
\equiv  {j^*_\eps}'
\\
&
 \label{convex-analysis-2}
 J_\eps (x) + J^*_\eps(\lteta_\eps (x)) = x \lteta_\eps (x),
 \qquad
j_\eps (x) + j^*_\eps(\ltetas_\eps (x)) = x \ltetas_\eps (x) \qquad
\forall\, x \in \R\,,
\\
&
 \label{convex-analysis-3}
 J^* _\eps (y)= J^* (y) + \frac{\eps}{2} |y|^2 \qquad \forall\, y \in
\R\,, \qquad  j^* _\eps (y)= j^* (y) + \frac{\eps}{2} |y|^2 \qquad
\forall\, y \in \R\,,
\end{align}
(we refer to, e.g.,~\cite[Prop.~3.3, p.~266]{Attouch} for the proof
of the latter relation).
\end{notation}

In the first lemma we prove that the coercivity properties
\eqref{cond-L3} and \eqref{cond-ell4} transfer to the approximate
level and that, consequently, the functions $\mathcal{I}_\eps,
{\it i}_\eps: \R \to \R$  {\ele (see \eqref{mathcal-i-eps}--\eqref{ipiccolo-eps}) }satisfy suitable growth conditions.
\begin{lemma}
\label{l:new-lemma1} Assume \eqref{cond-L1} and \eqref{cond-L3} on $\lteta$
 and \eqref{cond-ell1} and \eqref{cond-ell4} on $\ltetas$. Then,
\begin{subequations}
\label{ci-serve-anticipata}
\begin{align}
& \label{interesting-relation-teta-e-s}
 \mathcal{I}_\eps(x) = \frac \eps2 x^2 +J_\eps^* (\lteta_\eps (x)) + J_\eps (0) \quad \text{ and } \quad
 {\it i}_\eps(x) =  \frac \eps2 x^2 +j_\eps^* (\ltetas_\eps (x))+j_\eps (0)
 \qquad \text{for all } x \in \R\,.
\end{align}
As a consequence,
\begin{align}
 &
\label{ci-serve-anticipata-1}
 \exists\, C_1^*, \,  C_2^*  >0\, \
\forall\, \eps \in (0,1), \ \forall\, x \in \R\, :
\quad
 \mathcal{I}_\eps(x) \geq \frac{\eps} {2} x^2+ C_1^* |x|
 - C_2^*\,,
\\
& \label{ci-serve-anticipata-2}
  \exists\, c_1^*, \,  c_2^*  >0\, \
\forall\, \eps \in (0,1), \ \forall\, x \in \R\, :
\quad {\it i}_\eps(x) \geq \frac{\eps} {2} x^2+ c_1^* |x|
 - c_2^*\,.
\end{align}
\end{subequations}
\end{lemma}
\begin{proof}
We  develop the proof of the first of
\eqref{interesting-relation-teta-e-s}, and of
\eqref{ci-serve-anticipata-1}, only, since the arguments for the
second of \eqref{interesting-relation-teta-e-s}
 and for \eqref{ci-serve-anticipata-2} are identical.
 Integrating by parts, we have
\begin{equation}
\label{ineq3}
\begin{aligned}
 \mathcal{I}_\eps(x)  =
 x \applogteta(x)-\int_0^x \applogteta(s) \dd s  =
 \frac{\eps}{2}x^2 + x \lteta_\eps(x) -\int_0^x \lteta_\eps(s) \dd s
 &  =   \frac{\eps}{2}x^2 + x \lteta_\eps(x) -J_\eps(x) +J_\eps (0)
 \\ &  = \frac{\eps}{2}x^2 +J_\eps^* (\lteta_\eps(x)) +J_\eps (0)
\end{aligned}
\end{equation}
where the second identity follows from the fact that $\applogteta(x) = \eps x + \lteta_\eps(x)$, the third one
 from the fact that $J_\eps$ is a primitive of
 $\lteta_\eps$,
 and  the
last one from \eqref{convex-analysis-2}.
Next,
 we show that
 \[
  \exists\, C_1^*, \,  C_2^*  >0\, \
\forall\, \eps \in (0,1), \ \forall\, x \in \R\, :
\quad
J^*_\eps(\lteta_\eps (x)) \geq C_1^* |x| -C_2^*\,.
 \]
 Indeed,
  we use \eqref{convex-analysis-3},
\eqref{e:max-mon-3-bis}, and \eqref{cond-L3} to infer that
\begin{equation}
\label{e-stima-2} J^*_\eps(\lteta_\eps (x)) = J^* (\lteta_\eps (x))
+ \frac{\eps}{2} |\lteta_\eps (x)|^2 \geq   \frac{\eps}{2}
|\lteta_\eps (x)|^2 + C_1 |\resteta (x)| -C_2 \ \qquad \forall\, x
\in \R\,.
\end{equation}
On the other hand,  due to the definition \eqref{e:max-mon-2}  of
$\lteta_\eps$,
\begin{equation}
\label{e-stima-3} \ber C_1 |\resteta (x) - x| = C_1 |\eps \lteta_\eps (x)| \leq C_1^2
+\frac{\eps}4 |\lteta_\eps (x)|^2 \edr \qquad \forall\, x \in \R\,, \
\eps \in (0,1)\,.
\end{equation}
Therefore, collecting \eqref{e-stima-2}--\eqref{e-stima-3} we
conclude that for every $\eps \in (0,1)$
$$
J^*_\eps(\lteta_\eps (x))   \geq  \ber C_1 |x| - C_1^2
-C_2 \edr \qquad \forall\, x \in \R\,.
$$
 Finally,
  in view of \eqref{e:max-mon-5} we have
 \[
 J_\eps (0) \geq J(\resteta (0)) \geq -C |\resteta (0)| -C' \geq -C
 \]
 for a constant independent of $\eps$: this follows from the fact that the convex function $J$ is bounded
 from below by a linear function, and from $\ber \resteta \edr (0) \to 0$ as $\eps \to 0$, since $0 \in \overline{D(\lteta)}$.
 \end{proof}
Our next result {\ele is} crucial for the passage to the limit as $\eps \to 0$, in particular
to obtain the energy inequality \eqref{en-ineq-local}.
\begin{lemma}
\label{l:pass-lim-logtetaeps}
 Assume \eqref{cond-L1} and \eqref{cond-L3} on $\lteta$. Let $(\theta_\eps)_\eps \subset H$
 fulfill
 \begin{equation}
 \label{convergence-bound}
\theta_\eps \to \theta \quad \text{ in } H, \qquad \sup_\eps \| \lteta_\eps (\theta_\eps)\|_H \leq C.
 \end{equation}
Then,
\begin{equation}
\label{convergence-ltheta}
\lim_{\eps \to 0} \int_\Omega \calI_\eps (\theta_\eps(x)) \dd x = \int_\Omega J^* \ber (\lteta(\theta (x))) \edr \dd x\,.
\end{equation}
Under conditions \eqref{cond-ell1} and \eqref{cond-ell4} on $\ltetas$, the analogue of \eqref{convergence-ltheta} holds for $({\it i}_\eps)_\eps$.
\end{lemma}
\begin{proof}
It follows from the second of \eqref{convergence-bound} that,
for every sequence $(\eps_n) \down 0$ there exist a (not relabeled) subsequence
$(\theta_{\eps_n})$ and $\omega \in H$ such that
$\lteta_{\eps_n} (\theta_{\eps_n}) \weakto \omega$ in $H$.
Therefore
\[
\limsup_{n \to \infty} \int_\Omega L_{\eps_n} (\theta_{\eps_n}) \theta_{\eps_n} \dd x \leq
\int_\Omega \omega \theta \dd x,
\]
which yields $\omega = \lteta(\theta)$ thanks to \cite[Lemma 1.3, p.\ 42]{barbu76}.
Since the limit does not depend on the subsequence, we ultimately conclude that
\begin{equation}
\label{weak-leps}
\lteta_\eps (\theta_\eps) \weakto \lteta(\teta) \qquad \text{in } H \text{ as } \eps \to 0.
\end{equation}
On the one hand,
\[
\begin{aligned}
\liminf_{\eps \to 0} \int_\Omega \calI_\eps (\theta_\eps(x)) \dd x
 & = \liminf_{\eps \to 0} \int_\Omega \left(
\frac \eps2 |\theta_\eps(x)|^2 +J_\eps^* (\lteta_\eps (\theta_\eps(x)))  + J_\eps (0) \right) \dd x
 \\ & =   \liminf_{\eps \to 0} \int_\Omega  J_\eps^* (\lteta_\eps (\theta_\eps(x))) \dd x
 \\ & \geq   \liminf_{\eps \to 0}
 \int_\Omega J^* (\lteta_\eps (\theta_\eps(x))) \dd x
 \geq \int_\Omega J^* (\lteta (\theta(x))) \dd x
 \end{aligned}
\]
where the first identity follows from \eqref{interesting-relation-teta-e-s}, the second one from \eqref{convergence-bound}  and the fact that $J_\eps (0) \to J(0)=0$
due to \eqref{not-restrictive-prima},
the third inequality is due to \eqref{convex-analysis-3}, and the last one to the weak convergence
\eqref{weak-leps} combined with the Ioffe theorem, cf.\ \cite{ioffe} as well as \cite[Thm.\ 21]{valadier}.

On the other hand,
from \eqref{convex-analysis-2} we infer that
\[
\begin{aligned}
 \limsup_{\eps \to 0} \int_\Omega \calI_\eps (\theta_\eps(x)) \dd x
 & = \limsup_{\eps \to 0} \int_\Omega  J_\eps^* (\lteta_\eps (\theta_\eps(x))) \dd x
= \ber \limsup_{\eps \to 0} \edr \int_\Omega \left( \theta_\eps(x) \lteta_\eps (\theta_\eps (x)) -J_\eps(\theta_\eps (x))
\right) \dd x
\\ & \leq
 \int_\Omega \left( \theta(x) \lteta (\theta (x)) -J(\theta (x))
\right) \dd x
= \int_\Omega J^* \ber (\lteta (\theta (x))) \edr \dd x
\end{aligned}
\]
where the second  \ber equality \edr follows from \eqref{convex-analysis-2},  the third one from \eqref{convergence-bound} and
\eqref{weak-leps}, combined with the fact that the (integral functional associated with) $J_\eps$ Mosco-converges, as $\eps \down 0$, to (the integral functional associated with) $J$, and the last identity follows from elementary convex analysis.
This concludes the proof of \eqref{convergence-ltheta}.
\end{proof}
 \noindent With our next result we investigate the  Lipschitz
continuity of $\applogteta$ (of $\applogtetas$, respectively) and of
$\frac 1{\applogteta}$ (of $\frac 1{\applogtetas}$, respectively). The latter will play a crucial role in the
proof of  Lemma \ref{l:new-lemma3} \bec below. \eec
\begin{lemma}
\label{l:new-lemma2}  The functions $\applogteta:\R \to \R$ and
$\applogtetas:\R \to \R$
 satisfy
\begin{align}
& \label{bi-Lip} 
\eps <\applogteta'(x) \leq \eps+\frac{2}{\eps}
 \quad \text{for all } x \in \R\,,
\\
& \label{Lip} \exists\, C_L>0 \ \forall\, x,y \in \R\, : \qquad
\left|\frac{1}{\applogteta'(x)}-\frac{1}{\applogteta'(y)}\right|\leq
C_L|x-y|\,,
\\
& \label{bi-Lip-tetas} 
\eps <\applogtetas'(x) \leq \eps+\frac{2}{\eps}
 \quad \text{for all } x \in \R\,,
\\
& \label{Lip-tetas} \exists\, C_\ell>0 \ \forall\, x,y \in \R\, :
\qquad
\left|\frac{1}{\applogtetas'(x)}-\frac{1}{\applogtetas'(y)}\right|\leq
C_l|x-y|\,.
\end{align}
\end{lemma}
\begin{proof}
We detail only the proof of \eqref{bi-Lip} and \eqref{Lip}, \bec since  \eec the one for
\eqref{bi-Lip-tetas} and \eqref{Lip-tetas}
is completely analogous. The
left-hand side  inequality in \eqref{bi-Lip} directly follows from
the definition \eqref{def-applogteta} of $\applogteta$. 
%
 Plugging \eqref{e:max-mon-2} into the
definition \eqref{def-applogteta} of $\applogteta$ and using that
$\resteta$ is a contraction, we immediately deduce the right-hand
side inequality in \eqref{bi-Lip}.
Observe that, by  the first of \eqref{bi-Lip} the function $x\mapsto
\frac1{\applogteta'(x)}$ is well-defined on $\R$. In order to show
that it is itself \ber Lipschitz continuous \edr, we use the  formula
\begin{equation}
\label{very-useful-formula} \lteta_\eps'(x)=
\frac{\lteta'(\resteta(x))}{1+\eps \lteta'(\resteta(x))} \qquad
\text{for all } x \in \R\,.
\end{equation}
Therefore, for every $x,\,y \in \R$
\[
\begin{aligned}
\left|\frac1{\applogteta'(x)}-\frac1{\applogteta'(y)}\right|  & =
\left|\frac{1+\eps \lteta'(\resteta(x))}{\eps+\eps^2
\lteta'(\resteta(x))+\lteta'(\resteta(x))}- \frac{1+\eps
\lteta'(\resteta(y))}{\eps+\eps^2
\lteta'(\resteta(y))+\lteta'(\resteta(y))}\right|
\\
& = \left| \frac{L'(\resteta(x)) -
L'(\resteta(y))}{\left(\eps+\eps^2
\lteta'(\resteta(x))+\lteta'(\resteta(x))\right)\left( \eps+\eps^2
\lteta'(\resteta(y))+\lteta'(\resteta(y)) \right)} \right|
\\
 & \leq\left|
\frac{1}{L'(\resteta(x))}-\frac{1}{L'(\resteta(y))}\right| \leq
C_L|\resteta(x){-}\resteta(y)|
\end{aligned}
\]
where the last inequality follows from \ber \eqref{cond-L2} \edr ($C_L$
denoting the Lipschitz constant of $\frac{1}{\lteta'}$). Thus
\eqref{Lip} ensues, taking into account that $\resteta$ is a
contraction.
\end{proof}

\noindent Finally, we address the  properties
 of the functions
$\ftetas_\eps$ \eqref{def-appf} \ber and $H_\eps(x)$ \eqref{def-accaapp}.\edr

\begin{lemma}
\label{l:new-lemma3}
Under conditions \eqref{cond-ell} on $\ltetas$,
  the function $\ftetas_\eps:\R \to \R$
  is strictly increasing, with $f_\eps(0) =0$ and $f_\eps(x) >0$ if and only if $x>0$, bi-Lipschitz, and
  satisfies
  \begin{equation}
 \label{fprimoapp}
  \exists\, \bar{c}_1, \, \bar{c}_2>0 \ \forall\, x \in
\R\, : \qquad (\ftetas'_\eps(x))^2 \leq \bar{c}_1
|\ftetas_\eps(x)|+\bar{c}_2 \,.
\end{equation}
The function
$H_\eps:\R \to \R$
is strictly increasing on $(0,+\infty)$ and strictly decreasing on $(-\infty, 0)$ (hence $0$ is its absolute minimum),
and it
 satisfies
\begin{align}
& \label{accaapp} \exists\, \bar{c}_1^*, \, \bar{c}_2^*>0 \
\forall\, x \in \R\, : \qquad
 H_\eps (x)\geq \bar{c}_1^* |\ftetas_\eps(x)|- \bar{c}_2^*
|\applogtetas(x)-\applogtetas(0)|\quad  \ \text{for all } x \in \R \,.
\end{align}
\end{lemma}
\noindent Before developing the proof, we preliminarily observe that
 it is not restrictive to suppose,    in addition to
\eqref{cond-ell}, that
\begin{equation}
\label{ltetas-bonus}
 \exists\, x_0\in  \mathrm{D}(\ltetas)\ : \  \ltetas(x_0)=0
 \end{equation}
 Indeed, let $x_0$ be a fixed point in $\mathrm{D}(\ltetas)$ with
 $\ell(x_0)=f_0$, and set $\ell_{f_0}(x):= \ell(x) - f_0$. Then
 $\ell_{f_0}$ clearly fulfills \eqref{ltetas-bonus} and still
 complies with \eqref{cond-ell1}--\eqref{cond-ell2}. Since  the
 conjugate $j_{f_0}^*$ of a(ny) primitive of $\ell_{f_0}$
 is given by
 $j_{f_0}^* (w) = j^* (w+f_0)$, it is immediate to check that, if
 $\ell$ and $j^*$ fulfill \eqref{cond-ell4}, so do $\ell_{f_0}$ and
 $j_{f_0}^*$.
 We will use \eqref{ltetas-bonus}  to prove  estimate \eqref{limitato}
 below.
\begin{proof}
The properties of $f_\eps$ trivially follow from its definition \eqref{def-appf},
also in view of Lemma \ref{l:new-lemma2}.
Owing to \ber the Lipschitz continuity of $\frac{1}{\applogtetas'}$ (cf.\ \eqref{Lip-tetas}), \edr
 we can deduce that
\begin{align}
\label{derivlimit} &
-C_l\leq\frac{\dd}{\dd x}\left(\frac{1}{\applogtetas'(x)}\right)\leq C_l
\quad \ber \text{for a.a. } x \in \R\,,\edr
\end{align}
whence
\begin{align}
&  \frac{\dd}{\dd x}\left(\frac{1}{\applogtetas'(x)}\right)^2\leq 2
C_l\frac{1}{\applogtetas'(x)}  \quad \ber \text{for a.a. } x \in \R\,.\edr
\end{align}
Integrating between $0$ and $x\in\R^+$ and taking into account
definition \eqref{def-appf}, we find
\begin{align}
\label{moltiplicata-dopo} & \left({\ftetas_\eps'(x)}\right)^2\leq 2
C_l \ftetas_\eps(x)+ \frac{1}{\left(\applogtetas'(0)\right)^2}\quad
\text{for all } x >0\,.
\end{align}
In order to estimate $\displaystyle\frac{1}{\applogtetas'(0)}$ we
observe that, given $x_0$ as in \eqref{ltetas-bonus}, we have
$\restetas(x_0)=x_0$.
\ber Then, from \eqref{Lip-tetas} and \edr
\eqref{very-useful-formula} written for $\ltetas$,
 we deduce
\begin{align}
\label{limitato}
&\frac{1}{\applogtetas'(0)}\leq\left|\frac{1}{\applogtetas'(0)}-\frac{1}{\applogtetas'(x_0)}\right|+
\left|\frac{1}{\applogtetas'(x_0)}\right|\leq C_l |x_0|+
\frac{1+\ltetas'(x_0)}{\ltetas'(x_0)}\,.
\end{align}
Thus \eqref{fprimoapp} is proved
for all $x>0$. In order to complete the proof for $x<0$,
we multiply the left-hand side inequality in \eqref{derivlimit} by
$\displaystyle\frac{1}{\applogtetas'(x)}$ and we integrate from $0$
to $x<0$. In view of \eqref{def-appf} and \eqref{limitato}, we
deduce
\begin{align}
\label{moltiplicata-dopo-2}
 (\ftetas'_\eps(x))^2 \leq -2C_l
\ftetas_\eps(x)+\frac{1}{\left(\applogtetas'(0)\right)^2}\leq 2C_l
|\ftetas_\eps(x)| +C
 \quad \text{for all } x <0\,.
\end{align}

Finally, in order to prove \eqref{accaapp}, we again distinguish the
case $x>0$ and $x<0$. For $x>0$,
 we multiply \eqref{moltiplicata-dopo} by
$\applogtetas'(x)$ and we integrate on $(0,x)$. Recalling the definition
\eqref{def-appf} of $\ftetas_\eps$ as well as estimate \eqref{limitato}, we readily obtain
\eqref{accaapp}. For $x<0$, we
\bec develop calculations analogous to \eec
\eqref{moltiplicata-dopo-2}.
\end{proof}
\noindent We also use \bec the following \eec result
to pass to the limit
 as $\eps \down 0$ 
in the term $f_\eps (\tetase)$.
\begin{lemma}
\label{l:pass-limeps} Under conditions \eqref{cond-ell} on
$\ltetas$,  for every $x \in \overline{\mathrm{D(\ell)}}$ and every
 $(x_\eps) \subset \R$ with $x_\eps \to x$ there holds
$f_\eps(x_\eps) \to f(x)$ and $f_\eps'(x_\eps) \to f'(x)$
 as $\eps \to 0$.
\end{lemma}
\begin{proof}
We use
\[
|f_\eps (x_\eps) -f(x)| \leq |f_\eps (x_\eps) -f_\eps(x)| +|f_\eps (x) -f(x)|   \doteq \Delta_\eps^1 + \Delta_\eps^2.
\]
First,
\[
\Delta_\eps^1 \leq \left| \int_x^{x_\eps} \left( \frac{1}{\applogtetas'(s)} - \frac{1}{\applogtetas'(0)} \right) \dd s \right| +  \frac{1}{\applogtetas'(0)} |x_\eps -x |
\leq C_l  \max(|x|, |x_\eps|)|x_\eps -x| + C |x_\eps -x | \to 0
 \]
 where the second inequality follows from \eqref{Lip-tetas} and \eqref{limitato}. Second, in order to prove that $\Delta_\eps^2 \to 0$, we observe that
 \begin{equation}
 \label{ci-serve-dopo-ohyeah}
 \frac{1}{\applogtetas'(s)} \to \frac1{\ltetas'(s)} \quad \text{for all }
  s \in \overline{\mathrm{D}(\ltetas)},
  \end{equation}
  while, using \eqref{very-useful-formula} it is not difficult to check that
  \[
  0 <
   \frac{1}{\applogtetas'(s)} \leq \frac1{\ltetas'(r_\eps(s))} \bec +\eps \eec  \quad \text{for all } s \in \R\,.
   \]
Since for every $x \in \overline{\mathrm{D(\ell)}}$ we have that
$\frac1{\ltetas'(r_\eps)} \to \frac1{\ltetas'}$ in $L^1 (0,x)  $ by
dominated convergence,
  an extended version of the Lebesgue theorem  yields that also $\frac{1}{\applogtetas'} \to \frac1{\ltetas'} $
  in $L^1 (0,x)$. This concludes the proof that
  $f_\eps (x) \to f(x)$, \bec whence $f_\eps(x_\eps) \to f(x)$. \eec

  In order to check the last assertion, it is sufficient to observe
  that
  \[
\left|\frac{1}{\applogtetas'(x_\eps)} - \frac1{\ltetas'(x)}
\right|\leq \left|\frac{1}{\applogtetas'(x_\eps)} -
\frac{1}{\applogtetas'(x)} \right| + \left|
\frac{1}{\applogtetas'(x)}- \frac1{\ltetas'(x)} \right|
  \]
  and use that $\frac 1{\applogtetas'}$ is Lipschitz (cf.\
  \eqref{Lip-tetas}) to estimate the first term, and
  \eqref{ci-serve-dopo-ohyeah} for the second summand.
\end{proof}

The forthcoming Lemmas \bec address \eec the construction of a family of initial data
$(\vartheta_0^\eps, \vartheta_s^{0,\eps})_\eps$ fulfilling  properties
\bec \eqref{datitetaapp}--\eqref{bounddatiteta-4}. \eec
In Lemma \ref{l:new-lemma4} we exhibit an example of sequence $(\vartheta_0^\eps, \vartheta_s^{0,\eps})_\eps$ complying
 with the first set of properties,  \bec i.e.\ \eqref{datitetaapp}--\eqref{bounddatiteta-2}.  Since the construction developed in the proof
 of Lemma  \ref{l:new-lemma4}  does not guarantee the other requirements, we tackle them in two different results. Namely,
  in \eec
   Lemma \ref{l:new-lemma5}
we detail how the family $(\vartheta_s^{0,\eps})_\eps$ chosen in the Lemma \ref{l:new-lemma4}
 satisfies the additional \bec properties \eec \eqref{bounddatiteta-3}--\eqref{bounddatiteta-4}
 in the case of a special class of functions $\ltetas$ (and {\ele function} $\ftetas$) in \eqref{teta-s-weak}. Finally, in Lemma \ref{l:new-lemma6}, a family of data $(\vartheta_s^{0,\eps})_\eps$ \bec complying with the whole set of properties  \eqref{datitetaapp}--\eqref{bounddatiteta-4}
 is exhibited  in the case of the (physically relevant) choice \eec $\ltetas(\teta_s)=\ln(\teta_s)$ (and $\ftetas(\teta_s)=(\teta_s)^2$). 
\noindent
\begin{lemma}
\label{l:new-lemma4}  Assume that the initial data $\vartheta_0$ and $\vartheta_s^{0}$ respectively comply with \eqref{cond-teta-zero}--\eqref{cond-teta-esse-zero}.
Then, there exist sequences $(\vartheta_0^\eps)_\eps$ and $(\vartheta_s^{0,\eps})_\eps$
\bec fulfilling \eqref{datitetaapp}--\eqref{convdatiteta}  and \eec
 such that
\begin{align}
&\label{bounddatiteta-bislemma}
\ber \exists\, \bec \bar{S}_0>0  \quad \forall\, \eps>0 \, : \ \ \  \eec
 \|\lteta_\eps(\vartheta_0^\eps)\|_H\leq \bar{S}_0(1+\|\lteta(\teta_0)\|_H)\,,\qquad \edr \|\ltetas_\eps(\vartheta_s^{0,\eps})\|_{\Hc}\leq \|\ltetas(\teta_s^0)\|_{\Hc}\,,
 \end{align}
\berc and convergences
\eqref{bounddatiteta-1lemma} and \eqref{bounddatiteta-2lemma} hold. \eerc
\end{lemma}
\begin{proof}
We start by developing the construction of the sequence
$(\vartheta_0^\eps)_\eps$.
Denote by \ber $\gamma= (J^*)'$ \edr the inverse function $\lteta^{-1}$  and \bec let $\gamma_\eps$
 be its Yosida regularization and $\varrho_\eps$ its resolvent for any fixed $\eps>0$. \eec
We set
\begin{equation}
\label{def:tetazep} \vartheta_0^\eps(x):= R_\eps^{-1} \left(
\gamma_\eps(
w_0 (x)) \right) \qquad \forae\, x \in \Omega\,,
\end{equation}
where $w_0 (x):=L(\teta_0(x))$.
In view of~\eqref{e:max-mon-3-bis}, we have
$$
\lteta_\eps \left(\vartheta_0^\eps(x) \right) = \lteta \left( R_\eps
(\vartheta_0^\eps(x)) \right) =\lteta (\gamma_\eps(w_0 (x)))
\qquad \forae\, x \in \Omega\,.
$$
Again by~\eqref{e:max-mon-3-bis}, $\gamma_\eps (w_0 (x)) =
\gamma (\varrho_\eps (w_0 (x))$ and hence it holds
\begin{equation}
\label{e:crucial} \lteta_\eps \left(\vartheta_0^\eps(x) \right)= \varrho_\eps
(w_0 (x))
 \qquad
\forae\, x \in \Omega\,.
 \end{equation}
 Moreover, recalling \eqref{def:tetazep} and \eqref{e:max-mon-1}, we have
 \begin{equation}
\label{e:crucial1}
\vartheta_0^\eps(x)=\gamma_\eps(w_0 (x))+\eps \varrho_\eps (w_0 (x))
 \qquad \forae\, x \in \Omega\,.
 \end{equation}
By the Lipschitz continuity (with constant $1/\eps$) of $\gamma_\eps$
and using that  $\varrho_\eps$ is a contraction, in view of \eqref{cond-teta-zero} \bec (which ensures that $w_0 \in H$), \eec we readily deduce
\bec \eqref{datitetaapp} \eec and \eqref{bounddatiteta-bislemma}.

Next, in order to prove
\bec \eqref{convdatiteta} and \eec
\eqref{bounddatiteta-1lemma}, we recall \eqref{interesting-relation-teta-e-s}
\ber and \eqref{convex-analysis-3} and \eqref{e:crucial}.  \edr
We find that
\begin{align}
\label{stima1}
\mathcal{I}_\eps(\teta_0^\eps(x))= \frac{\eps}{2}|\teta_0^\eps(x)|^2+J^*_\eps(\lteta_\eps (\teta_0^\eps(x)))+J_\eps(0 & )=
\frac{\eps}{2}|\teta_0^\eps(x)|^2+\frac{\eps}{2}|\lteta_\eps(\teta_0^\eps(x))|^2+J^*(\lteta_\eps (\teta_0^\eps(x)))+J_\eps(0)\no\\
&=\frac{\eps}{2}|\teta_0^\eps(x)|^2+\frac{\eps}{2}|\lteta_\eps(\teta_0^\eps(x))|^2+J^*(\varrho_\eps (w_0(x)))+J_\eps(0)
 \end{align}
$ \forae\, x \in \Omega. $
On the other hand,
writing~\eqref{e:max-mon-5} for $J^*$, we find
\begin{equation}
\label{stima2}
J^*(\varrho_\eps (w_0(x)))=(J^*)_\eps(w_0(x))-\frac{\eps}{2}|\gamma_\eps(w_0(x))|^2
\qquad \forae\, x \in \Omega\,,
 \end{equation}
where $(J^*)_\eps$ denotes the Yosida regularization of $J^*$.
Now, we \bec combine  \eqref{e:crucial1}--\eqref{stima2} 
with  \eqref{not-restrictive-prima}, the Lipschitz continuity (with constant $1/\eps$) of $\gamma_\eps$, and
the fact \eec that  $\varrho_\eps$ is a contraction. Thus
we deduce that 
\begin{align}
\label{stima3}
\mathcal{I}_\eps(\teta_0^\eps(x)) =
\frac{\eps}{2}|\lteta_\eps(\teta_0^\eps(x))|^2+(J^*)_\eps(w_0(x))+\frac{\eps^3}{2}\varrho_\eps^2 (w_0(x))+
\eps^2 \varrho_\eps (w_0(x))\, \gamma_\eps (w_0(x))+J_\eps(0)\no\\
 \end{align}
 $\forae\, x \in \Omega$ and  for all $\eps\in (0,1).$
 \bec Then, to prove \eqref{convdatiteta} \eec we observe that
\begin{equation}
\label{e:conv-puntuale} \varrho_\eps (w_0(x)) \to w_0(x) \qquad
\text{as $\eps \downarrow 0$} \ \ \forae\, x \in \Omega
\end{equation}
\noindent
due to \bec our assumption that  $J^*(w_0) = J^*(L(\teta_0))\in L^1(\Omega)$ which implies that  $w_0(x) \in
\overline{\dom(J^*)}= \overline{\dom (\gamma)}$ for a.a.\ $x \in
\Omega$. \eec
Furthermore
\begin{equation}
\gamma_\eps(L(\teta_0(x)))\to \teta_0(x) \qquad
\text{as $\eps \downarrow 0$} \ \ \forae\, x \in \Omega
\end{equation}
and hence
\begin{equation}
\label{e:conv-puntualetetazero} \vartheta_0^\eps(x) \to \teta_0(x) \qquad
\text{as $\eps \downarrow 0$} \ \ \forae\, x \in \Omega.
\end{equation}
\bec Then, combining \eqref{bounddatiteta-1lemma},  \bec estimate \eec \eqref{ci-serve-anticipata-1},
\eqref{e:conv-puntualetetazero},  and \eec applying
 the dominated convergence theorem, we conclude \bec \eqref{convdatiteta}.
 Convergence  \eqref{bounddatiteta-1lemma} follows from
 \[
 \liminf_{\eps \downarrow 0} \int_\Omega \mathcal{I}_\eps(\teta_0^\eps)\dd x \geq \int_\Omega J^*(L(\teta_0)) \dd x
 \]
 (due to \eqref{stima1},  \eqref{e:conv-puntuale}, and the Fatou Lemma), combined with
 \[
 \limsup_{\eps \downarrow 0} \int_\Omega \mathcal{I}_\eps(\teta_0^\eps)\dd x \leq \int_\Omega J^*(L(\teta_0)) \dd x\,,
 \]
 which ensues from \eqref{stima3}, noting that
 $J_\eps (0) \to J(0)=0$ and that
  the first, the second, and the fourth terms on the r.h.s.\ of
 \eqref{stima3} tend to zero by \eqref{e:crucial},   \eqref{e:conv-puntuale}, the fact that $\gamma_\eps$ is Lipschitz with constant $1/\eps$, and the dominated convergence theorem.
 \eec

Concerning the initial data for $\teta_s$, in view of \eqref{cond-teta-esse-zero},
we can make the choice
\begin{equation}
\label{tetazeroesse}
\vartheta_s^{0,\eps}(x):=\vartheta_s^{0}(x) \ \ \forae\, x \in \gc\,.
\end{equation}
In this case, 
\bec \eqref{datitetaapp} and \eqref{convdatiteta} \eec
 are trivially verified. Moreover, \eqref{bounddatiteta-bislemma}
follows from the well-known properties of the Yosida regularization $\ltetas_\eps$.
Now, we only have to prove \eqref{bounddatiteta-2lemma}. To this aim, we recall
 \eqref{convex-analysis-2} and
\eqref{interesting-relation-teta-e-s}, \bec yielding  $ \forae\, x \in \gc $ and for all $\eps\in (0,1)$ \eec
\begin{align}
\label{stima4}
 i_\eps(\vartheta_s^{0}(x)) =\frac{\eps}{2}|\vartheta_s^{0}(x)|^2+j^*_\eps(\ltetas_\eps (\vartheta_s^{0}(x)))+j_\eps(0) & =
\frac{\eps}{2}|\vartheta_s^{0}(x)|^2+\vartheta_s^{0}(x) \ltetas_\eps(\vartheta_s^{0}(x))-
j_\eps (\vartheta_s^{0}(x))+j_\eps(0)\no\\
&\leq c(1+|\vartheta_s^{0}(x)|^2+|\ltetas(\vartheta_s^{0}(x))|^2)-j_\eps(\vartheta_s^{0}(x))+j(0)
\end{align}
for some positive constant $c$ (independent of $\eps$).
Moreover, taking into account \eqref{e:max-mon-5}, and using the convexity of $j$, and the Lipschitz continuity of $r_\eps$, we deduce
\begin{equation}
\label{stima5}
j_\eps(\vartheta_s^{0}(x))\geq j(r_\eps(\vartheta_s^{0}(x)))\geq -c(1+|\vartheta_s^{0}(x)|) \ \ \forae\, x \in \gc\,.
\end{equation}
Thus, from \eqref{stima4} and \eqref{stima5}, we can conclude
$$
i_\eps(\vartheta_s^{0}(x))\leq c(1+|\vartheta_s^{0}(x)|^2+|\ltetas(\vartheta_s^{0}(x))|^2)
\qquad \forae\, x \in \gc\,.
$$
\berc Since $i_\eps(\vartheta_s^{0}(x)) \to j^*(\ell(\vartheta_s^{0}(x)))$ for almost all $x \in \Gammac$, \eqref{bounddatiteta-2lemma} 
follows from the dominated convergence theorem. \eerc
\end{proof}

\noindent
\begin{lemma}
\label{l:new-lemma5}  Assume that the initial datum $\vartheta_s^{0}$ complies with \eqref{cond-teta-esse-zero}
and that the function $\ltetas$
\bec from \eqref{cond-ell1} is also bi-Lipschitz, i.e.\ \eec
\begin{equation}
\label{lbilip}
\exists\, C_1\,,C_2>0\,: C_1\leq\ltetas'(x)\leq C_2\ \ \forall\, x \in D(\ltetas)\,.
\end{equation}
Then, there holds
\begin{align}
&
\label{boundacca}
\exists\, \bec \bar{S}_1>0 \eec \ \ \forall\, \eps \in (0,1)\, : \quad
\int_\gc H_\eps(\teta_0^s)\dd x \leq \bar{S}_1\|\teta_0^s\|^2_H\,,
\\
&
\label{boundeffe}
\exists\, \bec \bar{S}_2>0 \eec \ \ \forall\, \eps \in (0,1)\,: \quad
\quad \|\ftetas_\eps(\teta_s^{0})\|_{V_\gc} \leq \bar{S}_2 \|\ftetas(\teta_s^0)\|_{\Vc}.
\end{align}
\end{lemma}
\bec Observe that, since $\ell$ is bi-Lipschitz, $f$ is also bi-Lipschitz. Therefore,  \eqref{cond-teta-esse-zero} automatically
guarantees that $\teta_s^0 \in V_{\gc}$. \eec
\begin{proof}
First, we note that $\ltetas_\eps$ \bec inherits \eec  the bi-Lipschitz continuity of $\ltetas$ (cf.\ \eqref{very-useful-formula}
written for $\ltetas$). Then, from \bec the definition \eqref{def-appf} of
$\ftetas_\eps$ it follows that \eec
 there exists a positive constant $C$ such that
\begin{equation}
\label{crescitaf}
|\ftetas_\eps(x)|\leq C|x| \ \ \forall\, x\in \R\,,  \ \ \forall\, \eps \in (0,1).
\end{equation}
Thus, by \eqref{def-accaapp} \bec (i.e.\ the definition of $H_\eps$), \eec \eqref{lbilip}, and \eqref{crescitaf}, we deduce \eqref{boundacca}.
Moreover,  observe that \eqref{lbilip} implies the bi-Lipschitz continuity 
\bec of $\ftetas_\eps$, \eec  whence \eqref{boundeffe} easily follows.
\end{proof}

\noindent
\begin{lemma}
\label{l:new-lemma6}  Assume that the initial datum $\vartheta_s^{0}$ complies with \eqref{cond-teta-esse-zero}
and that  \bec
\[
\ltetas(\tetas)=\ln(\tetas) \qquad \text{(whence $\ftetas(\tetas)=(\tetas)^2$)}\,.
\]
\eec  
Then, there exists a sequence $(\vartheta_s^{0,\eps})_\eps\subset H_{\Gammac}$ such that
\begin{align}
&\label{convdatitetalog}
\vartheta_s^{0,\eps} \to \teta_s^0 \text{ in } H_{\Gammac}
\text{ as $\eps\down0$\,,}
\\
&\label{bounddatiteta-bislog}
 \|\ln_\eps(\vartheta_s^{0,\eps})\|_{\Hc}\leq \|\ln(\teta_s^0)\|_{\Hc} \quad \text{ for all } \, \eps>0\,,
\end{align}
\berc convergence \eqref{bounddatiteta-2lemma} holds, \eerc and moreover
\begin{align}
&
\label{boundaccalog}
\exists\, \bec \bar{S}_4>0 \ \ \forall\, \eps \in (0,1)\, : \eec \quad
\int_\gc H_\eps(\vartheta_s^{0,\eps})\dd x \leq \bar{S}_4(1+\|\teta_0^s\|^2_H)\,,
\\
&
\label{boundeffelog}
\exists\, \bec \bar{S}_5>0 \  \ \forall\, \eps \in (0,1)\,: \eec  \quad
\quad \|\ftetas_\eps(\vartheta_s^{0,\eps})\|_{V_\gc} \leq \bar{S}_5(1+\|\ftetas(\teta_s^0)\|^2_{\Vc}).
\end{align}
\end{lemma}
\begin{proof}
We will construct a sequence of data $(\teta_s^{0,\eps})_\eps$ satisfying \eqref{convdatitetalog}--\eqref{boundeffelog}
partially adapting the argument of \cite[Example 4.3]{bbr6}.
For \bec  $\eps \in (0,1)$
we set \eec
\begin{equation}
\label{costr-appr-data} \teta_s^{0,\eps}:=\max\{\vartheta_0^s, \eps^{\alpha}\}\qquad
\quad \text{for some } \alpha>0\
\end{equation}
\bec to be chosen later.
It is not difficult to check that $\teta_s^{0,\eps}$ can be written as $\nu_\eps (f(\teta_s^0))$ for a suitable Lipschitz function $\nu_\eps$. Therefore,
$\teta_s^{0,\eps}$ is also in $H^1(\gc)$. \eec
Observe that $\teta_s^{0,\eps}>0$ a.e.\ in $\gc$ (since $\teta_0^s >0$
a.e.\ in $\gc$ thanks to the second of~\eqref{cond-teta-esse-zero}
\bec and  the fact that \eec
$\ltetas(\vartheta_0^s)=\ln (\vartheta_0^s)$),   that $\teta_s^{0,\eps}\to\vartheta_0^s$ a.e.\ in $\gc$, and moreover that $(\teta_s^{0,\eps})_\eps$ is uniformly integrable in $\Hc$.
Thus, \eqref{convdatitetalog} follows.
Furthermore,
relying on the definition \eqref{costr-appr-data} and on  well-known properties of the Yosida regularization $\ln_\eps$, we find
\begin{equation}
\label{appr-data_2} |\ln_\eps(\teta_s^{0,\eps})|\leq |\ln_\eps(\vartheta_0^s)|\leq |\ln(\vartheta_0^s)|
 \quad \text{a.e.\ in } \gc,
\end{equation}
whence \eqref{bounddatiteta-bislog}.
Next, arguing as in
\cite[Lemma 4.1]{bbr6}, we can deduce that
$$
i_\eps(x)\leq\frac{\eps}{2} x^2+2x \qquad \text{for every } x \in [0,+\infty)  \text{ and } \eps \in (0,1),
$$
\berc  which implies that 
 \[
i_\eps(\teta_s^{0,\eps}(x)) \leq 
C(1 + (\teta_s^0(x))^2) \qquad \foraa\, x \in \Gammac.
\]
Then, \eqref{bounddatiteta-2lemma} follows from the pointwise convergence of $i_\eps(\teta_s^{0,\eps})$ to $j^* (\ell (\teta_s^0))$, via the dominated convergence theorem. \eerc

Now we recall definitions \eqref{def-appf}, \eqref{def-accaapp}, and \eqref{very-useful-formula} written for $\ln$ \bec (yielding in particular \eec
that $\ln'_\eps$ is a decreasing function), and find
\begin{align}
\label{accastima}
&H_\eps(x)=\int_0^x\applogtetas'(s)\, \int_0^s \frac{1}{\applogtetas'(\tau)} \dd \tau \dd s =
\int_0^x\int_0^s \frac{\applogtetas'(s)}{\applogtetas'(\tau)}
\dd \tau  \dd s \leq \int_0^x s \dd s =\frac{x^2}{2}
\quad \text{for all }
x\geq 0\,,
\end{align}
whence \eqref{boundaccalog}.

\noindent
In order to show \eqref{boundeffelog}, we first recall that there exists a positive constant $C$ such that
\begin{equation}
\label{fstima1}
\|\ftetas_\eps(\vartheta_s^{0,\eps})\|_{L^1(\gc)} \leq C(1+\|\teta_s^0\|^2_{\Hc})
\end{equation}
again by definition \eqref{def-appf} and \eqref{very-useful-formula}
written for $\ln$.
Now, we have to prove that
\begin{equation}
\label{fstima2}
\|\nabla\ftetas_\eps(\vartheta_s^{0,\eps})\|_{\Hc} \leq C(1+\|\nabla\ftetas(\teta_s^0)\|_{\Hc})
\end{equation}
To this aim, we recall definition \eqref{def-appf} and relation \eqref{very-useful-formula} written for $\ln_\eps$,
\bec whence $\ln_\eps'(x) = \frac 1{\eps +r_\eps (x)}$. \eec
We have
\begin{equation}
\label{fstima3}
\int_\gc|\nabla\ftetas_\eps(\vartheta_s^{0,\eps})|^2 \dd x \leq
 \int_\gc (\eps+r_\eps (\vartheta_s^{0,\eps}))^2 |\nabla(\vartheta_s^{0,\eps})|^2 \dd x\,.
\end{equation}
We first observe that
\begin{align}
\label{fstima4}
&\eps^2\int_\gc \bec |\nabla \vartheta_s^{0,\eps}|^2 \eec  \dd x =
\eps^2\int_{\displaystyle\{ \vartheta_s^{0}>\eps^\alpha \}}|\nabla\vartheta_s^{0}|^2 \dd x
\leq
\eps^{2-2\alpha}\int_\gc |\vartheta_s^{0}|^2\, |\nabla\vartheta_s^{0}|^2 \dd x =
\frac{\eps^{2-2\alpha}}{4}\|\nabla\ftetas(\vartheta_s^{0})\|^2_{\Hc}\,.
\end{align}
Next, we evaluate the term
\begin{equation}
\label{fstima5}
\begin{aligned}
&
 \int_\gc (r_\eps (\vartheta_s^{0,\eps}))^2 |\nabla(\vartheta_s^{0,\eps})|^2 \dd x
 \\
 &
 =
 \int_{\displaystyle\{ \vartheta_s^{0,\eps}\geq 1 \}}(r_\eps (\vartheta_s^{0,\eps}))^2 |\nabla(\vartheta_s^{0,\eps})|^2 \dd x+
 \int_{\displaystyle\{ \vartheta_s^{0,\eps}< 1 \}}(r_\eps (\vartheta_s^{0,\eps}))^2 |\nabla(\vartheta_s^{0,\eps})|^2 \dd x\,.
 \end{aligned}
 \end{equation}
Noting that (cf.\ definition \eqref{e:max-mon-3-bis} written for $\ln$)
\begin{equation}
\label{fstima6}
\bec r_\eps (\vartheta_s^{0,\eps})\leq \vartheta_s^{0,\eps} \quad \text{on the set } \{ x\in \gc\, : \  \vartheta_s^{0,\eps}(x)\geq 1 \}\,,\eec
 \end{equation}
 we find
\begin{align}
\label{fstima7}
 \int_{\displaystyle\{ \vartheta_s^{0,\eps}\geq 1 \}}(r_\eps (\vartheta_s^{0,\eps}))^2 |\nabla(\vartheta_s^{0,\eps})|^2\dd x & \leq
 \int_{\displaystyle\{ \vartheta_s^{0,\eps}\geq 1 \}}(\vartheta_s^{0,\eps})^2 |\nabla(\vartheta_s^{0,\eps})|^2\dd x
 \no\\
&
\leq
  \int_\gc (\vartheta_s^{0})^2 |\nabla(\vartheta_s^{0})|^2\dd x =
  \frac14 \|\nabla\ftetas(\vartheta_s^{0})\|^2_{\Hc}\,.
 \end{align}
In order to estimate the second term on the right-hand side of \eqref{fstima5}, we use the relation (cf.\ \eqref{e:max-mon-3-bis})
$
r_\eps(x)=x-\eps\ln_\eps(x) $   for all $x\in \R,
$
and we obtain
$$
\int_{\displaystyle\{ \vartheta_s^{0,\eps}< 1 \}}(r_\eps (\vartheta_s^{0,\eps}))^2 |\nabla(\vartheta_s^{0,\eps})|^2\dd x
 \leq c \int_{\displaystyle\{ \vartheta_s^{0,\eps}< 1 \}}((\vartheta_s^{0,\eps})^2+\bec \eps^2 \eec \ln_\eps^2 (\vartheta_s^{0,\eps})) |\nabla(\vartheta_s^{0,\eps})|^2\dd x\,.
 $$
 Again
\begin{equation}
\label{fstima8}
\int_{\displaystyle\{ \vartheta_s^{0,\eps}< 1 \}} (\vartheta_s^{0,\eps})^2 |\nabla(\vartheta_s^{0,\eps})|^2 \dd x
 \leq c \|\nabla \ftetas((\vartheta_s^{0})\|^2_{\Hc}
\end{equation}
Moreover, the  inequalities $\eps^\alpha\leq\vartheta_s^{0,\eps}<1$ and well-known properties of the Yosida regularization $\ln_\eps$ imply that
$|\ln_\eps(\teta_s^{0,\eps})|\leq |\ln_\eps(\eps^\alpha)|\leq |\ln(\eps^\alpha)|$,
whence
\begin{align}
\label{fstima9}
 \int_{\displaystyle\{ \vartheta_s^{0,\eps}<1 \}} \eps^2 \ln_\eps^2 (\vartheta_s^{0,\eps}) |\nabla(\vartheta_s^{0,\eps})|^2\dd x
   & \bec =  \eps^2 \int_{\displaystyle\{ \vartheta_s^{0,\eps}<1 \}}  \frac1{(\vartheta_s^{0,\eps})^2 }    (\vartheta_s^{0,\eps})^2  \ln_\eps^2 (\vartheta_s^{0,\eps}) |\nabla(\vartheta_s^{0,\eps})|^2\dd x \eec \no
 \\
 &
  \leq
 \eps^{2-2\alpha}\ln^2(\eps^\alpha)
 \int_{\displaystyle\{ \vartheta_s^{0,\eps}\bec <  \eec 1 \}}
 (\vartheta_s^{0,\eps})^2 |\nabla(\vartheta_s^{0,\eps})|^2\dd x
 \no\\
&\leq
c \eps^{2-2\alpha}\ln^2(\eps^\alpha)
\|\nabla\ftetas(\vartheta_s^{0})\|^2_{\Hc}\,.
 \end{align}
 Observe that the right-hand side of \eqref{fstima9} is bounded independently of $\eps$, choosing $\alpha<1$ and $\eps$ sufficiently small.
Finally, collecting \eqref{fstima4}, \eqref{fstima7}, \eqref{fstima8}, and \eqref{fstima9}, we deduce \eqref{fstima2}.
Thus \eqref{boundeffelog} is also proved.
\end{proof}



\end{document}